\documentclass[reqno, 11pt]{amsart}

\usepackage{amscd,amsfonts,amssymb} 
\usepackage[nocompress]{cite} 
\usepackage{graphicx}
\usepackage{amsmath,mathtools} 
\usepackage{bm}
\usepackage[margin=1.05in]{geometry}
\usepackage[open]{bookmark} 
\usepackage[nocompress]{cite} 
\usepackage{stmaryrd} 
\usepackage{bbm} 
\usepackage{url} 

\usepackage{hyperref}
\hypersetup{hidelinks}

 \usepackage[us,12hr]{datetime}    

\numberwithin{equation}{section}
\numberwithin{figure}{section}
\numberwithin{table}{section}

\newtheorem{theorem}{Theorem}[section]

\newtheorem{proposition}{Proposition}[section]

\theoremstyle{definition}
\newtheorem{definition}{Definition}[section]
\theoremstyle{remark}
\newtheorem{remark}{Remark}[section]

\usepackage[colorinlistoftodos,prependcaption,textsize=tiny]{todonotes}

\renewcommand{\div}{ {\rm div} }

\providecommand{\abs}[1]{\left| #1 \right|}   
\providecommand{\norm}[1]{\left\| #1 \right\|}  
\providecommand{\jump}[1]{ \left\llbracket #1 \right\rrbracket}

\providecommand{\prts}[1]{ \left( #1 \right) }
\providecommand{\depvar}[1]{ \left[   #1   \right] }
\providecommand{\seminorm}[1]{ \left[   #1   \right] }

\author{Tian Jing \and Dehua Wang}
\address{Department of Mathematics, University of Michigan, Ann Arbor, MI 48109, USA.}
\email{tnjing@umich.edu}
\address{Department of Mathematics, University of Pittsburgh, Pittsburgh, PA 15260, USA.}
\email{dhwang@pitt.edu}

\title[two-phase MHD]
{Strong solutions to the  three-dimensional two-phase magnetohydrodynamic equations}



\begin{document}

\begin{abstract}

In this paper, we study the existence of strong solutions to the two-phase magnetohydrodynamic equations in a bounded domain $\Omega\subseteq \mathbb{R}^3$. The fluids are incompressible, viscous, and resistive. The surface tension is considered. The equations are reformulated using the Hanzawa transformation, which turns the free interface into a fixed one for a short time. The study of the new equations is then divided into the principal part and the nonlinear part. Due to the effect of the magnetic field and the complexity of the transformation in generic bounded domains, the Fr{\'e}chet derivatives of nonlinearities have to be carefully estimated. The equations can then be solved using the fixed-point argument by finding a contraction mapping, which follows the estimates of the nonlinear part.

\end{abstract}

\maketitle

\section{Introduction and Main Results}

In this paper, we study the existence of strong solutions 
to the two-phase magnetohydrodynamic (MHD) equations. 
The two immiscible fluids are incompressible, viscous and resistive.
They occupy an open, bounded, simply connected $C^3$ domain $\Omega\subseteq \mathbb{R}^3$. 
The domains of
inner and outer fluids are represented by the open sets $\Omega^+(t)$
and $\Omega^-(t)$, respectively. The fluid-fluid interface is denoted by the set $\Gamma(t)$, which is a closed surface in $\Omega$; see Figure \ref{F1}.
These three sets are disjoint and we have $\Omega^+(t) \cup \Gamma(t) \cup \Omega^-(t)  =\Omega$.
In this work, we assume that $\Gamma(t)$ and $\Omega^+(t)$ 
do not touch the boundary $\partial\Omega$.
We consider the following equations of the two-phase MHD flow:
\begin{align}
	\partial_t u + (u\cdot \nabla) u -(\nabla\times B)\times B 
	+ \nabla p -  \nu^\pm \triangle u  =0   \quad & \text{in} \;  \Omega^{\pm}(t), 
	\label{main.1.u}
	\\ 
	\partial_{t}B-\nabla\times(u\times B)+\nabla\times(\sigma\nabla\times B)=0    \quad  & \text{in} \;   \Omega,
	\label{main.1.B}
	\\
	{\rm div} u=0  \quad & \text{in} \;   \Omega^{\pm}(t),  \label{main.divu.0}
	\\ 
	{\rm div} B=0 \quad & \text{in} \;   \Omega, 
	\label{main.divB.0}
	\\
	-  \jump{ \nu^\pm (\nabla u + \nabla u^\top) -pI } \! n = \kappa Hn    \quad & \text{on} \;   \Gamma(t), \label{eq:s,interface} \\ 
	V_{\Gamma}  = u\cdot n   \quad & \text{on} \;   \Gamma(t), 
	\label{main.1.chi}
	\\
	u|_{\partial\Omega}=0,  \ B|_{\partial\Omega}=0, & 
	\label{main.boundary} 
	\\
	u|_{t=0}=u_{0}, \ B|_{t=0}=B_{0} \label{main.1.initial}, & 
\end{align}
where $u$, $ B$ and $p$ stand for the velocity, magnetic field and pressure, respectively.
The density of both fluids is assumed to be 1 everywhere and the magnetic diffusion coefficient $\sigma$ remains a constant.
The viscosity coefficient  
takes different constant values $\nu^\pm$  in $\Omega^\pm (t)$. 
The surface tension coefficient $ \kappa$ is a positive constant.
 On the interface $\Gamma(t)$,  the mean curvature $H$ is a function, 
the outward (pointing to $\partial\Omega$) unit normal vector and the 
 speed  of the interface $\Gamma(t)$ are denoted by $n$ and $V_\Gamma$, respectively.  
The notation $\jump{f}$ stands for the jump of $f$ across the interface $\Gamma$, 
 i.e., for all $t\geq 0$ and all $x\in \Gamma(t)$,
 \[
 	\jump{f}(x) := \lim_{\varepsilon\to 0^+} f(x+\varepsilon n (x)) - \lim_{\varepsilon\to 0^+} f(x-\varepsilon n(x));
 \]
 when $f$ does not have enough regularity, its one-side limits at $\Gamma(t)$ are 
 considered in the sense of trace.
We refer to  \cite{Jing.varifold.MHD} for more details on the model above; and \cite{Abels.varifold, Fischer.Hensel.uniqueness, Pruss.quali} for the two-phase Navier-Stokes equations. For interested readers, we also refer to \cite{Robinson.3D-NS, Novotny.compressible.book} and \cite{Davidson.MHD} for the classical theory of Navier-Stokes equations and MHD equations.

\begin{figure}[h]
		\begin{tikzpicture}  
		\draw plot [smooth cycle, tension=1] coordinates {(1.3,0)  (0,0.8) (-0.9,0) (0,-1.3)  };

            \draw plot [smooth cycle, tension=1] coordinates {(2.2,0)  (0,1.9) (-2.4,0) (0,-2.1)  };

   \draw[->, thick] (1.3, 0.2) -- (1.6, 0.3) ;
   
   \node at (1.7,0) {$n$};
   
			\node at (0.8,0.8) {$\Gamma(t)$};
                \node at (0.2,-0.2) {$\Omega^+(t)$};
			\node at (-0.8,1.1) {$\Omega^-(t)$};
            \node at (-2.1,-0.3) {$\partial\Omega$};
		\end{tikzpicture}
		\caption{Two-phase flow in a bounded domain}
		\label{F1}
\end{figure}
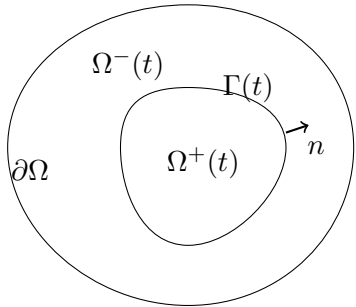

\subsection{Related research}

When the magnetic field vanishes, the problem \eqref{main.1.u}-\eqref{main.1.initial} turns into the two-phase Navier-Stokes equations with surface tension.
In \cite{Pruss.two-phase}, Prüss and Simonett studied the problem in $\mathbb{R}^{n+1}$, 
where the interface can be expressed as the graph of a  function defined on $\mathbb{R}^n$. 
For any time interval, 
if the initial velocity and initial interface satisfy some smallness conditions (dependent on the time interval),
then the unique strong solution exists. 
In \cite{Pruss.analytic.solu} also by Prüss and Simonett, a different type 
of existence theory was obtained.
The smallness condition 
in \cite{Pruss.analytic.solu}
is  only required for the initial interface,
which implies the local existence of strong solutions.
The same equations have been studied in a bounded domain
by K\"ohne, Pr\"uss and Wilke in \cite{Pruss.quali} for strong solutions.
Moreover, Abels and Wilke have studied the two-phase
Navier–Stokes-Mullins–Sekerka system in \cite{Abels.Wilke.NSMS.strong}.

Besides the strong solutions studied in the above-mentioned works, there has also been some research on global weak solutions to two-phase flows with surface tension.
In \cite{Plotnikov}, Plotnikov has studied the two-phase Navier-Stokes equations for incompressible non-Newtonian fluids in $\mathbb{R}^2$.
The case of incompressible non-Newtonian fluids in $\mathbb{R}^3$ has been studied by Abels in \cite{Abels.varifold}.
In  \cite{Yere},  Yeressian has studied the case of Newtonian fluids in $\mathbb{R}^3$.
In \cite{Novotny.dissipative},
Feireisl and Novotn\'y have studied the
dissipative varifold solution.
The weak-strong uniqueness of strong solutions and varifold solutions to the two-phase incompressible Navier-Stokes equations has been studied by Fischer and Hensel in \cite{Fischer.Hensel.uniqueness}.

To the best of our knowledge, research on free-boundary MHD equations has mainly focused on the fluid-vacuum model, where the fluid is surrounded by the vacuum.
We refer interested readers to \cite{Solonnikov.free-boundary.mhd, Lee.uniform.viscous.mhd, Lee.IVP.mhd, Guo.Zeng.Ni.decay-rates} and their references
for the case of viscous flows; 
and
 \cite{Gu.Luo.mhd.surTen, Luo.Zhang.apriori.MHD} and their references
for the case with no viscosity and no magnetic diffusivity.
The existence of varifold solutions to the two-phase MHD problem \eqref{main.1.u}-\eqref{main.1.initial} was obtained in \cite{Jing.varifold.MHD}. 
The goal of this paper is to study the strong solutions to the  problem \eqref{main.1.u}-\eqref{main.1.initial}.

\subsection{Main results}

Based on the settings in  \cite{Pruss.quali}, 
we give the definition of strong solutions
to the two-phase MHD equations.

\begin{definition}[Strong solution] \label{definition.strong-solution}
Let $q\geq 1 $ be a fixed number.
Let $u_0\in W^{2-\frac{2}{q}, \, q}(\Omega\setminus \Gamma_0) \cap C(\overline{\Omega}) $ and
$B_0\in W^{2-\frac{2}{q}, \, q}(\Omega)$,
where
$\Gamma_0$ is a closed  $W^{3-\frac{2}{q} ,  \, q}$ surface in $\Omega$.
We call $(u,B,p,\Gamma)$ a strong solution to the two-phase MHD equations \eqref{main.1.u}-\eqref{main.1.initial} on $[0,T]$ 
if:
\begin{enumerate}
\item $ u\in  W^{1,q}([0,T]; L^q(\Omega)) \cap  L^q([0,T]; W^{2,q}(\Omega\setminus\Gamma(t))) 
	\cap C([0,T]\times \Omega) $;
\item $ B \in  W^{1,q}([0,T]; L^q(\Omega)) \cap  L^q([0,T]; W^{2,q}(\Omega))$;
\item $ p \in  L^q([0,T]; \dot{W}^{1,q}(\Omega\setminus\Gamma(t))) $;
\item $\Gamma(t)$ is the graph of a height function $h$ on some closed $C^3$ reference surface $\Sigma$ (see Figure \ref{F2}),
such that
$\Gamma(0)=\Gamma_0$ and
$ h \in  W^{2-\frac{1}{2q} ,q } ([0,T]; L^q(\Sigma)) 
\cap  W^{1,q} ([0,T]; W^{2-\frac{1}{q}, q}(\Sigma))
\cap  L^q ([0,T]; W^{ 3 - \frac{1}{q}, q}(\Sigma)) $
i.e., $\Gamma(t)= \{ x+h(t,x)n_\Sigma (x): x\in\Sigma \}$, where $n_\Sigma$ is the outward 
unit normal vector of $\Sigma$;
\item
 There exists a function $\tilde{p} \in W^{\frac{1}{2} - \frac{1}{2q} ,q } ([0,T]; L^q(\Sigma))
\cap L^q ([0,T]; W^{ 1 - \frac{1}{q} ,q }(\Sigma)) $ 
such that  $ \jump{p}(t, x+ h(t,x)n_\Sigma(x)) = \tilde{p}(t,x) $ almost everywhere on $[0,T]\times\Sigma$;
\item For almost every $t\in [0,T]$, the equations \eqref{main.1.u}-\eqref{main.1.initial} 
are satisfied almost everywhere on $\Omega$ or $\Gamma(t)$.
\end{enumerate}

\end{definition}
\begin{remark}
    In Definition \ref{definition.strong-solution}, the notation $f\in L^q([0,T]; W^{s,q}(\Omega\setminus\Gamma(t))) $ means that there exists a diffeomorphism $\Phi(t,\cdot)$ on $\Omega$ dependent on time, such that $f\circ \Phi \in L^q([0,T]; W^{s,q}(\Omega\setminus\Sigma) )$, where $f\circ\Phi (t,x):= f(t,\Phi(t,x))$. 
\end{remark}

\begin{figure}[h]
	\begin{tikzpicture}
		\draw[blue] plot [smooth cycle, tension=1] coordinates {(1.2,0)  (0,1.2)  (-1.4,0)   (0,-1.3)  };
		
		\node[blue] at (1,-1) {$\Sigma$};
		
		\draw plot [smooth cycle, tension=1] coordinates {(2.2,0)  (0,1.9) (-2.4,0) (0,-2.1)  };
		
		\draw[red] plot [smooth cycle, tension=1] coordinates {(1.5,0) (0.7,0.8) (0,1.6) (-0.8, 0.7)(-1.6,0) (-0.9,-0.7) (0,-1.5) (0.7, -0.7) }; 
		\node[red] at (-0.5,0.5) {$\Gamma(t)$};
		
		\draw[thick] (0,1.6) -- (0,1.2);
		\draw[->,dashed] (1.8,1.45) -- (0.1,1.45);
		\node at (2.5,1.4) {$h(t,x)$};
	\end{tikzpicture}
	\caption{Reference surface and height function}
	\label{F2}
\end{figure}
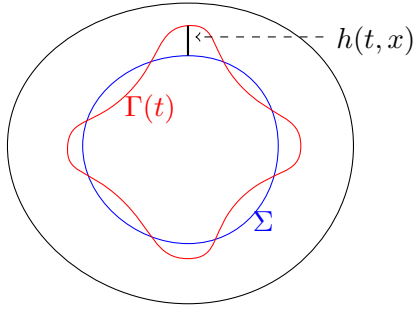

When the initial interface $\Gamma_0$ is $C^3$, we have the following  existence result.
\begin{theorem} \label{existence.result.1}
Let $q>5$ be a fixed number and $\Omega$ be a bounded $C^3$ domain. 
Assume that the initial data
\[
u_0\in W^{2-\frac{2}{q}, \, q}(\Omega\setminus \Gamma_0) \cap C(\overline{\Omega}),
\quad B_0\in W^{2-\frac{2}{q}, \, q}(\Omega),
\quad and \quad \Gamma_0\in C^3,
\]
 satisfy the following compatibility conditions:
\begin{enumerate}
\item  $\div u_0 =0$  in $\Omega\setminus \Gamma_0$,  $\div B_0=0$ in $\Omega$;
\item  $u_0 =B_0 =0$ on $\partial\Omega$;
\item  $\Gamma_0$ is a  closed interface  and $\Gamma_0 \cap \partial\Omega=\emptyset$;
\item  $(I-n_{\Gamma_0} \otimes n_{\Gamma_0} )\jump{\nu^\pm \prts{\nabla u_0+ \nabla u_0^\top} }n_{\Gamma_0} =0 $.
\end{enumerate}
Then there exists $T>0$ such that the original problem \eqref{main.1.u}-\eqref{main.1.initial}
has a unique strong solution on $[0,T]$
with the reference surface $\Sigma = \Gamma_0$.
\end{theorem}

\begin{remark}
    The  compatibility condition (4) in Theorem \ref{existence.result.1} guarantees the solvability of the linearized problem with initial data with the given conditions and zero source terms. This solution is an auxiliary solution, and the strong solution is a perturbation of this linearized solution.
\end{remark}

When the initial interface has less regularity $W^{3-\frac{2}{q} ,  \, q}$, we have the following result.

\begin{theorem} \label{existence.result.2}
Let $q>5$ be a fixed number and $\Omega$ be a bounded $C^3$ domain. 
Let $\Sigma $ be an arbitrary closed $C^3$ surface in $\Omega$ with
$\Sigma \cap \partial\Omega=\emptyset$.
For all $M_0 >0$, there exists $\varepsilon_0 = \varepsilon_0 (\Sigma, M_0)>0$, such that 
for all admissible initial data $(u_0, B_0,\Gamma_0)$ in Definition \ref{existence.result.1}
with  
\begin{equation}\label{IC1}
\norm{u_0}_{W^{2-\frac{2}{q} ,  \, q } (\Omega\setminus\Gamma_0)}\leq M_0,\quad
\norm{u_0}_{L^\infty}\leq M_0,\quad
\norm{B_0}_{W^{2-\frac{2}{q} ,  \, q } (\Omega)}\leq M_0,\quad
\norm{h_0}_{W^{3-\frac{2}{q} ,  \, q } (\Sigma)} \leq \varepsilon_0,
\end{equation}
there exists $T = T(\Sigma,M_0,\varepsilon_0)>0$ and 
the problem \eqref{main.1.u}-\eqref{main.1.initial}
has a unique strong solution on $[0,T]$ with reference surface $\Sigma$.
Here $h_0$ denotes the height function of $\Gamma_0$ on $\Sigma$.  
\end{theorem}

\begin{remark}
    The first condition in \eqref{IC1} implies  
    $\norm{u_0}_{L^\infty}\leq C(\Gamma_0) M_0 $. In our proof, a uniform upper bound is needed. Thus, an additional condition $\norm{u_0}_{L^\infty}<M_0$ is included. 
\end{remark}

Theorem \ref{existence.result.1} and Theorem \ref{existence.result.2} can be proved in the spirit of \cite{Pruss.quali}.
The main difficulties of this work stem from two aspects: the coupling of magnetic equations; and the complexity of nonlinear terms after the Hanzawa transformation.
It becomes unclear if the maximal regularity theory and the techniques used in estimates of nonlinear terms in \cite{Pruss.quali} can still be applied to obtain a contraction mapping for the fixed-point argument, since the magentic terms have changed the structure of equations.

The Hanzawa transformation  requires a composition of each variable with the pulling diffeomorphism $\Theta_h$,
which is generated using the height function $h$. 
We usually abbreviate $\Theta_h$ to $\Theta$ and  denote the displacement  by $\theta:=\Theta- I$, where $I$ is the identity map.
The mapping $\Theta$ pulls the reference surface $\Sigma$  to the free interface $\Gamma(t)$, i.e., $f_{\text{new}}=f_{\text{old}} \circ \Theta$, where $f$ stands for variables $u$, $B$ or $p$. 
As a  result, the expressions of derivatives of  
$f_{\text{new}}$ 
contain one or multiple of the following terms related to $\Theta$:
$$\mathcal{M}_0 := \prts{ I-hL_{\Sigma}}^{-1},\quad
\mathcal{M}_1 :=
 (I+\nabla\theta)^{-1} \nabla\theta ,\quad
\mathcal{M}_2 :=
\prts{ \triangle\Theta^{-1} }\circ \Theta, $$
$$ \mathcal{M}_3 := \partial_t\theta \prts{ I - (I+\nabla\theta)^{-1} \nabla\theta  }, \quad \text{and}\quad
\mathcal{M}_4 :=
  \prts{\nabla \Theta}^{-\top}  \prts{\nabla \Theta}^{-1} - I;$$
see Section
\ref{Transformed equations} and \cite{Pruss.quali} for more details.
Similar to \cite{Pruss.quali}, the equations will be written into the form $Lz=Gz $ where $z=(u,B,p,h)$,  $L$ and $G$ stand for the linear principal terms and the nonlinear terms, respectively.
The main challenges then become the following issues: 
\begin{enumerate}
\item  whether the linear operator $L$ is still invertible;
\item  whether we can construct a contraction mapping using $G$.
\end{enumerate}
The first issue can be resolved using the maximal regularity theory of two-phase Stokes equations in \cite{Pruss.quali} and the theory of parabolic equations in \cite{Ladyzenskaja}.

The second difficulty  requires a careful analysis of the Fr\'{e}chet derivatives of the operator $G$.
In \cite{Pruss.analytic.solu,Pruss.two-phase}, the detailed estimates of nonlinear terms have been studied when the reference surface $\Sigma$ is a hyperplane $\Sigma=\mathbb{R}^n$ in $\mathbb{R}^{n+1}$. In this case, the Weingarten tensor $L_\Sigma=0$, which implies
$\mathcal{M}_0 = I$ and $\Theta= x+ \eta(x_{n+1}/\rho_0) h(x') e_{n+1}$.
The case of curved $\Sigma$ has been discussed in \cite{Pruss.quali}.
In our work, the reference surface $\Sigma$ is a generic closed surface and the magnetic field is coupled into the system. 
Thus, we need a rigorous calculation of all Fr\'{e}chet derivatives in the general case. 
After the Hanzawa transformation, the nonlinear terms become intricate combinations of various operators.
Thus,  their Fr\'{e}chet derivatives are quite complicated; see Section \ref{Nonlinear part} for details.

We briefly explain the ideas in the estimates of these terms. 
To estimate the Fr\'{e}chet derivatives of $\mathcal{M}_0$ and $\mathcal{M}_1$, we need to represent the inverse of the matrix using the adjugate matrix divided by the determinant.
Since $\text{det}(I-hL_\Sigma)$ is  a polynomial of $h$, we can obtain a lower bound provided that $h$ satisfies some suitable smallness conditions. We can then obtain the $C^k$ norms of these terms. 
Using the inverse function theorem, we  can decompose the Laplace operator in $\mathcal{M}_2$ into
\[
\sum_{i}\sum_{j}  \prts{ (\nabla\Theta)^{-1}_{ij} \circ  \Theta^{-1} }
\prts{ \prts{ \partial_j (\nabla\Theta)^{-1}_{ik} } \circ  \Theta^{-1} } .
\]
Then the product rule, chain rule and  estimates of inverse matrices can  be applied to estimate the Fr\'{e}chet derivative of $\mathcal{M}_2$.
The principle components of $\mathcal{M}_3$ and $\mathcal{M}_4$ are 
similar to terms in $\mathcal{M}_1$ and $\mathcal{M}_2$, and thus can be estimated using similar methods.
By applying similar ideas,  we can obtain the estimate of the surface tension term and the mean curvature term. The latter one has a stronger nonlinearity and contains higher-order differential operators, which leads to  more involved Fr\'{e}chet derivatives and estimates.
When the initial interface $\Gamma_0$ is $C^3$, using the fixed-point arguments as in \cite{Pruss.quali}, we can prove Theorem \ref{existence.result.1}. 
If the initial interface $\Gamma_0$ is generalized to a $W^{3-\frac{2}{q},q}$ surface, then an argument similar to \cite{Pruss.quali} can also be applied to prove Theorem \ref{existence.result.2},
provided that $\Gamma_0$ is a perturbation of some $C^3$ reference surface $\Sigma$.

We will organize the rest of this paper as follows. 
In Section \ref{preliminary}, we review some basic background knowledge.
In Section \ref{Transformation of problem}, 
we use the Hanzawa transformation 
to transform the free-interface problem
into a fixed-interface problem. 
The new equations will be separated into the linear (principal) part and the nonlinear part.
In Section \ref{Linear part}, we study the solvability of the linear part. 
Then we express the nonlinear part using an operator and estimate its Fréchet derivative 
in Section \ref{Nonlinear part}.
Finally, we prove Theorem \ref{existence.result.1} and
Theorem \ref{existence.result.2} in Section \ref{Local existence}.

\section{Preliminaries} \label{preliminary} 

In this section, we introduce the notation, definitions and basic properties for the convenience of the reader.

\subsection{Notation}

The gradient  $\nabla f$ of a scalar function $f$ is considered  as a column vector by default.
When $f$ is a vector-valued function, the gradient of each component is viewed as a column vector in the matrix by default, i.e., 
$(\nabla f )_{ij}:= \partial_i f_j $. 
Then we have the formula
\[\nabla (f\circ g) = \prts{ \nabla g } \prts{ (\nabla f)\circ g }. \]
We denote the $r$-neighborhood of a point $x$ by $B(x;r)$.
For a set $A$, we define 
\[
B(A;r) := \bigcup_{x\in A} B(x;r).
\]
We sometimes use $a \lesssim b$ if $a\leq Cb$ for some universal constant  $C>0$. 
The projection matrix on a surface $S$ is denoted by  $\mathcal{P}_S := I-n_S\otimes n_S$,
where $n_S$ is the normal vector of $S$.
In complicated formulas, we use $\depvar{ \ \cdot \ }$ to denote the values that functions are evaluated at,
e.g. we use $f \depvar{g(x)}$ to express $ (f\circ g)(x)$.

\subsection{Function spaces}

\subsubsection{Continuous and differentiable functions}
In this paper, we mainly consider two types of domains in $\mathbb{R}^3$:
a bounded, open, 3-dimensional domain $\Omega$; and a closed, 2-dimensional surface $\Sigma$.
We say $\Sigma$ is $C^k$ if it can be locally parameterized using $C^k$ functions.
We say $\Omega$  is a $C^k$ domain if its boundary $\partial\Omega$ is a $C^k$ surface.

Let $f$ be a function from $[0,T]$ to a Banach space $X$. 
For any $t_0\in [0,T] $, we say $f$ is continuous at $t_0$ if 
\[
\lim_{t\to t_0} \norm{f (t)- f (t_0)}_{X}=0.
\]
If $f$ is continuous on $[0,T]$ then we say $f\in C([0,T];X)$.
We say $  g(t_0) \in X$ is the derivative of $f$ at $t_0$ if
\[
\lim_{t\to t_0} \norm{ \frac{ f (t)- f (t_0)} { t-t_0 } - g(t_0) }_{X}=0.
\]
If $f$ has a continuous derivative $g\in C([0,T];X)$, then we say $f\in C^1([0,T];X)$.
Similarly, we can define the space $C^k([0,T];X)$ for $k\in\mathbb{Z}$, $k\geq 0$.
We will frequently use the special case that $X=C^m(\Omega)$ or $X=C^m(\Sigma)$ for $m\in\mathbb{Z}$, $m\geq 0$.
\begin{remark}
For vector-valued 
$C^k([0,T];C^m(\Omega))$ or $C^k([0,T];C^m(\Sigma))$,
we define their norms by taking the $C^k([0,T];C^m)$ norm of each component  
and then taking the vector norm.
We view matrices (and thus for matrix-valued functions) as vectors when calculating their matrix norms.
We will use $\norm{f}_{\infty}$ to abbreviate 
$\norm{f}_{L^\infty([0,T];L^\infty (\Omega) )}$, 
$\norm{f}_{L^\infty([0,T];L^\infty (\Sigma) )}$,
$\norm{f}_{ L^\infty (\Omega) }$, or  $\norm{f}_{ L^\infty (\Sigma) }$,
 depending on the context.

\end{remark}

\subsubsection{Fractional Sobolev spaces}

Given $s \in(0,+\infty)\setminus \mathbb{Z}$,
the fractional Sobolev norm is defined as
 \[
\norm{f}_{W^{s,q}(\Omega)} 
:= \norm{f}_{W^{\lfloor s \rfloor,q}(\Omega)}  +  \seminorm{f}_{W^{s- \lfloor s \rfloor,q}(\Omega)} ,
\]
where $\seminorm{f}_{W^{\alpha,q}(\Omega)}$ with $\alpha\in(0,1)$
stands for the Gagliardo seminorm
\[
\seminorm{f}_{W^{\alpha,q}(\Omega)} 
:= \prts{ 
				\int_\Omega \int_\Omega \frac{\abs{f(x)-f(y)}^q   }{\abs{x-y}^{n+ \alpha q}} dydx
				}^{\frac{1}{q}} .
\]
For a Banach-space-valued function, i.e., $f:\Omega\to X$ where $X$ is a Banach space,
the Gagliardo seminorm is defined similarly by replacing $\abs{f(x)-f(y)}  $ with $\norm{f(x)-f(y)}_X$.
This enables us to define the space $W^{s,q}([0,T];X)$.
For a compact  hypersurface $\Sigma$ in $\mathbb{R}^n$, 
the Sobolev space $W^{s,q}(\Sigma)$ can be defined similarly since $\Sigma$ 
can be locally mapped to Euclidean spaces.

\subsection{Calculus on surfaces}

Let $f$ be a function defined in an open domain $\Omega\subseteq \mathbb{R}^m$.
Let $\Sigma$ be a hypersurface in $\Omega$ whose normal vector field is denoted by $n=n_\Sigma$.
For all $x\in\Sigma$, the surface gradient $\nabla_\Sigma f$ is defined as 
\[
\nabla_\Sigma f := \nabla f  - (n \cdot \nabla f  ) n = (I-n\otimes n) \nabla f,
\]
which is the  projection of $\nabla f$ onto $T_x \Sigma$.
If $F$ is a vector-valued function, then we can define the surface divergence by
\[
\div_\Sigma F := {\rm tr} (\nabla_\Sigma F) .
\]
Thus, we can also define the Laplace–Beltrami operator $\triangle_\Sigma$ by
\[
\triangle_\Sigma f := \div_\Sigma \ \nabla_\Sigma f .
\]
In fact, the surface derivatives only depend on the value of the function on $\Sigma$, 
which is discussed in e.g.  \cite[Remark 7.26]{Ambrosio.BV}. 
 The Weingarten tensor of $\Sigma$ is defined as $L_\Sigma := -\nabla_\Sigma n_\Sigma$,
 which is a matrix-valued function defined on $\Sigma$.
For each $x\in\Sigma $, we have $L_\Sigma n_\Sigma =0$; 
and $L_\Sigma  $ is an isomorphism on $T_x \Sigma$. 
The principal curvatures of $\Sigma$ at $x$ are the eigenvalues of the matrix $L_\Sigma\depvar{x}$;
and the mean curvature $H_\Sigma \depvar{x}$  
is defined  as $H_\Sigma \depvar{x}= {\rm tr} L_\Sigma \depvar{x}$. 
 For more details about the calculus on surfaces, we refer  to \cite{Ambrosio.BV, Maggi.BV}.

\subsection{Nearest point projection} \label{section.nearest,point}

For the reader's convenience, we restate the theorem of  nearest point projection in 
 \cite[Section 2.12.3]{Simon.nearest-point}   with  some modification.
\begin{theorem}[\!\!\cite{Simon.nearest-point}, Theorem 1, simplified]
Let $\Sigma$ be a  compact, $(m-1)$-dimensional,   $C^k$ manifold  in $\mathbb{R}^m$.
There exists $\varrho_0(\Sigma)>0$ and a $C^{k-1}$ projection mapping 
$\Pi: B(\Sigma; \varrho_0)\to \Sigma$, such that for all $x\in  B(\Sigma; \varrho_0)$:
\begin{enumerate}  
\item $x-\Pi(x) \perp T_{\Pi (x)}\Sigma$ ;
\item $ {\rm dist} (x,\Sigma) = \abs{x-\Pi(x)}$;
\item for all $y \in\Sigma$ and $y\neq x$ we have ${\rm dist} (x,\Sigma) < \abs{x-y}$;
\item for all $y\in\Sigma$ and $\lambda\in(0,\varrho_0)$, $\Pi(y + \lambda n(y))= y$.
\end{enumerate}
\end{theorem}

\subsection{Fr\'{e}chet derivative}

Given Banach spaces $X$ and $Y$ and an operator $F:X\to Y$. 
Suppose that for $x\in X$ there exists a linear operator $A: X\to Y$ 
such that 
\begin{equation}
\lim_{\norm{h}_X \to 0}\frac{ \norm{F(x+h)-F(x) -Ah}_Y }{ \norm{h}_X }=0 ,
\end{equation}
 then we say that $F$ is Fr\'{e}chet differentiable at $x$. 
The linear operator $A$ is called the Fr\'{e}chet derivative of $F$ at $x$ and is 
denoted by $DF(x)$. 

Suppose that $DF$ exists in an open neighborhood $U$ of $x\in X$, then we have an operator
 $DF: U\to \mathcal{L}(X;Y)$ and the second derivative $D^2 F$ at $x$ can be defined using the same way.
Notice that $D^2 F(x)\in \mathcal{L}(X; \mathcal{L}(X;Y)) $.
Similarly, we can obtain the $n$-th derivative $D^n F$. 
The space 
\begin{equation}
\mathcal{L}(X; \mathcal{L}(X; \cdots \mathcal{L}(X;Y))) 
\end{equation} is equal to the space of multilinear operators
\begin{equation}
\mathcal{L}^{(n)}(X\times\cdots \times X ;Y),
\end{equation}
 which is usually abbreviated to 
$\mathcal{L}^{(n)}(X^n;Y)$.

The product rule and chain rule are still valid for Fr\'{e}chet derivatives.
Given $F: X\to Y_1$ and  $G: X\to Y_2$.  Suppose that the product is well-defined,
 i.e., there exists a bilinear mapping  $Y_1 \times Y_2 \to Z$ which is called the ``product''. 
 Then the value of $D(FG)$ at   $x\in X$ is a linear mapping, i.e.
  $D(FG)\depvar{x} \in\mathcal{L}( X ; Z)$, such that for each $h\in X$ we have 
\begin{equation}
D(FG)\depvar{x}h= \prts {DF\depvar{x} h} \prts {G\depvar{x}} +  \prts {F\depvar{x}} \prts {DG\depvar{x} h}. 
\end{equation}
Given $F: X\to Y$ and  $G: Y\to Z$, the Fr\'{e}chet derivative of $G\circ F : X\to Z$ at an arbitrary $x\in X$ is a linear operator 
$D (G\circ F)[x]\in \mathcal{L}(X;Z)$.
For all $h\in X$ we have
\begin{equation}
D(G\circ F)\depvar{x} h=  DG\depvar{F(x)} ( DF\depvar{x} h ) . 
\end{equation}
The proofs of these properties can be found in    \cite[Section 4.3]{Zeidler.Functional.1}.

\section{Reformulation of the Problem} \label{Transformation of problem}

In this section, we apply the Hanzawa transformation to the original equations 
\eqref{main.1.u} - \eqref{main.1.initial}.
This allows us to turn the free-interface problem into a fixed-interface problem. 
We refer to \cite{Pruss.quali, Pruss.green.book} for more details about this method.

\subsection{Representation of the free interface}

The key point of this part is to represent the free interface $\Gamma(t)$ 
using  a fixed reference surface and a height function.
This method  is called normal parameterization; see \cite[Section 2]{Pruss.green.book}
for more details.

Given a closed $C^3$ surface $\Sigma\subseteq \Omega$. By the theory of nearest point projection introduced in Section \ref{preliminary},
there exists a tubular neighborhood  $B(\Sigma;\varrho_0)\subseteq \Omega$,
such that the mapping
\begin{equation}
F(x,r):=x+r n_{\Sigma}(x) 
\end{equation}
is a diffeomorphism from  $\Sigma\times (-\varrho_0,\varrho_0) $ to $B(\Sigma;\varrho_0)$. 
Its inverse mapping is
\begin{equation}
F^{-1}(x)=(\Pi(x), d(x)),
\end{equation}
where $\Pi(x)$ is the projection of $x$ onto $\Sigma$, 
and $d(x)$ is the signed distance between $x$ and $\Sigma$, 
where the positive side of $\Sigma$  coincides with the direction of exterior normal vector $n_{\Sigma}$.
Suppose that the free interface $\Gamma(t)\subseteq B(\Sigma;\rho_0) $ for $t\in[0,T]$, then it can be represented using a height function $h(t,x)$ defined on $[0,T]\times \Sigma$.

For every height function $h$, 
we can define a diffeomorphism in $\Omega$ using the same idea as in \cite{Pruss.quali}:
\begin{equation} \label{def,Theta,diffeo}
\Theta_h(x) := 
\left\{ 
\begin{aligned}
x   + \eta(d(x) / \varrho_0 ) & h(\Pi(x)) n_{\Sigma}(\Pi(x)), & \quad x\in B(\Sigma;\varrho_0),
\\
& x  , & \quad x\notin B(\Sigma;\varrho_0),
\end{aligned}
\right.
\end{equation}
where $\eta$ is a smooth cut-off function on $\mathbb{R}$ such that
 $0\leq \eta \leq 1$; $\eta(s)=1$ when $\abs{s} < 1/3$, and $\eta(s)=0$ when $\abs{s} > 2/3$.
For convenience, we  define $\rho:= \rho_0/3$; and define the displacement of the point  $x$ under the diffeomorphism by
\begin{equation} \label{def,theta,change}
\theta_h(x) :=  \Theta_h(x) -x .
\end{equation}

\begin{remark}
In order to make $\Theta_h$ a bijection, we need to guarantee that  $\norm{h}_{C^0(\Sigma)}$ is sufficiently small.
As an example,
we consider a mapping in $\mathbb{R}^2$.
Suppose that $\Sigma$ is the  $x$-axis and $h(x)=b>0$, 
then $\Theta(x,y)= (x, y+ b\eta(y / \varrho_0))$.
To make $\Theta$ a bijection, 
it is necessary that the function $f(y):= y+ b\eta(y/\varrho_0)$ should be a bijection, 
which requires $b$, i.e., $\abs{h}$, to be sufficiently  small.
\end{remark}

\begin{remark}

The derivative of the distance function is 
\begin{equation}
\nabla d(x) = n(\Pi(x)).
\end{equation} 
To calculate the derivative of the projection mapping $\Pi$,
we take the derivative on the equation $ x - \Pi (x) = d(x)n(\Pi (x)) $, which implies
\begin{equation}
\begin{aligned}
&
I - \nabla \Pi (x) = d(x) \nabla_x \prts{ n (\Pi(x)) }  + \nabla_x d(x) \otimes n (\Pi(x))
\\
&
=  d(x) \nabla \Pi (x) \nabla_{\Sigma} n (\Pi(x))  + n (\Pi(x)) \otimes n (\Pi(x)).
\end{aligned}
\end{equation}
Recall that the Weingarten tensor is $L_{\Sigma}:= - \nabla_{\Sigma} n  $, 
so we have
\begin{equation}
 I - n (\Pi(x)) \otimes n (\Pi(x))  
= \nabla \Pi (x) \prts{ I - d(x) L_{\Sigma} (\Pi(x)) } ,
\end{equation}
which implies 
\begin{equation} \label{derivative.proj}
\nabla \Pi (x)
=
 \mathcal{P}_{\Sigma}  (\Pi(x)) \prts{ I - d(x) L_{\Sigma} (\Pi(x)) }^{-1}   .
\end{equation}

\end{remark}

We refer to  \cite[Section 2.3]{Pruss.green.book} and \cite[Section 2.12.3]{Simon.nearest-point}
for more details on the nearest point projection and related calculations.

\subsection{Fluid terms}
Using the diffeomorphism $\Theta_h$,
the original equations with interface $\Gamma(t)$ can be transformed into
equations with interface $\Sigma$. 
The interface $\Gamma(t)$ can be determined 
as long as the function $h(t,x): [0,T]\times\Sigma$ can be obtained. 
In order to do this,
we need to change the variables in \eqref{main.1.u}-\eqref{main.1.initial}  
with the help of $\Theta_h$.
We shall use similar arguments to \cite{Pruss.quali} 
with more details included for completeness.

We first consider \eqref{main.1.u}, which is equivalent to 
\begin{equation} \label{diff,rewrite.eq:u}
\left(	\partial_t u + (u\cdot \nabla) u - (B\cdot \nabla) B + \frac{1}{2}\nabla(\abs{B}^2)
	+ \nabla p   -  \nu^\pm \triangle u    
\right) \circ \Theta_h =0
\end{equation}
since $\Theta_h$  is a bijection.
We  use  $\Theta_h^{-1}(t,\cdot)$ 
to denote the inverse mapping (with respect to the space variable)  of
$ \Theta_{h(t,\cdot)} $ at time $t$,
which implies
 $\Theta_h^{-1}(t,\Theta_h (t,x))=x$.
Given any function $f$ in $\Omega$, we define 
\begin{equation}
\overline{f}(t, x):= f(t,\Theta_h (t,x) ),
\end{equation}
where $f$ can be $u$, $B$, $p$, etc.  
For any fixed $t$, by the definition of $\Theta_h$ we have
 $\Theta_h^{-1} (t,\cdot) =(\Theta_{h(t,\cdot) })^{-1}   $,
 which implies that  
\begin{equation}
f(t,x ) = \overline{f}(t, \Theta_h^{-1} (t,x)) .
\end{equation}
In this paper, we will omit the subscript $h$ in $\Theta_h$ when there is no confusion.
As opposed to partial derivatives with respect to specific variables, 
we will use $\partial_0$ temporarily  
to denote the partial derivative with respect to the first time variable only. 
This helps distinguish $\partial_0 u(t,\Theta(t,x))$ and $\partial_t u(t,\Theta(t,x))$, where the latter requires the chain rule.

Using  similar arguments  to \cite{Pruss.quali},
we write the transformed  equations in terms of the new variables.
 More details can be found in  \cite[Section 1.3]{Pruss.green.book}.
For the time derivative, we have
\begin{equation} \label{t,derivative,u;rewrite}
\partial_t u(t, x) 
= \partial_t( \overline{u}(t,\Theta^{-1} (t,x) )) 
= (\partial_0 \overline{u}) \circ \Theta^{-1}
+
\sum_{i}((\partial_i \overline{u}) \circ \Theta^{-1} ) \partial_t \Theta_{i}^{-1},
\end{equation}
which implies
\begin{equation} \label{t,derivative,u;rewrite2}
\partial_t u \circ \Theta 
= \partial_0 \overline{u} 
+ (\partial_t \Theta^{-1}\circ \Theta) \nabla  \overline{u}
=\partial_t \overline{u} 
+ (\partial_t \Theta^{-1}\circ \Theta) \nabla  \overline{u}.
\end{equation}
Next, we calculate the formula of each entry of $\nabla\overline{u}$, which is
\begin{equation} \label{grad.diff.rewrite}
\begin{aligned}
\partial_{\alpha} u_{\beta} (t,x)
&=
\partial_{x_{\alpha}} u_{\beta} (t,x)
=
\partial_{x_{\alpha}} \overline{u}_{\beta}( t , \Theta^{-1} (t,x) )
= 
\sum_{i} ((\partial_{i} \overline{u}_{\beta})\circ \Theta^{-1} ) \partial_{x_{\alpha}} \Theta^{-1}_{i}
\\
&
=
\sum_{i} ((\partial_{i} \overline{u}_{\beta})\circ \Theta^{-1} ) (((\nabla \Theta)^{-1})_{\alpha i}\circ \Theta^{-1}).
\end{aligned}
\end{equation}
Taking the time derivative of the equation $\Theta(t,\Theta^{-1}(t,x))=x$, we obtain
\begin{equation} \label{dt,Theta,transform}
(\partial_0 \Theta)\circ \Theta^{-1}
+
(\partial_t\Theta^{-1}) \prts{(\nabla\Theta)\circ \Theta^{-1} } =0.
\end{equation}
Composing \eqref{dt,Theta,transform} with $\Theta$ gives
\[
\partial_0 \Theta
+
(\partial_t\Theta^{-1} \circ \Theta) (\nabla\Theta)  =0.
\]
Then we have
\begin{equation} \label{dt,Theta,inverse,1}
\prts{\partial_t\Theta^{-1} } \circ \Theta
=
-  (\partial_0 \Theta) \prts{ \nabla\Theta }^{-1} 
= 
-  (\partial_t \Theta) \prts{ \nabla\Theta }^{-1},
\end{equation}
where $\partial_0 \Theta$ is changed back to  $\partial_t  \Theta$ in the second equality.
From \eqref{dt,Theta,inverse,1} and the equality 
$$(\nabla\Theta)^{-1}=I - (I+\nabla\theta)^{-1} \nabla\theta,$$ we obtain
\begin{equation} \label{dt,rewrite,simple}
\prts{ \partial_t\Theta^{-1}} \circ\Theta
=
-  \partial_t\Theta (\nabla\Theta)^{-1}
=
-  \partial_t\theta  \prts{ I - (I+\nabla\theta)^{-1} \nabla\theta  }.
\end{equation}
Similarly, we obtain from \eqref{grad.diff.rewrite} that
\begin{equation} \label{grad,rewrite,simple}
(\nabla u)\circ  \Theta 
=
(\nabla \Theta)^{-1} \nabla\overline{u}
=
\nabla\overline{u} 
-
\prts{ (I+\nabla\theta)^{-1} \nabla\theta }\nabla\overline{u},
\end{equation}
\begin{equation} \label{diffeo.rewrite.div}
({\rm div}u )\circ\Theta
= 
\prts{ (\nabla \Theta)^{-1} }^{\top}: \nabla \overline{u}
= 
{\rm div}\overline{u} - \prts{ (I+\nabla\theta)^{-1} \nabla\theta }: \nabla\overline{u}
.
\end{equation}
Taking one more derivative on \eqref{grad.diff.rewrite},
we obtain the Laplacian of each entry of $u$: 
\begin{equation} \label{laplace,u,diffeo}
\begin{aligned}
\triangle u_{\beta} 
&=
\sum_{\alpha} \partial_{x_{\alpha}} \partial_{x_{\alpha}} \overline{u}_{\beta}(t,\Theta^{-1}(t,x)) 
=
\sum_{\alpha} \partial_{x_{\alpha}}
\left(
\sum_{i} ((\partial_{i} \overline{u}_{\beta})\circ \Theta^{-1} ) \partial_{x_{\alpha}} \Theta^{-1}_{i}
\right)
\\
&
=
\sum_{\alpha}\sum_{i} \partial_{x_{\alpha}} \left( (\partial_{i} \overline{u}_{\beta})\circ \Theta^{-1}  \right) \partial_{x_{\alpha}}\Theta_{i}^{-1}
+
\sum_{\alpha}\sum_{i}  \left( (\partial_{i} \overline{u}_{\beta})\circ \Theta^{-1}  \right) \partial_{x_{\alpha} x_{\alpha}} \Theta_{i}^{-1} 
\\
&
=
\sum_{\alpha}\sum_{i} \sum_{j} \left(   (\partial_{ji} \overline{u}_{\beta})\circ \Theta^{-1}  \right) \partial_{x_{\alpha}}\Theta_{i}^{-1} \partial_{x_{\alpha}}\Theta_{j}^{-1}
+
\sum_{i} \triangle \prts{ \Theta_{i}^{-1} }  \left( (\partial_{i} \overline{u}_{\beta})\circ \Theta^{-1}  \right)  
\\
&
= \prts{ \prts{\nabla \Theta^{-1}}^{\top}  \prts{\nabla \Theta^{-1}}  } : 
\prts{ \prts { \nabla^2 \overline{u}_{\beta} } \circ \Theta^{-1}  }
+
\triangle \Theta^{-1} \cdot \prts{ (\nabla \overline{u}_{\beta}) \circ \Theta^{-1} },
\end{aligned}
\end{equation}
where the notation $\nabla^2 \overline{u}_{\beta}$ denotes the Hessian matrix of $\overline{u}_{\beta}$.
The subscript in \eqref{laplace,u,diffeo} can be removed to obtain the vector equality
\begin{equation} \label{laplace,u;rewrite2}
(\triangle u )\circ \Theta
= \prts{ \prts{ \prts{\nabla \Theta^{-1}}^{\top}  \prts{\nabla \Theta^{-1}}   } \circ \Theta }  : \nabla^2 \overline{u}
+
\prts{ \prts{\triangle  \Theta^{-1} } \circ \Theta }   (\nabla \overline{u}).
\end{equation}
In the first term of the right-hand side of \eqref{laplace,u;rewrite2},
using the formula of inverse functions, we have
\begin{equation} \label{laplace,u;rewrite,another}
\begin{aligned}
&
 \prts{ \prts{\nabla \Theta^{-1}}^{\top}  \prts{\nabla \Theta^{-1}}   } \circ \Theta 
 =
 \prts{\nabla \Theta}^{-\top}  \prts{\nabla \Theta}^{-1}   ,
\end{aligned}
\end{equation}
which implies that
\begin{equation} \label{laplace,u;rewrite}
\begin{aligned}
(\triangle u )\circ \Theta
&= \prts{  \prts{\nabla \Theta}^{-\top}  \prts{\nabla \Theta}^{-1} }  : \nabla^2 \overline{u}
+
\prts{ \prts{\triangle  \Theta^{-1} } \circ \Theta }  \cdot  (\nabla \overline{u})
\\
&=
\triangle \overline{u} +
\prts{  \prts{\nabla \Theta}^{-\top}  \prts{\nabla \Theta}^{-1} - I }  : \nabla^2 \overline{u}
+
\prts{ \prts{\triangle  \Theta^{-1} } \circ \Theta }  \cdot  (\nabla \overline{u}).
\end{aligned}
\end{equation}

Using exactly the same argument as for $u$,
we can transform the terms which contain $B$ and $p$.
Using also \eqref{dt,Theta,inverse,1}, we can 
 rewrite \eqref{diff,rewrite.eq:u} as
\begin{equation} \label{eq:u,transformed}
\begin{aligned}
&
 \partial_t \overline{u}  
-  \partial_t \Theta  \, \prts{ \nabla\Theta }^{-1}   \nabla  \overline{u} 
+  \overline{u}   \prts{ (\nabla \Theta)^{-1} } \nabla  \overline{u} 
-    \overline{B}   \prts{ (\nabla \Theta)^{-1} } \nabla  \overline{B} 
\\
&
+ \frac{1}{2} \prts{ (\nabla \Theta)^{-1} } \nabla \prts{ \abs{ \overline{B} }^2 }
+ \prts{ (\nabla \Theta)^{-1} } \nabla \overline{p}
- \nu^\pm
 \prts{   (\nabla \Theta)^{-\top} (\nabla \Theta)^{-1}     }  : \nabla^2 \overline{u}
\\
&
- \nu^\pm
\prts{ \prts{\triangle  \Theta^{-1} } \circ \Theta }  \cdot  (\nabla \overline{u})
=0.
\end{aligned}
\end{equation}
Using the fact that
$\nabla\times(u\times B)= -(u\cdot \nabla )B + (B\cdot \nabla )u $ 
and $\nabla\times(\nabla \times B)=\nabla({\rm div} B) -\triangle B$,
we rewrite \eqref{main.1.B} as
\begin{equation} \label{main.1.B.simple}
\begin{aligned}
\prts{ \partial_t B + (u\cdot \nabla) B - (B\cdot \nabla)u -\sigma \triangle B} \circ \Theta =0.
\end{aligned}
\end{equation}
Using the same arguments as in
\eqref{t,derivative,u;rewrite2} , \eqref{grad.diff.rewrite} 
and \eqref{laplace,u;rewrite}, 
we can obtain the representation of $\partial_t B$, $\nabla B$ and $\triangle B$,
which imply
\begin{equation} \label{main.1.B;rewrite}
\begin{aligned}
&
 \partial_t \overline{B} 
-  \partial_t \Theta  \, (\nabla\Theta)^{-1}      \nabla  \overline{B} 
+  \overline{u}   (\nabla\Theta)^{-1} \nabla  \overline{B} 
-    \overline{B}   (\nabla\Theta)^{-1} \nabla  \overline{u} 
\\
&
- \sigma 
 \prts{   (\nabla \Theta)^{-\top} (\nabla \Theta)^{-1}     }  : \nabla^2 \overline{B}
- \sigma 
\prts{ \prts{\triangle  \Theta^{-1} } \circ \Theta }  \cdot  (\nabla \overline{B})
=0.
\end{aligned}
\end{equation}
The divergence-free conditions \eqref{main.divu.0} and \eqref{main.divB.0}
can be treated in the same way using \eqref{diffeo.rewrite.div}, which implies
\begin{equation} \label{divu,0,transform}
\begin{aligned}
0 = \div u = 
{\rm div}\overline{u} - \prts{ (I+\nabla\theta)^{-1} \nabla\theta }: \nabla\overline{u},
\end{aligned}
\end{equation}
\begin{equation} \label{divB,0,transform}
\begin{aligned}
0 = \div B = 
{\rm div}\overline{B} - \prts{ (I+\nabla\theta)^{-1} \nabla\theta }: \nabla\overline{B}.
\end{aligned}
\end{equation}
We will show in Section \ref{Transformed equations} that \eqref{divB,0,transform} can actually be ignored in the transformed problem.

\subsection{Geometric terms}

In the previous section, we have obtained the transformation of 
\eqref{main.1.u} to \eqref{main.divB.0}. 
It remains to transform 
\eqref{eq:s,interface} and \eqref{main.1.chi},
which require suitable representations of tangent vectors, normal vectors and curvatures of the free interface $\Gamma(t)$. 
To treat these terms,  we follow the arguments in e.g. \cite{Pruss.quali, Pruss.green.book}.
We have included more details in Appendix \ref{Appendix. Geometric Terms} for completeness.  

Without loss of generality, we temporarily omit the time variable in the free interface $\Gamma(t)$ and  height function $h(t,x)$.
The normal vector $n_\Gamma$ and mean curvature $H_\Gamma$  of  $\Gamma$ can be expressed in terms of $\Sigma$ and $h$ using formulas:
\begin{equation} \label{interface normal vector representing}
n_\Gamma
=  \beta \prts{ n_\Sigma - \prts{ I-hL_{\Sigma}}^{-1} \nabla_{\Sigma}h } ,
\end{equation}
\begin{equation}
\label{interface curvature representing}
H_{\Gamma}
=
\beta
{\rm tr} \prts{ \prts{I-hL_{\Sigma}}^{-1} \prts{ L_{\Sigma} + \nabla_{\Sigma}\alpha } } 
-
\beta^3 \prts{  \prts{I-hL_{\Sigma}}^{-1} \alpha    } 
 \,
 \nabla_{\Sigma} \alpha \, \alpha ^\top.
\end{equation}
where 
\begin{equation} 
\alpha := \prts{ I-hL_{\Sigma}}^{-1} \nabla_{\Sigma}h,
\end{equation}  
\begin{equation} \label{curvature.beta.term}
  \beta : =
\frac{  1 }{  \abs{ n_\Sigma - \prts{ I-hL_{\Sigma}}^{-1} \nabla_{\Sigma}h} },
\end{equation}
 as used in \cite{Pruss.green.book}.

Using the formulas of $n_\Gamma$ and $H_\Gamma $ in \eqref{interface normal vector representing} 
and \eqref{interface curvature representing}, 
we are able to transform equations  
\eqref{eq:s,interface} and \eqref{main.1.chi}.
Composing $\Theta$ with \eqref{eq:s,interface}, we have
\begin{equation} \label{stress,eq,rewrite,temp1}
- \prts{  \jump{ \nu^\pm \prts{ \nabla u + \nabla u^\top } -pI }n_{\Gamma} }\circ \Theta
= \prts{ \kappa H_{\Gamma} n_{\Gamma} }\circ \Theta.   
\end{equation}
Due to the effect of $\Theta$, the equation \eqref{stress,eq,rewrite,temp1} is defined on $\Sigma$ rather than $\Gamma(t).$
We calculate the projections of \eqref{stress,eq,rewrite,temp1} to $n_{\Sigma}(x)$ 
and $T_x\Sigma$ respectively using the same arguments as in  \cite[Section 2]{Pruss.quali} 
with some details included for  convenience.
Since $\prts{ I-hL_{\Sigma}}^{-1} \nabla_{\Sigma}h$ is tangent to $\Sigma$, we have
 $n_{\Gamma}\cdot n_{\Sigma}=\beta$.
Taking the inner product of \eqref{stress,eq,rewrite,temp1} and $n_{\Sigma} / \beta$, we obtain
the projection of the equation onto the normal vector
\begin{equation}  \label{stress,eq,rewrite,temp1.5}
- \prts{  \jump{\prts{ \nu^\pm \prts{ \nabla u + \nabla u^\top
}\circ \Theta
 } }
\prts{ n_{\Sigma} - \prts{ I-hL_{\Sigma}}^{-1} \nabla_{\Sigma}h } 
} \cdot n_{\Sigma}
 + \jump{p} 
=
\kappa H_{\Gamma} .
\end{equation}
From $\nabla u\circ \Theta= \nabla \overline{u} - \mathcal{M}_1  \nabla \overline{u}$ 
(see  Section \ref{Transformed equations} for definitions of $\mathcal{M}_i$, $i=0,\cdots, 4$)
we have
\begin{equation}
\begin{aligned}
\prts{
\nu^\pm \prts{ \nabla u + \nabla u^\top
}
}\circ \Theta
=
\nu^\pm
\prts{
	\nabla \overline{u} - \mathcal{M}_1  \nabla \overline{u}
	+ (\nabla \overline{u})^{\top} - (\mathcal{M}_1  \nabla \overline{u} )^{\top}
} .
\end{aligned}
\end{equation}
Thus, we can rewrite \eqref{stress,eq,rewrite,temp1.5} as
\begin{equation} \label{stress,eq,rewrite,temp2}
\begin{aligned}
&
\jump{p}  - \kappa H_{\Gamma}
\\
&
=
\prts{  \jump{\prts{ \nu^\pm \prts{ \nabla \overline{u} + \nabla \overline{u}^\top
} }} n_{\Sigma}} \cdot  n_{\Sigma}
 -
\prts{  \jump{\prts{ \nu^\pm \prts{ \nabla \overline{u} + \nabla \overline{u}^\top
} }} \mathcal{M}_0 \nabla_{\Sigma}h} \cdot  n_{\Sigma}
\\
& \quad -
\prts{  \jump{ \nu^\pm \prts{ \mathcal{M}_1\nabla \overline{u} +  (\mathcal{M}_1\nabla \overline{u} ) ^\top
} } \prts{ n_{\Sigma} - \mathcal{M}_0 \nabla_{\Sigma}h }  } \cdot  n_{\Sigma}
\\
&=:
\prts{  \jump{ \nu^\pm \prts{ \nabla \overline{u} + \nabla \overline{u}^\top
} } n_{\Sigma}} \cdot  n_{\Sigma}
+ \mathcal{G}_1.
\end{aligned}
\end{equation}
Letting 
\begin{equation} \label{def,G2}
 \mathcal{G}_2:= \kappa ( H_{\Gamma} - DH_{\Gamma}\depvar{0} h ),
\end{equation}
where $DH_{\Gamma}\depvar{0} $ is the Fr{\'e}chet derivative of $H_\Gamma$ at $h=0$.
Similarly to \cite{Pruss.quali}, we  linearize the mean curvature $H_{\Gamma}$ using the equality 
\[ \kappa  H_{\Gamma} = \kappa DH_{\Gamma}\depvar{0} h +  \mathcal{G}_2. \]
The equation \eqref{stress,eq,rewrite,temp2} can then be written as
\begin{equation} \label{normal,proj,stress,eq}
\jump{p} - \kappa DH_{\Gamma}\depvar{0}h -  \mathcal{G}_2 
=
\prts{  \jump{ \nu^\pm \prts{ \nabla \overline{u} + \nabla \overline{u}^\top
} } n_{\Sigma}} \cdot  n_{\Sigma}
+ \mathcal{G}_1  ,
\end{equation}
which is the normal projection of \eqref{stress,eq,rewrite,temp1}.
Next, we calculate the tangential projection of \eqref{stress,eq,rewrite,temp1}.
We first notice that for a symmetric matrix 
$A=(a_{ij})_{d\times d}$ and the normal vector $n$, we have 
for any $1\leq i\leq d$ that
\begin{equation}
(\mathcal{P}An)_i 
= \sum_{k=1}^{d} \sum_{j=1}^{d} (\delta_{ik}- n_i n_k) a_{kj} n_j
=\sum_{j=1}^{d} a_{ij} n_j 
  -    n_i \sum_{k=1}^{d}  n_k  \prts{ \sum_{j=1}^{d} a_{kj} n_j },
\end{equation}
i.e., $\mathcal{P}An = An - ((An)\cdot n)n$.
Now we let $A:=\jump{\nu^\pm (\nabla u + \nabla u^{\top})}\circ\Theta$
and
recall that $\alpha= \mathcal{M}_0 \nabla_{\Sigma}h$.
We cancel $ \kappa H_{\Gamma} $ and $ \jump{p}$ by multiplying   \eqref{stress,eq,rewrite,temp1.5}  with $n_\Gamma$ and subtract it from 
\eqref{stress,eq,rewrite,temp1}. This implies 
\begin{equation} \label{stress,eq,rewrite,temp3}
A  (n_{\Sigma}-\alpha)
=
\prts{ \prts{ A (n_{\Sigma}-\alpha) }\cdot n_{\Sigma} }  (n_{\Sigma}-\alpha)
\end{equation}
Using \eqref{stress,eq,rewrite,temp3} (in the third equality below), we have
\begin{equation} \label{tangential,stress,eq,1}
\begin{aligned}
\mathcal{P}_{\Sigma} A (n_{\Sigma}-\alpha)
&= ( I-n_{\Sigma}\otimes n_{\Sigma} ) A (n_{\Sigma}-\alpha)
\\
&
= A (n_{\Sigma}-\alpha) - \prts{ \prts{A (n_{\Sigma}-\alpha)}\cdot n_{\Sigma} } n_{\Sigma} 
\\
&
= \prts{ \prts{ A (n_{\Sigma}-\alpha) }\cdot n_{\Sigma} }  (n_{\Sigma}-\alpha)
	- \prts{ \prts{A (n_{\Sigma}-\alpha)}\cdot n_{\Sigma} } n_{\Sigma} 
	\\
	&
= - \prts{ \prts{ A (n_{\Sigma}-\alpha) }\cdot n_{\Sigma} }  \alpha.
\end{aligned}
\end{equation}
Substituting with
\begin{equation}
\begin{aligned}
A&=\jump{\nu^\pm (\nabla u + \nabla u^{\top})}\circ\Theta
\\
&= \jump{ \nu^\pm (\nabla \overline{u} + \nabla \overline{u}^{\top}  
-\mathcal{M}_1 \nabla \overline{u} - \prts{\mathcal{M}_1 \nabla \overline{u}}^{\top}
)},
\end{aligned}
\end{equation}
we obtain
\begin{equation} \label{tangential,stress,eq,2}
\begin{aligned}
& \mathcal{P}_{\Sigma}   \jump{ \nu^\pm (\nabla \overline{u} + \nabla \overline{u}^{\top}  
-\mathcal{M}_1 \nabla \overline{u} - \prts{\mathcal{M}_1 \nabla \overline{u}}^{\top}
)}    (n_{\Sigma}-\alpha)
\\
& =
- \prts{ \prts{ \jump{ \nu^\pm (\nabla \overline{u} + \nabla \overline{u}^{\top}  
-\mathcal{M}_1 \nabla \overline{u} - \prts{\mathcal{M}_1 \nabla \overline{u}}^{\top}
)} (n_{\Sigma}-\alpha) }\cdot n_{\Sigma} }  \alpha.
\end{aligned}
\end{equation}
Expanding the brackets in the left-hand side of \eqref{tangential,stress,eq,2} 
and rearranging the terms, we have
\begin{equation} \label{def,G3}
\begin{aligned}
&\mathcal{P}_{\Sigma} 
\jump{
\nu^\pm \prts{\nabla \overline{u} + \nabla \overline{u}^{\top}  
}
}n_{\Sigma}
\\
&=
\mathcal{P}_{\Sigma}
\jump{
  \nu^\pm (I-\mathcal{M}_1) \nabla \overline{u} +  \nu^\pm \prts{ (I-\mathcal{M}_1) \nabla \overline{u} }^{\top}  
} \mathcal{M}_0 \nabla_{\Sigma}h \\
&\quad +
\mathcal{P}_{\Sigma}
\jump{
 \nu^\pm  \prts{ \mathcal{M}_1 \nabla \overline{u} 
+  ( \mathcal{M}_1 \nabla \overline{u} ) ^{\top}  }  
} n_{\Sigma}
\\
& \quad -
\prts{
\prts{
\jump{
 \nu^\pm (I-\mathcal{M}_1) \nabla \overline{u} +  \nu^\pm \prts{ (I-\mathcal{M}_1) \nabla \overline{u} }^{\top}  
} ( n_{\Sigma} -  \mathcal{M}_0 \nabla_{\Sigma}h ) 
} \cdot n_{\Sigma}
}
\mathcal{M}_0 \nabla_{\Sigma}h
\\
 & =: - \mathcal{G}_3 . 
\end{aligned}
\end{equation}
We refer readers to \cite{Pruss.quali, Pruss.green.book, Pruss.analytic.solu} for more details on the derivation of $\mathcal{G}_1$, $\mathcal{G}_2$, and $\mathcal{G}_3$.

Now we can combine the tangential and normal projections by multiplying
\eqref{normal,proj,stress,eq} with $n_\Sigma$ and then add it to \eqref{def,G3}.
This derives the transformation of \eqref{eq:s,interface}.
From  \eqref{grad,rewrite,simple},  \eqref{interface curvature representing}, \eqref{normal,proj,stress,eq}, and \eqref{def,G3}, 
we obtain
\begin{equation} \label{eq:s,interface,transformed,temp1}
\begin{aligned}
- \jump{\prts{ \nu^\pm \prts{ \nabla \overline{u} + \nabla \overline{u}^\top
} }} n_{\Sigma}
+  \jump{p} n_{\Sigma} 
- \kappa \prts{ DH\depvar{0} h } n_\Sigma
=  \prts{ \mathcal{G}_1 + \mathcal{G}_2 } n_\Sigma + \mathcal{G}_3 .
\end{aligned}
\end{equation}
In \eqref{derivative,curvature,at,0},
we will find that 
$DH_{\Gamma}\depvar{0}=\prts{ {\rm tr}L_{\Sigma}^2 + \triangle_{\Sigma} }$,
which enables us to obtain the final version of the transformed equation (as in \cite{Pruss.quali})
\begin{equation} \label{eq:s,interface,transformed}
\begin{aligned}
- \jump{\prts{ \nu^\pm \prts{ \nabla \overline{u} + \nabla \overline{u}^\top
} }} 
+  \jump{p} n_{\Sigma} 
- \kappa \triangle_\Sigma h  
=  \prts{ \mathcal{G}_1 + \mathcal{G}_2 +  \kappa {\rm tr}L_\Sigma^2 h } n_\Sigma + \mathcal{G}_3 .
\end{aligned}
\end{equation}

Now we transform \eqref{main.1.chi}. 
By \cite[Section 2.5.2]{Pruss.green.book}, the velocity of the interface satisfies
\begin{equation} \label{normal,velocity,eq,rewrite,temp1}
\beta\partial_t h = V_{\Gamma} \circ \Theta_t^{-1}  = (u\cdot n_{\Gamma}) \circ \Theta_t^{-1} =\beta  \overline{u}\cdot \prts{ n_{\Sigma} - \mathcal{M}_0\nabla_{\Sigma}h   } ,
\end{equation}
which implies that
$
\partial_t h   =   \overline{u} \cdot \prts{ n_{\Sigma}- \mathcal{M}_0\nabla_{\Sigma}h }
$
and can then be rewritten as
\begin{equation} \label{normal,velocity,eq,rewrite}
\partial_t h   -   \overline{u} \cdot n_{\Sigma} +  b\cdot\nabla_\Sigma h
= \prts{ I- \mathcal{M}_0 } \nabla_{\Sigma}h \cdot \overline{u}   + (b-\overline{u})\nabla_{\Sigma}h. 
\end{equation}
The term $ b\in W^{1-\frac{1}{2q} , q}([0,T];L^q(\Sigma)) \cap L^q([0,T]; W^{2-\frac{1}{q} ,q}(\Sigma) )$ 
is an auxiliary function,
which will be specially selected in latter sections.
For details on $V_{\Gamma}$ and the derivation of \eqref{normal,velocity,eq,rewrite}, we refer to Section 2.2.5 in \cite{Pruss.green.book}.

\subsection{Transformed equations} \label{Transformed equations}

Similarly to \cite{Pruss.quali}, 
we abbreviate some terms which will be frequently used in later calculations.
We refer readers to \cite{Pruss.quali, Pruss.green.book} for more details about these terms.
In \eqref{interface normal vector representing}, we define
\begin{equation}
\begin{aligned}
\mathcal{M}_0 := \prts{ I-hL_{\Sigma}}^{-1} .
\end{aligned}
\end{equation}
In \eqref{grad,rewrite,simple}, we define
\begin{equation}
\begin{aligned}
\mathcal{M}_1 :=
 (I+\nabla\theta)^{-1} \nabla\theta.
\end{aligned}
\end{equation}
In \eqref{laplace,u;rewrite}, we define
\begin{equation}
\begin{aligned}
\mathcal{M}_2 :=
\prts{ \triangle\Theta^{-1} }\circ \Theta, \quad \text{and}  
\quad
\mathcal{M}_4 :=
  \prts{\nabla \Theta}^{-\top}  \prts{\nabla \Theta}^{-1} - I  .
\end{aligned}
\end{equation}
In \eqref{dt,rewrite,simple}, we define
\begin{equation}
\begin{aligned}
\mathcal{M}_3 := \partial_t\theta \prts{ I - (I+\nabla\theta)^{-1} \nabla\theta  }  .
\end{aligned}
\end{equation}
We also  remind the readers that the abbreviations $\mathcal{G}_1$, $\mathcal{G}_2$
and $\mathcal{G}_3$ in the transformation of the surface tension equation are defined in
\eqref{stress,eq,rewrite,temp2}, \eqref{def,G2} and \eqref{def,G3}.

Notice that  \eqref{main.1.B} has the form $\partial_t B = \nabla\times F$. 
Thus, taking the divergence of the equation, we obtain
\begin{equation}
\begin{aligned}
\partial_t {\rm div} B  = 0.
\end{aligned}
\end{equation}
Since the solution $\overline{B}$ in the transformed equations satisfies ${\rm div} B_0 =0$,
the equation \eqref{main.divB.0} will always be satisfied and thus can be ignored in the transformed problem.

Notice that the viscosity $\nu^\pm$ in the transformed problem is independent of time
since the interface in the transformed problem is a fixed surface. 
Finally, using \eqref{eq:u,transformed},  \eqref{main.1.B;rewrite},  
\eqref{divu,0,transform}, \eqref{eq:s,interface,transformed}
and \eqref{normal,velocity,eq,rewrite},
the equations 
\eqref{main.1.u} - \eqref{main.1.initial} 
can be  transformed  to
\begin{equation} \label{transformed.eq.1}
\begin{aligned}
 \partial_t \overline{u}  
+ \nabla\overline{p}
-\nu^\pm \triangle \overline{u}
  =  
&- \frac{1}{2} \nabla \prts{ \abs{ \overline{B} }^2 }
-  \overline{u} \nabla  \overline{u} 
+    \overline{B}   \nabla  \overline{B} 
+\mathcal{M}_3 \nabla\overline{u}
+  \overline{u}   \mathcal{M}_1   \nabla  \overline{u} 
\\
& 
-  \overline{B}   \mathcal{M}_1   \nabla  \overline{B}   
+ \frac{1}{2} \mathcal{M}_1  \nabla \prts{ \abs{ \overline{B} }^2 }
+ \mathcal{M}_1  \nabla\overline{p}
+
\overline{\nu}
\mathcal{M}_4
: \nabla^2 \overline{u}\\
&+
\nu^\pm
\mathcal{M}_2  \cdot  (\nabla \overline{u})  ,
\end{aligned}
\end{equation}
\begin{equation} 
\begin{aligned}
 \partial_t \overline{B}  
-\sigma\triangle \overline{B}
=& 
- \overline{u}    \nabla  \overline{B} 
+    \overline{B}    \nabla  \overline{u} 
+\overline{u} \mathcal{M}_1   \nabla  \overline{B} 
-    \overline{B}  \mathcal{M}_1  \nabla  \overline{u} 
\\
& +  \mathcal{M}_3 \nabla\overline{B}
+
\sigma
\mathcal{M}_4
: \nabla^2 \overline{B}
+
\sigma
\mathcal{M}_2  \cdot  (\nabla \overline{B})  ,
\end{aligned}
\end{equation}
\smallskip
\begin{equation}
\begin{aligned}
{\rm div}\overline{u} = \mathcal{M}_1 : \nabla\overline{u} , 
\end{aligned}
\end{equation}
\begin{equation}
\begin{aligned}
 - \jump{  \nu^\pm \prts{\nabla \overline{u} + \nabla \overline{u}^\top  } } n_{\Sigma} + \jump{\overline{p} } n_{\Sigma}  
- \kappa (\triangle_{\Sigma}h)n_{\Sigma} 
=
\prts {  \mathcal{G}_1 + \mathcal{G}_2 +\kappa {\rm tr} L_\Sigma^2 h } n_{\Sigma}
+ \mathcal{G}_3 ,
\end{aligned}
\end{equation}
\begin{equation}  \label{transformed.interface.speed.eq}
\partial_t h   -   \overline{u} \cdot n_{\Sigma} +  b\cdot\nabla_\Sigma h
= \prts{ I- \mathcal{M}_0 } \nabla_{\Sigma}h \cdot \overline{u}   + (b-\overline{u} )\nabla_{\Sigma}h ,
\end{equation}
\begin{equation}  \label{transformed.eq.100}
\begin{aligned}
\overline{u}|_{\partial\Omega}=0,  \ \overline{B}|_{\partial\Omega}=0, 	\quad
	\overline{u}|_{t=0}=\overline{u}_{0}:= u\circ\Theta , 
	\ \overline{B}|_{t=0}=\overline{B}_{0}:= B \circ\Theta,
\end{aligned}
\end{equation}
where the variable $b$ in \eqref{transformed.interface.speed.eq} is the auxiliary term same as in \cite{Pruss.quali}.

For convenience, we  omit the bars in $\overline{u}$, $\overline{B}$ and $\overline{p}$
in the rest of the paper when there is no confusion. 
The height function $h$ only occurs in the transformed equations,
 so it always appears without a bar.

\section{Linear Part} \label{Linear part}

In this section,
we consider the linear part of the equations \eqref{transformed.eq.1}-\eqref{transformed.eq.100},
which can be rewritten as the following linear problem:
\begin{align}
 \partial_t {u}  + \nabla p -\nu^\pm \triangle u =g_1   \quad & \text{in} \; \Omega\setminus\Sigma, \label{linear.eq.u.1}
\\
{\rm div} u = g_3   \quad & \text{in} \; \Omega\setminus\Sigma, 
 \\
 \jump{ -  \nu^\pm \prts{ \nabla u + \nabla u^\top
}  +pI   } n_{\Sigma}    
-\kappa (\triangle_{\Sigma}h) n_{\Sigma} = g_4  \quad & \text{on} \; \Sigma
\\
\jump{u}=0  \quad & \text{on} \; \Sigma
\\
u = 0  \quad & \text{on} \; \partial\Omega
\\
\partial_t h - u\cdot n_{\Sigma} + b\cdot \nabla_{\Sigma}h = g_5  \quad & \text{on} \; \Sigma
\\
u(0) =u_0 \quad & \text{in}  \; \Omega\setminus\Sigma, 
\\
h(0)= h_0  \quad & \text{on} \; \Sigma .  \label{linear.eq.u.100}
\\
\partial_t B   -\sigma\triangle B = g_2
\quad & \text{in}  \; \Omega, 
\label{linear.eq.B.1}
\\
B = 0  \quad & \text{on} \; \partial\Omega,
\\
B(0) =B_0 \quad & \text{in}  \; \Omega. 
\label{linear.eq.B.100}
\end{align}
The linear problem can be divided into two sub-problems: \eqref{linear.eq.u.1}-\eqref{linear.eq.u.100}  
and \eqref{linear.eq.B.1}-\eqref{linear.eq.B.100},
which can be solved using the theory of two-phase Stokes problems in \cite{Pruss.green.book, Pruss.quali} and  the classical theory of parabolic equations in \cite[Theorem 9.1]{Ladyzenskaja}. 
\begin{remark}
We remark that the condition  ${\rm div} B=0$ in the original problem can be removed in the transformed problem as discussed in Section \ref{Transformed equations}.
\end{remark}

The equations  \eqref{linear.eq.u.1}-\eqref{linear.eq.u.100}
can be solved using  \cite[Theorem 1]{Pruss.quali}. 
We restate the theorem for the case of equal density.
\begin{theorem}[\!\!\cite{Pruss.quali} Theorem 1]
 \label{u-problem,max,regularity}
Let $\Omega$ be a $C^3$ domain in $\mathbb{R}^d$. Let $T>0$ and $q>d+2$.
Let $\Sigma\subseteq\Omega$ be a closed $C^3$-hypersurface. 
Suppose 
\begin{equation}
b\in W^{1-\frac{1}{2q},q}([0,T];L^q(\Sigma))\cap L^q([0,T]; W^{2-\frac{1}{q},q}(\Sigma)).
\end{equation}
Let the source terms and initial values be as follows:
\begin{enumerate}
\item $g_1\in L^q([0,T]\times\Omega)$; 
\item  $u_0\in W^{2-\frac{2}{q},q}(\Omega\setminus\Sigma)$, $u_0=0$ on $\partial\Omega$; 
\item  $g_3\in L^q([0,T];W^{1,q}(\Omega\setminus\Sigma))$, ${\rm div} u_0= g_3(0)=0$; 
\item  $g_4\in  W^{\frac{1}{2}-\frac{1}{2q},q}([0,T];L^q(\Sigma)) \cap L^q([0,T];W^{1-\frac{1}{q},q}(\Sigma))$;
\item  $\mathcal{P}_\Sigma \jump{ \nu^\pm \prts{ \nabla u_0 + \nabla u_0^\top } } = \mathcal{P}_\Sigma g_4 (0);$ 
\item  $g_5 \in W^{1-\frac{1}{2q},q}([0,T];L^q(\Sigma))  \cap L^q([0,T];W^{2-\frac{1}{2q},q}(\Sigma)) $;
\item  $h_0 \in W^{3-\frac{2}{q},q}(\Sigma)$.
\end{enumerate}
Then there exists a unique solution $(u,p , h)$ to \eqref{linear.eq.u.1}-\eqref{linear.eq.u.100}, such that
\begin{enumerate}
\item  $u\in W^{1,q}([0,T];L^q(\Omega))\cap L^q([0,T];W^{2,q}(\Omega\setminus\Sigma))$;  
\item  $p\in L^q([0,T];\dot{W}^{1,q}(\Omega\setminus\Sigma))$;  
\item   $  \jump{p}\in W^{\frac{1}{2}-\frac{1}{2q},q}([0,T];L^q(\Sigma))\cap L^q([0,T];W^{1-\frac{1}{2q},q}(\Sigma))$;  
\item   $ h \in W^{2-\frac{1}{2q},q}([0,T];L^q(\Sigma))\cap W^{1,q}([0,T];W^{2-\frac{1}{q},q}(\Sigma))
  \cap L^q([0,T];W^{3-\frac{1}{q},q}(\Sigma)) $; 
\item   The mapping $(g_1, g_3, g_4, g_5 ,u_0, h_0, b) \mapsto (u,p,\jump{p},h)$ is continuous.
\end{enumerate}
\end{theorem}
For the principal part of the magnetic equations, we use the theory of
 parabolic equations in bounded domains, which can be found in e.g.
 \cite[Theorem 9.1]{Ladyzenskaja}. 
We also refer to   \cite[Chapter I, Section 1; and Chapter IV, Section 4]{Ladyzenskaja} and \cite[Chapter 8, Section 3]{Krylov}
for more details on the localization and flattening of bounded domains.
We will consider the case that $\partial\Omega$ is at least a $C^3$ surface.

\begin{theorem}[\!\!\cite{Ladyzenskaja} Chapter IV Theorem 9.1, simplified]
 \label{domain.parabolic,eq}
Let $\Omega$ be a $C^2$ domain. Let $T>0$ and $q>3/2$.
Suppose $g_2 \in L^q(  [0,T]\times\Omega  )$, $B_0 \in W^{2-\frac{2}{q},q}(\Omega)$ and $B_0|_{\partial \Omega}=0$.
Then there exists a unique solution 
\begin{equation}
B\in W^{1,q}([0,T];L^q(\Omega))\cap L^q([0,T];W^{2,q}(\Omega))
\end{equation}
to  \eqref{linear.eq.B.1}-\eqref{linear.eq.B.100}.
The solution has the estimate
\begin{equation}\label{B,parabolic,estimate}
 \norm{B} _{ W^{1,q}([0,T];L^q(\Omega)) }
+
\norm{B} _{L^q([0,T];W^{2,q}(\Omega))}
\leq  C(T) 
\prts{ \norm{g_2} _{L^q( [0,T]\times\Omega )}  + 
 \norm{B_0}_{W^{2-\frac{2}{q},q}(\Omega)}   },
\end{equation}
where the constant $C(T)$ is  bounded when  $T$ is finite.
\end{theorem}

Consequently, we can obtain a continuous solution operator defined as follows.
\begin{definition}
Let $T>0$ and $q>5$.
Given $ (u_0, B_0, h_0, b) $ such that 
\begin{enumerate}
\item 
$u_0\in W^{2-\frac{2}{q},q}(\Omega\setminus\Sigma)\cap C(\Omega)$, $u_0=0$ on $\partial\Omega$;
\item
$ B_0 \in W^{2-\frac{2}{q},q}(\Omega)$, $B_0=0$ on $\partial\Omega$;
\item
$h_0 \in W^{3-\frac{2}{q},q}(\Sigma)$;
\item
$b\in W^{1-\frac{1}{2q},q}([0,T];L^q(\Sigma))\cap L^q([0,T]; W^{2-\frac{1}{q},q}(\Sigma))$.
\end{enumerate}
For source terms $(g_1,g_2,g_3,g_4,g_5)$ such that
\begin{enumerate}
\item $g_1\in L^q([0,T]\times\Omega)$; 
\item  $g_2 \in L^q( [0,T]\times\Omega )$;
\item  $g_3\in L^q([0,T];W^{1,q}(\Omega\setminus\Sigma))$, \quad ${\rm div} u_0= g_3(0)=0$;
\item  $g_4\in  W^{\frac{1}{2}-\frac{1}{2q},q}([0,T];L^q(\Sigma)) \cap L^q([0,T];W^{1-\frac{1}{q},q}(\Sigma))$,  \quad 
  $\mathcal{P}_\Sigma g_4 (0) = \mathcal{P}_\Sigma \jump{ \nu^\pm \prts{\nabla u_0 +\nabla u_0^\top} } $; 
\item  $g_5 \in W^{1-\frac{1}{2q},q}([0,T];L^q(\Sigma))  \cap L^q([0,T];W^{2-\frac{1}{q},q}(\Sigma)) $,
\end{enumerate}
we define the solution operator $S_{(u_0,B_0,h_0,b)}$ (or simply $S$ if there is no confusion) by 
\begin{equation}
S_{ (u_0,B_0,h_0,b) }(g_1,g_2,g_3,g_4,g_5) := (u,B,p,\varpi,h),
\end{equation}
where $(u,p,\varpi,h)$ is the solution to 
\eqref{linear.eq.u.1} - \eqref{linear.eq.u.100} with $\varpi =\jump{p}$ being an auxiliary variable similar to \cite{Pruss.quali};
and $B$ is the solution to
\eqref{linear.eq.B.1} - \eqref{linear.eq.B.100}.
\end{definition}
From Theorem \ref{u-problem,max,regularity} and Theorem \ref{domain.parabolic,eq},
we know that $S_{(u_0,B_0,h_0,b)}$ is a continuous operator.
When the initial value vanishes, the solution operator becomes linear,
which implies that $S_{(0,0,0,b)}$ is a bounded linear operator.

\section{Nonlinear Part} \label{Nonlinear part}

In this  section, we estimate the terms on the right-hand side of \eqref{transformed.eq.1} - \eqref{transformed.eq.100} by calculating and estimating their Fr{\'e}chet derivatives.

\subsection{Equations and spaces}

For convenience, we define the linear parts and the nonlinear parts of the transformed equations by $L_i$ and $G_i$ similarly to \cite{Pruss.quali}.
The bars over variables are omitted for convenience.
The term $\nu^\pm$ is a piecewisely-constant function which equals to $\nu^+$ and $\nu^-$ inside  and  outside $\Sigma$, respectively.
We define
\begin{equation}
\begin{aligned}
L_1 :=
 \partial_t u  
+ \nabla p
-\nu^\pm \triangle u ,
\end{aligned}
\end{equation}
\begin{equation}
\begin{aligned}
L_2 :=
 \partial_t B  
-\sigma\triangle B , 
\end{aligned}
\end{equation}
\begin{equation}
\begin{aligned}
L_3 := {\rm div} u , 
\end{aligned}
\end{equation}
\begin{equation}
\begin{aligned}
L_4 := \jump{ - \nu^\pm \prts{\nabla u + \nabla u ^\top} } n_{\Sigma} + \varpi n_{\Sigma}  
- \kappa (\triangle_{\Sigma}h)n_{\Sigma}  , 
\end{aligned}
\end{equation}
\begin{equation}
\begin{aligned}
L_5 := \partial_t h  - u\cdot n_{\Sigma} + b\cdot \nabla_{\Sigma}h , 
\end{aligned}
\end{equation}
\begin{equation}
\begin{aligned}
G_1  : = 
&- \frac{1}{2} \nabla \prts{ \abs{ B }^2 }
-  u \nabla  u 
+    B   \nabla   B  
+\mathcal{M}_3 \nabla u 
+  u   \mathcal{M}_1   \nabla  u  
-  B   \mathcal{M}_1   \nabla  B  
\\
&
+ \frac{1}{2} \mathcal{M}_1  \nabla \prts{ \abs{ B }^2 }
+ \mathcal{M}_1  \nabla p 
+
\nu^\pm
\mathcal{M}_4
: \nabla^2  u
+
\nu^\pm
\mathcal{M}_2  \cdot  (\nabla u) , 
\end{aligned}
\end{equation}
\begin{equation}
\begin{aligned}
G_2 & := 
-  u  \nabla  B 
+    B    \nabla  u
+u \mathcal{M}_1   \nabla  B 
-    B  \mathcal{M}_1  \nabla  u 
+  \mathcal{M}_3 \nabla B
+
\sigma
\mathcal{M}_4
: \nabla^2 B
+
\sigma
\mathcal{M}_2  \cdot  (\nabla B) , 
\end{aligned}
\end{equation}
\begin{equation}
\begin{aligned}
G_3 &= 
\mathcal{M}_1 : \nabla u , 
\end{aligned}
\end{equation}
\begin{equation}
\begin{aligned}
G_4 &= 
\prts {  \mathcal{G}_1 + \mathcal{G}_2 +\kappa ({\rm tr} L_\Sigma^2)  h } n_{\Sigma}
+
\mathcal{G}_3  , 
\end{aligned}
\end{equation}
\begin{equation}
\begin{aligned}
G_5 &= 
\prts { \prts {I - \mathcal{M}_0 }\nabla_{\Sigma}h } \cdot u + (b-u) \cdot \nabla_{\Sigma}h  .
\end{aligned}
\end{equation}
For convenience, we rewrite the formulas of $\mathcal{M}_0$ to $\mathcal{M}_4$:
\begin{equation}  \label{diff,term,M0}
\mathcal{M}_0 =  (I- hL_{\Sigma})^{-1} , 
\end{equation} 
\begin{equation}  \label{diff,term,M1}
\mathcal{M}_1 =  (I+\nabla\theta)^{-1}\nabla\theta   , 
\end{equation} 
\begin{equation}  \label{diff,term,M2}
\mathcal{M}_2 =  \prts{\triangle  \Theta^{-1} } \circ \Theta , 
\end{equation} 
\begin{equation}  \label{diff,term,M3}
\mathcal{M}_3 =  \partial_t\theta \prts{ I-(I+\nabla\theta)^{-1} \nabla\theta } , 
\end{equation} 
\begin{equation}  \label{diff,term,M4}
\begin{aligned}
\mathcal{M}_4
=  \prts{\nabla \Theta}^{-\top}  \prts{\nabla \Theta}^{-1}   -I .   
\end{aligned}
\end{equation} 
To make the arguments concise, we abbreviate some common function spaces. Based on the settings in 
 \cite{Pruss.quali}, we define
\begin{align}
&  u\in \mathcal{W}_1^T := W^{1,q}([0,T]; L^q(\Omega)) \cap  L^q([0,T]; W^{2,q}(\Omega\setminus\Sigma)),
\\
& B \in \mathcal{W}_2^T :=  W^{1,q}([0,T]; L^q(\Omega)) \cap  L^q([0,T]; W^{2,q}(\Omega)),
\\
& p \in \mathcal{W}_3^T :=  L^q([0,T]; \dot{W}^{1,q}(\Omega\setminus\Sigma)),
\\
& \varpi \in \mathcal{W}_4^T := W^{\frac{1}{2} - \frac{1}{2q} ,q } ([0,T]; L^q(\Sigma))
\cap L^q ([0,T]; W^{ 1 - \frac{1}{q} ,q }(\Sigma)),
\\
& h \in \mathcal{W}_5^T := W^{2-\frac{1}{2q} ,q } ([0,T]; L^q(\Sigma)) 
\cap  W^{1,q} ([0,T]; W^{2-\frac{1}{q}, q}(\Sigma))
\cap  L^q ([0,T]; W^{ 3 - \frac{1}{q}, q}(\Sigma)).
\end{align}
For convenience, we  define
\begin{equation}
 \mathcal{W}_6^T := 
W^{\frac{1}{2},q}([0,T]; L^q(\Omega)) \cap  L^q([0,T]; W^{1,q}(\Omega\setminus\Sigma)),
\end{equation}
which is the space that $\nabla u$  and $\nabla B$ belong to.
Similarly to \cite{Pruss.quali}, we  define the solution space
\begin{equation}
\mathcal{W}^T : =  
\left\{ (u,B,p,\varpi,h) \in \mathcal{W}_1^T \times \mathcal{W}_2^T  \times \mathcal{W}_3^T  \times \mathcal{W}_4^T  \times \mathcal{W}_5^T : \jump{p} = \varpi \right\} 
\end{equation}
and denote by $\mathcal{S}_i$ the space that the source terms belong to, i.e.
\begin{align}
&  G_1 \in \mathcal{S}_1^T := L^{q}([0,T]; L^q(\Omega)), 
\\
& G_2 \in \mathcal{S}_2^T := L^{q}([0,T]; L^q(\Omega)),
\\
& G_3 \in \mathcal{S}_3^T := W^{1,q}([0,T]; \dot{W}^{-1,q}(\Omega)) \cap  L^q([0,T]; W^{1,q}(\Omega\setminus\Sigma)),
\\
& G_4 \in \mathcal{S}_4^T := W^{\frac{1}{2} - \frac{1}{2q} ,q } ([0,T]; L^q(\Sigma))
\cap  L^q ([0,T]; W^{ 1 - \frac{1}{q} ,q }(\Sigma)),
\\
& G_5 \in \mathcal{S}_5^T := W^{ 1 - \frac{1}{2q} ,q } ([0,T]; L^q(\Sigma))
\cap  L^q ([0,T]; W^{ 2 - \frac{1}{q} ,q }(\Sigma)).
\end{align}
Then we define
\begin{equation}
\mathcal{S}^T:=  \mathcal{S}_1^T \times \mathcal{S}_2^T  \times \mathcal{S}_3^T \times
  \mathcal{S}_4^T  \times \mathcal{S}_5^T .
\end{equation}
We will also frequently use
the following spaces:
\begin{equation}
\begin{aligned}
&\mathcal{C}_0^T:=C^0([0,T]; C^0(\Omega)),
\\
&\mathcal{C}_1^T:=C^0([0,T]; C^1(\Omega) )\cap C^1([0,T]; C^0(\Omega) ),
\\
&\mathcal{C}_2^T:=C^0([0,T]; C^2(\Omega) )\cap C^1([0,T]; C^1(\Omega) ).
\end{aligned}
\end{equation}
For convenience, we also define
\begin{equation}
\begin{aligned}
\mathcal{C}_3^T:=C^0([0,T]; C^1(\Omega)), \quad \text{and} \quad \mathcal{C}_4^T:=C^0([0,T]; C^2(\Omega)).
\end{aligned}
\end{equation}
The spaces $\mathcal{C}_i$, $i=0,1,2,3,4$,
 can also be defined on $[0,T]\times\Sigma$ in parallel ways, which can be expressed by replacing $\Omega$ with $\Sigma$. 
 We will omit the notation of domain (e.g. $\Omega$ or $\Sigma$) and the time variable $T$ when there is no confusion.
For all spaces $Z$ with the form $X([0,T];Y)$, we use
\begin{equation}
\mathring{Z}^T \quad \text{and}  \quad  \mathring{X}([0,T];Y)
\end{equation}
to denote the subspace whose elements all have initial value 0 when $t=0$ (in the sense of limit or trace). 
If $Z$ is the intersection of spaces, i.e.
\begin{equation}
Z^T = X_1([0,T];Y_1) \cap \cdots \cap X_k([0,T];Y_k)
\end{equation}
for some $k\in\mathbb{N}$,
then we use $\mathring{Z}^T$ to denote
\begin{equation}
\mathring{Z}^T = \mathring{X}_1([0,T];Y_1) \cap \cdots \cap \mathring{X}_k([0,T];Y_k).
\end{equation}

\subsection{Fr{\'e}chet derivatives of $\mathcal{M}_i$ and their estimates}

In this section, we assume  that $\Omega\subseteq\mathbb{R}^m$  for $m\geq 2$.
We will estimate the terms
 $\mathcal{M}_i$, $i=0,1,2,3,4$  in \eqref{diff,term,M0}-\eqref{diff,term,M4}, 
 which are introduced in \cite{Pruss.quali}.
 We will also calculate and estimate their Fr{\'e}chet derivatives $D\mathcal{M} _i$.
 We use  arguments in a similar spirit to those for flat reference surfaces in \cite{Pruss.analytic.solu}. 
 The case of closed reference surface has been discussed in \cite{Pruss.quali}. We include more details for completeness.
The estimates are studied on a generic time interval $[0,T]\subseteq [0,T_0]$ with $T_0>0$ being a fixed number. 
For convenience, we temporarily omit the parameter $T$ in
the notations of function spaces. 

\subsubsection{$\mathcal{M}_0$ and $D\mathcal{M}_0$}

We recall that $\mathcal{M}_0 := (I- hL_{\Sigma})^{-1} $ .
Since $\Sigma$ is a fixed  surface, its Weingarten tensor $L_{\Sigma}:= - \nabla_\Sigma n_\Sigma$ is a fixed, matrix-valued function on $\Sigma$.
The entries of $I-h L_{\Sigma}$ are 
\begin{equation} 
I-h L_{\Sigma}=
\begin{pmatrix}
&1- h l_{11} & - h l_{12}  & \cdots & - h l_{1m} \\
&- h l_{21} & 1 - h l_{22}  & \cdots & - h l_{2m} \\
&\vdots &  \vdots & \ddots & \vdots \\
&- h l_{m1} &  - h l_{m2}  & \cdots & 1 - h l_{mm} 
\end{pmatrix}.
\end{equation}
The determinant of $I-h L_{\Sigma}$  
is in the form of 
$1 + h P(h)$, where  $P(h)$ denotes a  polynomial of $h$. 
The entries of the adjugate matrix  ${\rm adj}( I-h L_{\Sigma}) $ can all be represented using polynomials $Q_{ij}(h)$.
From the  inverse matrix formula $A^{-1} = { \rm adj} (A)/ {\rm det}(A)$, we have 
\begin{equation}  \label{structure.M0}
\begin{aligned}
 \prts{I-hL_{\Sigma}}^{-1}_{ij} 
=
 \frac{Q_{ij}(h)}{1 + h P(h) } .
\end{aligned}
\end{equation}
Let  $\partial $ denote the derivative of either the time or space variables, we have 
\begin{equation} \label{derivative.M0.temp1}
\begin{aligned}
 \partial \prts{I-hL_{\Sigma}}^{-1}_{ij} 
=
 \frac{\partial \prts{ Q_{ij}(h) } \prts{ 1 + h P(h) } - \prts{ Q_{ij}(h) } \partial\prts{ 1 + h P(h) }  }
 {\prts{ 1 + h P(h) }^2 }
= 
 \frac{  \tilde{Q}_{ij}(h,\partial h) }  { 1 + h \tilde{P}(h) }  . 
\end{aligned}
\end{equation}
Let $\delta_0(\Sigma)>0$ be sufficiently small and assume without loss of generality that $\delta_0<1$, then for all $\norm{h}_{\mathcal{C}_0} <\delta_0$ we have
\begin{equation} 
\begin{aligned}
\norm{ 1 + h \tilde{P}(h) }_{\mathcal{C}_0} \geq \frac{1}{2},
\end{aligned}
\end{equation}
which implies a universal upper bound
\begin{equation}  
\begin{aligned}
  \norm{(I-hL_{\Sigma})^{-1}}_{\mathcal{C}_0}
\leq C(m,\Sigma) .
\end{aligned}
\end{equation}
Taking higher-order derivatives, we obtain
\begin{equation}  \label{derivative.M0.temp2}
\begin{aligned}
 \partial^k \prts{I-hL_{\Sigma}}^{-1}_{ij} 
= 
 \frac{  \overline{Q}_{ij}(h,\partial h,\cdots, \partial^k h) }  { 1 + h \overline{P}(h) }, 
\end{aligned}
\end{equation}
where $\overline{P}$ and $\overline{Q}_{ij}$ are also polynomials.
Similarly, in \eqref{derivative.M0.temp2}, we can stay away from the zeros of the denominator  by letting 
 $\norm{h}_{\mathcal{C}_0} <\delta_0$ for some $\delta_0(m,\Sigma,k)\in (0,1)$.
Moreover, let $M>0$ be an upper bound of the derivatives of $h$, then we can bound the numerator as well.
Specially, for $k=0,1,2$, there exists $\delta_0(m,\Sigma,k)\in(0,1)$
such that for all  $\norm{h}_{\mathcal{C}_0} <\delta_0$ and $\norm{h}_{\mathcal{C}_k} <M$
we have
\begin{equation}  \label{derivative.M0.temp3}
\begin{aligned}
  \norm{ \mathcal{M}_0 }_{\mathcal{C}_k  } =  \norm{(I-hL_{\Sigma})^{-1}}_{\mathcal{C}_k  } 
\leq C(m,\Sigma,M) . 
\end{aligned}
\end{equation}

Next, we estimate the Fr{\'e}chet derivative $D\mathcal{M}_0$.
We decompose $\mathcal{M}_0$ into $F : h\mapsto I-hL_{\Sigma} $ and $G: A\mapsto A^{-1}$, where $A$ is an invertible $m\times m$ matrix.
Then we can calculate their Fr{\'e}chet derivatives by the definition, which implies
\begin{equation}
DF\depvar{h}\varphi = - \varphi L_{\Sigma} \quad \text{and} \quad DG\depvar{A}H= -A^{-1} H A^{-1}
\end{equation}
for all $\varphi $ in the same space as for $h$; and all $H$ in the same space as for $A$.
The term $DG$ is obtained by considering 
\begin{equation}
(A+H)^{-1}-A^{-1}= A^{-1}\prts{( I+HA^{-1} )^{-1} -I }
\end{equation}
and expanding $( I+HA^{-1} )^{-1}$ using power series.
Thus, we have
\begin{equation} 
\begin{aligned}
D\mathcal{M}_0 \depvar{h} \varphi 
=  D(G \circ F)\depvar{h}  \varphi 
=  DG \depvar{F\depvar{h}} DF\depvar{h}  \varphi
\\
=   \prts{ I-hL_{\Sigma} }^{-1}  \prts{ \varphi L_{\Sigma} } \prts{ I-hL_{\Sigma} }^{-1}
= \mathcal{M}_0  L_{\Sigma} \mathcal{M}_0 \varphi.
\end{aligned}
\end{equation}
Therefore, from \eqref{derivative.M0.temp3} we know that 
 for $k=0,1,2$, there exists $\delta_0 (m,\Sigma,k)\in(0,1)$,
such that for all  $\norm{h}_{\mathcal{C}_0 } <\delta_0$, $\norm{h}_{\mathcal{C}_k  } <M$
and $\varphi\in \mathring{\mathcal{C}}_k $, we have
\begin{equation} 
\begin{aligned}
& \norm{ D\mathcal{M}_0 \depvar{h} \varphi }_{ \mathring{\mathcal{C}}_k  }
=
\norm{ \mathcal{M}_0  \varphi L_\Sigma \mathcal{M}_0}_{  \mathring{\mathcal{C}}_k  }
\\
& \leq
\norm{ \mathcal{M}_0 }_{\mathcal{C}_k  }^2 
\norm{ L_\Sigma  }_{C^k(\Sigma)}  \norm{\varphi } _{  \mathring{\mathcal{C}}_k}
\leq 
C(m,\Sigma,M) \norm{ \varphi} _{  \mathring{\mathcal{C}}_k } .
\end{aligned}
\end{equation}
When $k=1$, we also have for all $0\leq T\leq T_0$ that
\begin{equation} 
\begin{aligned}
& \norm{ D\mathcal{M}_0 \depvar{h} \varphi }_{  \mathring{\mathcal{S}}^T_4 }
\leq  C(m,T_0,\Sigma,q)
 \norm{ \mathcal{M}_0 }_{\mathcal{C}_1^T  } ^2 
\norm{ L_\Sigma  }_{C^1(\Sigma)}  \norm{\varphi } _{ \mathring{\mathcal{S}}^T_4 }
\\
& \leq 
C(m,T_0, \Sigma,q, M) \norm{ \varphi} _{ \mathring{\mathcal{S}}_4^T } .
\end{aligned}
\end{equation}

\subsubsection{ $\mathcal{M}_1$ and $D\mathcal{M}_1$ }

Now we estimate  $\mathcal{M}_1 = (I+\nabla\theta)^{-1}\nabla\theta $ in \eqref{diff,term,M1}. 
We recall that
\[
\theta_h(x) = \eta\prts{\frac{d(x)} {\varrho_0} } h(\Pi(x))\nu_{\Sigma}(\Pi(x))
\]  
when  $x\in B(\Sigma; \varrho_0)$;
and $\theta_h(x)=0$ when $x\notin B(\Sigma;\varrho_0)$.
Since $\Sigma$ is a  fixed surface, we can write $\theta$ (subscript $h$ omitted) as 
\begin{equation}  \label{def1,of,theta}
 \theta(x)=\mathbbm{h}(x)\mathbbm{n}(x), 
\end{equation}
where $\mathbbm{n} := \eta(d(x)/ \varrho_0 ) n_{\Sigma}(\Pi(x)) $ is an extension of the normal vector field $n_{\Sigma}$ to $\Omega$, which is supported in $B(\Sigma;\varrho_0)$; 
and $\mathbbm{h}$ is an extension of $h$ from $\Sigma$ to $\Omega$ by 
letting $\mathbbm{h}(x):= h(\Pi(x))$ for $x\in B(\Sigma;\varrho_0)$ and $\mathbbm{h}(x):= 0$ otherwise.
Thus, we have
\begin{equation} 
\begin{aligned}
\nabla\theta = \nabla(\mathbbm{h}  \mathbbm{n} )
=
\nabla \mathbbm{h}\otimes \mathbbm{n}  + \mathbbm{h}  \nabla\mathbbm{n},
\end{aligned}
\end{equation}
which implies that
\begin{equation} 
\begin{aligned}
& \norm{\nabla \theta}_{\mathcal{C}_1 }
:= \norm{\nabla \theta}_{C^1([0,T]; C(\Omega))} 
	+  \norm{\nabla \theta}_{C([0,T]; C^1(\Omega))} 
\\
&  =\norm{\nabla(\mathbbm{h}  \mathbbm{n} )}_{C^1([0,T]; C(B(\Sigma;\varrho_0)))} 
	+  \norm{\nabla(\mathbbm{h}  \mathbbm{n} )}_{C([0,T]; C^1(B(\Sigma;\varrho_0)))} 
\\
& \leq  C(\Sigma)
 \norm{\mathbbm{h} } _{ C^1([0,T]; C^1(B(\Sigma;\varrho_0))) }
 + C(\Sigma) \norm{ \mathbbm{h}  }_{C([0,T]; C^2(B(\Sigma;\varrho_0)))} 
\\
& \leq C(\Sigma) \norm{\mathbbm{h} } _{\mathcal{C}_2 }
\leq C(\Sigma) \norm{h} _{\mathcal{C}_2 } .
\end{aligned}
\end{equation}
Using similar arguments and the fact that $\Sigma$ is a $ C^3$ surface, we   obtain
\begin{equation} 
\begin{aligned}
& \norm{\nabla \theta}_{\mathcal{S}_5 }
= \norm{\nabla \mathbbm{h}\otimes \mathbbm{n}   }_{\mathcal{S}_5 }
	+ \norm{ \mathbbm{h}  \nabla\mathbbm{n}}_{\mathcal{S}_5 }
\\
& \leq  C(\Sigma)
\norm{\nabla \mathbbm{h} }_{\mathcal{S}_5 }
	+  C(\Sigma) \norm{ \mathbbm{h} }_{\mathcal{S}_5 }
\leq  C(\Sigma)
\norm{ \mathbbm{h} }_{\mathcal{W}_5 }
\leq   C(\Sigma)
\norm{ h }_{\mathcal{W}_5 } .
\end{aligned}
\end{equation}
Letting $F: \mathbbm{h} \to (I+\nabla\theta)^{-1}$, for $\varphi\in \mathring{\mathcal{W}}_5$
we have
\begin{equation} 
\begin{aligned}
DF\depvar{\mathbbm{h}} \varphi =
- \prts{I+ \nabla(\mathbbm{h}\mathbbm{n}  ) }^{-1}  
\nabla \prts{\varphi\mathbbm{n} } 
\prts{I+ \nabla(\mathbbm{h}\mathbbm{n}  ) }^{-1}
\end{aligned}
\end{equation}
The entries of $I+\nabla \prts{\mathbbm{h}\mathbbm{n} }$
are
\begin{equation} \label{M1,estimate,matrix1}
\begin{pmatrix}
& 1+ \partial_1 \mathbbm{h} \mathbbm{n}_1 +  \mathbbm{h} \partial_1\mathbbm{n}_1 & \partial_1 \mathbbm{h} \mathbbm{n}_2 + \mathbbm{h} \partial_1  \mathbbm{n}_2  & \cdots &  \partial_1 \mathbbm{h} \mathbbm{n}_m + \mathbbm{h} \partial_1  \mathbbm{n}_m \\
& \partial_2 \mathbbm{h} \mathbbm{n}_1 +  \mathbbm{h} \partial_2\mathbbm{n}_1 & 1+ \partial_2 \mathbbm{h} \mathbbm{n}_2 + \mathbbm{h} \partial_2  \mathbbm{n}_2  & \cdots &  \partial_2 \mathbbm{h} \mathbbm{n}_m + \mathbbm{h} \partial_2  \mathbbm{n}_m \\
& \vdots &  \vdots & \ddots & \vdots \\
& \partial_m \mathbbm{h} \mathbbm{n}_1 +  \mathbbm{h} \partial_m\mathbbm{n}_1 & \partial_m \mathbbm{h} \mathbbm{n}_2 + \mathbbm{h} \partial_m  \mathbbm{n}_2  & \cdots & 1+ \partial_m \mathbbm{h} \mathbbm{n}_m + \mathbbm{h} \partial_m  \mathbbm{n}_m 
\end{pmatrix}.
\end{equation}
Since $\nabla \mathbbm{n}$ only depends on $\Sigma$, it is a fixed matrix.
Thus, the determinant of $I+\nabla \prts{\mathbbm{h}\mathbbm{n} }$ 
is
\begin{equation} 
\begin{aligned}
\det \prts{ I+\nabla \prts{\mathbbm{h}\mathbbm{n} } }
=
1+ 	\eth \mathbbm{h} P(	\eth \mathbbm{h}).
\end{aligned}
\end{equation}
To simplify the statement, we   temporarily use the notation $\eth \mathbbm{h}$ when: 
\begin{enumerate}
\item this term is either $h$ or its derivatives;
\item in all cases, the term can be controlled using the  norm in the context.
\end{enumerate}
Without loss of generality, we also slightly abuse the notation $\eth \mathbbm{h} P(	\eth \mathbbm{h})$ to denote the sum of multiple   terms with the same structure, e.g.
\[
\mathbbm{h} (1+ \mathbbm{h}\partial_{x_i}\mathbbm{h} ) 
+ \partial_{x_i} \mathbbm{h} (1+\mathbbm{h}) .
\]
This notation does not bring any trouble as long as the term $\eth \mathbbm{h}$ can be controlled 
by the corresponding norms in the context.
The entries of the adjugate matrix of $I+\nabla(\mathbbm{h}\mathbbm{n})$ are all in the form of $P(\mathbbm{h} , \partial \mathbbm{h})$.
Thus, all entries of the matrix
$  \prts{ I+\nabla \prts{\mathbbm{h}\mathbbm{n} } }^{-1}$  can be expressed as
\begin{equation}  \label{M1,poly,structure}
\begin{aligned}
\frac{Q_{ij}(\eth \mathbbm{h})}{1+  \eth \mathbbm{h} \, P(\eth \mathbbm{h})}.
\end{aligned}
\end{equation}
Moreover, their time or space derivatives, i.e., 
\[
\partial_{t}\prts{ \frac{Q_{ij}(\eth \mathbbm{h})}{1+  \eth \mathbbm{h} \, P(\eth \mathbbm{h})}  }
\quad \text{or} \quad
\partial_{x_k} \prts{ \frac{Q_{ij}(\eth \mathbbm{h})}{1+  \eth \mathbbm{h} \, P(\eth \mathbbm{h})} },
\]
can also be expressed using the same formula as in \eqref{M1,poly,structure}.
We recall that the polynomial fractions in  \eqref{M1,poly,structure} only depend on $\Sigma$.
Thus, when $\norm{h}_{\mathcal{C}_3} \leq \delta_0$ for some sufficiently small $\delta_0 (\Sigma) $, we are able to exclude all singularities. Suppose we also have $\norm{h} _{\mathcal{C}_2}\leq M$ for some $M>0$,
then we have
\begin{equation} \label{M1,poly,structure2}
\begin{aligned}
\norm{  \prts{ I+\nabla \prts{\mathbbm{h}\mathbbm{n} } }^{-1} }_{\mathcal{C}_1} 
\leq C(m ,\Sigma, M).
\end{aligned}
\end{equation}
We remark that the $\mathcal{C}_3$ norm is enough in the condition $\norm{h}_{\mathcal{C}_3} \leq \delta_0$, since taking derivatives of polynomial fractions does not create higher-order derivatives in their denominators.

To get the formula of  $D\mathcal{M}_1$, 
we still need to calculate the Fr{\'e}chet derivative of $G: \mathbbm{h}\mapsto \nabla\theta$,
i.e., $DG\depvar{\mathbbm{h}}\varphi = \nabla(\varphi\mathbbm{n})$. 
Using the product rule of Fr{\'e}chet derivatives, we then obtain that 
\begin{equation} \label{M1,mathbbm,h}
\begin{aligned}
& D(FG) \depvar{\mathbbm{h}}\varphi
=
\prts{ (DF)\depvar{\mathbbm{h}}\varphi } \prts{ G\depvar{\mathbbm{h}} }
 +   \prts{ F\depvar{\mathbbm{h}} } \prts{ (DG)\depvar{\mathbbm{h}}\varphi } 
\\
& =
- \prts{I+ \nabla(\mathbbm{h}\mathbbm{n}  ) }^{-1}  
\nabla \prts{\varphi\mathbbm{n} } 
\prts{I+ \nabla(\mathbbm{h}\mathbbm{n}  ) }^{-1}
\nabla \prts{\mathbbm{h}\mathbbm{n} } 
+
\prts{I+ \nabla(\mathbbm{h}\mathbbm{n}  ) }^{-1}
\nabla \prts{\varphi\mathbbm{n} } 
\\
& =
 \prts{I+ \nabla(\mathbbm{h}\mathbbm{n}  ) }^{-1}  
\nabla \prts{\varphi\mathbbm{n} } 
\prts{I+ \nabla(\mathbbm{h}\mathbbm{n}  ) }^{-1}
\nabla \prts{\mathbbm{h}\mathbbm{n} } ,
\end{aligned}
\end{equation}
where the last equality is obtained using $-(I+\nabla\theta)^{-1}\nabla\theta +I = (I+\nabla\theta)^{-1}\nabla\theta$.
Since there is no singularity in $\nabla ( \mathbbm{h} \mathbbm{n})$,
we have  
\begin{equation} \label{M1,poly,structure3}
\begin{aligned}
 \norm{  \nabla (\mathbbm{h} \mathbbm{n})   }_{\mathcal{C}_1}
\leq C(m,\Sigma)\norm{\mathbbm{h}}_{\mathcal{C}_2}.
\end{aligned}
\end{equation}
The term $\nabla (\varphi \mathbbm{n})$ can be treated in the same way since 
$\mathbbm{h}$ and $\varphi$ are in the same function space.
Finally,
letting $H: h\mapsto \mathbbm{h}$, 
we have $D \mathcal{M}_1 =  D_h \mathcal{M}_1 = D((FG)\circ H)$ 
and
\begin{equation} 
\begin{aligned}
& D((FG)\circ H)\depvar{h}\psi
=  D(FG) \depvar{H\depvar{h}} DH\depvar{h}\psi
\\
& =
\prts{ DF\depvar{H\depvar{h}} DH\depvar{h}\psi} \prts{ G\depvar{H\depvar{h}} }
 +   \prts{ F\depvar{H\depvar{h}} } \prts{ DG\depvar{H\depvar{h}}   DH\depvar{h}\psi} 
\\
& =  
- \prts{I+ \nabla(\mathbbm{h}\mathbbm{n}  ) }^{-1}  
\nabla \prts{(\psi\circ\Pi)\mathbbm{n} } 
\prts{I+ \nabla(\mathbbm{h}\mathbbm{n}  ) }^{-1}
\nabla \prts{\mathbbm{h}\mathbbm{n} } 
+
\prts{I+ \nabla(\mathbbm{h}\mathbbm{n}  ) }^{-1}
\nabla \prts{(\psi\circ\Pi)\mathbbm{n} } 
\end{aligned}
\end{equation}
for all $\psi\in\mathring{\mathcal{C}}_2$.
Consequently, we obtain the following estimates for $ \mathcal{M}_1 $ and $ D \mathcal{M}_1 $.
\begin{proposition}
There exists $\delta_0(m,\Sigma) \in(0,1)$, such that for all $\norm{h}_{\mathcal{C}_3}<\delta_0$,
if $\norm{h}_{\mathcal{C}_1} < M$ for some $M>0$ then 
\begin{equation}  \label{M1.estimate2}
\begin{aligned}
\norm{ \mathcal{M}_1 \depvar{h} } _{ \mathcal{C}_1 }
\leq C(m,\Sigma,M) \norm{h}_{\mathcal{C}_2};
\end{aligned}
\end{equation}
and for all $\varphi\in { \mathring{\mathcal{C}}_2 } $ we have
\begin{equation}  \label{M1.estimate1}
\begin{aligned}
\norm{ D \mathcal{M}_1 \depvar{h} \varphi } _{ \mathring{\mathcal{C}}_1 }
\leq C(m,\Sigma,M)
\norm{ \varphi } _{ \mathring{\mathcal{C}}_2 }.
\end{aligned}
\end{equation}

\end{proposition}

\begin{remark}
For the operator $H:h\mapsto\mathbbm{h}$, 
its Fr{\'e}chet derivative is 
$DH\depvar{h}: \varphi \mapsto \varphi \circ \Pi$.
Notice that $H$ is a linear operator. 
\end{remark}

\subsubsection {$\mathcal{M}_2$ and $D\mathcal{M}_2$ }

Now we study $\mathcal{M}_2$ in \eqref{diff,term,M2}.
Letting $\mathbbm{i}: x\to x$ be the identity mapping in $\Omega$, 
we have $\Theta=\mathbbm{i}+\mathbbm{h}\mathbbm{n}$ and
$
\nabla\Theta = I+ \nabla\theta
= I+ \nabla(\mathbbm{h}\mathbbm{n})
$.
Using the inverse function theorem, we have
\begin{equation} 
\begin{aligned}
 \triangle  \Theta^{-1}_k 
&=
\sum_{i}\partial_i (\nabla\Theta^{-1}_k)_i
=
\sum_{i}\partial_i \prts{ (\nabla\Theta)^{-1}_{ik} \circ \Theta^{-1}  }
\\
& =
\sum_{i}\sum_{j} \partial_i \Theta^{-1}_j \prts{ \partial_j (\nabla\Theta)^{-1}_{ik} } \circ  \Theta^{-1}
\\
& =
\sum_{i}\sum_{j}  \prts{ (\nabla\Theta)^{-1}_{ij} \circ  \Theta^{-1} }
\prts{ \prts{ \partial_j (\nabla\Theta)^{-1}_{ik} } \circ  \Theta^{-1} } , 
\end{aligned}
\end{equation}
which implies that
\begin{equation}  \label{M2,rewrite}
\begin{aligned}
\prts{\mathcal{M}_2 }_k = 
\prts{\triangle  \Theta^{-1}_k }\circ\Theta
=
\sum_{i}\sum_{j}  (\nabla\Theta)^{-1}_{ij}
 \prts{ \partial_j (\nabla\Theta)^{-1}_{ik} }.
\end{aligned}
\end{equation}
Here we abbreviated $\prts{\prts{\nabla\Theta}^{-1} }_{ij}$
to $\prts{\nabla\Theta}^{-1}_{ij}$ for convenience.
In order to estimate $\prts{\triangle  \Theta^{-1} }\circ\Theta $, 
we only need to estimate its entries $\prts{\triangle  \Theta^{-1}_k }\circ\Theta$,
which only requires the estimate of 
$ (\nabla\Theta)^{-1}_{ij}
 \prts{ \partial_j (\nabla\Theta)^{-1}_{ik} }$.
We only need to estimate one entry, i.e., for a fixed choice of $(i,j,k)$, 
since all entries have the same structure.

Letting $F: \mathbbm{h} \mapsto   (\nabla\Theta)^{-1}_{ij}  $
and $G: \mathbbm{h} \mapsto  \partial_j (\nabla\Theta)^{-1}_{ik}  $.
Similar to the arguments for  $\mathcal{M}_1$, we obtain
\begin{equation} 
\begin{aligned}
DF\depvar{\mathbbm{h}}\varphi 
=
- \prts{
\prts{ \nabla(\mathbbm{i}+\mathbbm{h}\mathbbm{n})  }^{-1}
\nabla( \varphi\mathbbm{n})
\prts{ \nabla(\mathbbm{i}+\mathbbm{h}\mathbbm{n})  }^{-1}
 } _{ij},
\end{aligned}
\end{equation}
\begin{equation} 
\begin{aligned}
DG\depvar{\mathbbm{h}}\varphi 
=
-\partial_j \prts{
\prts{ \nabla(\mathbbm{i}+\mathbbm{h}\mathbbm{n})  }^{-1}
\nabla(\varphi\mathbbm{n})
\prts{ \nabla(\mathbbm{i}+\mathbbm{h}\mathbbm{n})  }^{-1}
} _{ik}.
\end{aligned}
\end{equation}
Notice that $\nabla(\mathbbm{i}+\mathbbm{h}\mathbbm{n}) = I+ \nabla(\mathbbm{h}\mathbbm{n}) $.
Then we have
\begin{equation}  \label{M2,one,entry}
\begin{aligned}
 D(FG)\depvar{\mathbbm{h}}\varphi 
&=
\prts{ DF\depvar{\mathbbm{h}}\varphi } \prts{ G\depvar{\mathbbm{h}} }
+  \prts{ F\depvar{\mathbbm{h}} } \prts{ DG\depvar{\mathbbm{h}}\varphi } 
\\
& =  
-  \prts{
\prts{ I+\nabla\prts{ \mathbbm{h}\mathbbm{n} }  }^{-1}
\nabla( \varphi\mathbbm{n})
\prts{ I+\nabla\prts{ \mathbbm{h}\mathbbm{n} } }^{-1}
} _{ij} 
\prts{ \partial_j  \prts{ I+\nabla\prts{ \mathbbm{h}\mathbbm{n} } }_{ik} } 
\\
& \quad-
\prts{  \prts{ I+\nabla\prts{ \mathbbm{h}\mathbbm{n} } }^{-1} } _{ij}
\partial_j \prts{
\prts{ I+\nabla\prts{ \mathbbm{h}\mathbbm{n} }  }^{-1}
\nabla( \varphi\mathbbm{n})
\prts{ I+\nabla\prts{ \mathbbm{h}\mathbbm{n} } }^{-1}
} _{ik} 
\\
& 
=: I_1 + I_2.
\end{aligned}
\end{equation}
In term $I_1$, we have
\begin{equation} 
\begin{aligned}
 \partial_j  \prts{ I+\nabla\prts{ \mathbbm{h}\mathbbm{n} } }_{ik} 
=
  \partial^2_{ji} \mathbbm{h} \mathbbm{n}_k 
 +   \partial_i \mathbbm{h}   \partial_j \mathbbm{n}_k
 +    \partial_j \mathbbm{h}   \partial_i \mathbbm{n}_k
 +   \mathbbm{h}  \partial^2_{ji} \mathbbm{n}_k,
\end{aligned}
\end{equation}
which implies
\begin{equation} 
\begin{aligned}
 \norm{ \partial_j  \prts{ I+\nabla\prts{ \mathbbm{h}\mathbbm{n} } }_{ik} }_{\mathcal{C}_0}
\leq
C(m,\Sigma) \norm{h}_{\mathcal{C}_4}.
\end{aligned}
\end{equation}
When $\norm{h}_{\mathcal{C}_3} < \delta_0 $ for some sufficiently small $\delta_0(\Sigma)$, 
using the same argument as in the derivation of \eqref{M1,poly,structure2},
we have
\begin{equation} \label{part,of,M1.estimate.1}
\begin{aligned}
\norm{\prts{I+\nabla\prts{ \mathbbm{h}\mathbbm{n} }  }^{-1} }_{\mathcal{C}_0}
\leq C(m,\Sigma).
\end{aligned}
\end{equation} 
We remark that conditions like $\norm{h}_{\mathcal{C}_1}<M$ are not required in the current estimation.
Using \eqref{part,of,M1.estimate.1}, we obtain that
\begin{equation} 
\begin{aligned}
& \norm{   
\prts{
\prts{ I+\nabla\prts{ \mathbbm{h}\mathbbm{n} }  }^{-1}
\nabla( \varphi\mathbbm{n})
\prts{ I+\nabla\prts{ \mathbbm{h}\mathbbm{n} } }^{-1}
} _{ik}
}_{ \mathring{\mathcal{C}}_0  }  
\\
& \lesssim
\norm{   
\prts{ I+\nabla\prts{ \mathbbm{h}\mathbbm{n} }  }^{-1}
\nabla( \varphi\mathbbm{n})
\prts{ I+\nabla\prts{ \mathbbm{h}\mathbbm{n} } }^{-1}
}_{ \mathring{\mathcal{C}}_0  }  
\\
& \lesssim
\norm{   
\prts{ I+\nabla\prts{ \mathbbm{h}\mathbbm{n} }  }^{-1}
}_{\mathcal{C}_0}^2
 \norm{   
\nabla( \varphi\mathbbm{n})
}_{ \mathring{\mathcal{C}}_0  }  
\\
& \leq C(m, \Sigma )
 \norm{ \varphi }_{ \mathring{\mathcal{C}}_3  }  
\end{aligned}
\end{equation}
for all $\varphi \in \mathring{\mathcal{C}}_3$,
which completes the estimate
\begin{equation} \label{M2,I1} 
\begin{aligned}
\norm{I_1}_{\mathcal{C}_0}
\leq C(m, \Sigma, M)
 \norm{\varphi}_{ \mathring{\mathcal{C}}_3 } .
\end{aligned}
\end{equation}
In order to estimate $I_2$, we first obtain that
\begin{equation}
 \norm { I + \nabla(\mathbbm{h}\mathbbm{n})  }_{\mathcal{C}_0} 
 \leq C(m,\Sigma) \prts{ 1 + \norm{h}_{\mathcal{C}_3} }.
\end{equation}
To estimate the  term 
$\partial_j \prts{
\prts{ I + \nabla(\mathbbm{h}\mathbbm{n})  }^{-1}
\nabla( \varphi\mathbbm{n})
\prts{ I + \nabla(\mathbbm{h}\mathbbm{n})  }^{-1}
} _{ik}
$,
we first notice that for three matrices $A$, $B$ and $C$ we have
\begin{equation}
\begin{aligned}
& \norm {\partial_j (ABC)_{ik} }_{\mathcal{C}_0}
= \norm{ (\partial_jABC)_{ik}+ (A\partial_jBC)_{ik}+ (AB\partial_jC)_{ik} }_{\mathcal{C}_0}
\\
& \lesssim 
\norm{ \partial_j A}_{\mathcal{C}_0} \norm{B}_{\mathcal{C}_0}\norm{C} _{\mathcal{C}_0}
+ \norm{  A} _{\mathcal{C}_0} \norm{\partial_jB}_{\mathcal{C}_0} \norm{C} _{\mathcal{C}_0}
+ \norm{  A}_{\mathcal{C}_0} \norm{B}_{\mathcal{C}_0} \norm{\partial_j C}_{\mathcal{C}_0},
\end{aligned}
\end{equation}
i.e., the entry of the product of matrices can be estimated using the matrix norm.
\begin{remark}
We remind the reader that for vector-valued or matrix-valued functions in $\mathcal{C}_k$,
the norm is defined by  first taking the $\mathcal{C}_k$ norm of components, and then taking the matrix or vector norm.
\end{remark}
 We recall that we already have 
\begin{equation} 
\begin{aligned}
\norm{\nabla( \varphi\mathbbm{n})}_{ \mathring{\mathcal{C}}_3 } 
 \leq C(\Sigma) \norm{\varphi}_{ \mathring{\mathcal{C}}_4 } ,
\end{aligned}
\end{equation}
\begin{equation} 
\begin{aligned}
\norm{\partial_j \prts{ \nabla( \varphi\mathbbm{n}) }   }_{ \mathring{\mathcal{C}}_0 } 
 \leq C(\Sigma) \norm{\varphi}_{ \mathring{\mathcal{C}}_4 } .
\end{aligned}
\end{equation}
 It remains to estimate the matrix norm of 
 $  \partial_j \prts{ I + \nabla(\mathbbm{h}\mathbbm{n})  }^{-1} $,
which can be obtained using the same argument as in the derivation of \eqref{M1,poly,structure2}.
In fact, 
 there exists $\delta_0(\Sigma) \in (0,1)$, 
such that for all $\norm{h}_{\mathcal{C}_3}< \delta_0$,
if $\norm{h}_{\mathcal{C}_4}< M$ for some $M>0$, then we have
\begin{equation} \label{M2,I2,temp1}
\begin{aligned}
\norm{ \partial_j \prts{ \prts{ I + \nabla(\mathbbm{h}\mathbbm{n}) }^{-1} } }_{\mathcal{C}_0}
\leq C(m,\Sigma, M),
\end{aligned}
\end{equation}
which completes the estimate
\begin{equation}
\norm{ I_2  }_{\mathring{\mathcal{C}}_0}
\leq C(m,\Sigma,M)\norm{ \varphi  }_{\mathring{\mathcal{C}}_4}.
\end{equation}
Consequently, we have the following estimates.
\begin{proposition}
There exists $\delta_0(\Sigma) \in(0,1)$, 
such that for all $\norm{h}_{\mathcal{C}_3}< \delta_0$,
if  $\norm{h}_{\mathcal{C}_4}< M$ for some $M>0$, then
\begin{equation} 
\begin{aligned}
\norm{ \mathcal{M}_2}_{ \mathcal{C}_0}
\leq
C(m,\Sigma,M) ;
\end{aligned}
\end{equation}
and for all $\varphi\in \mathring{\mathcal{C}}_4 $ we have
\begin{equation} 
\begin{aligned}
\norm{ D\mathcal{M}_2\depvar{h} \varphi }_{\mathring{\mathcal{C}}_0} 
\leq C(m,\Sigma,M) \norm{ \varphi  }_{\mathring{\mathcal{C}}_4}.
\end{aligned}
\end{equation}

\end{proposition}

\subsubsection{ $\mathcal{M}_3$ and $D\mathcal{M}_3$ }

The term $\mathcal{M}_3 =  \partial_t\theta \prts{ I-(I+\nabla\theta)^{-1} \nabla\theta }$ 
can be written as
\begin{equation}
\mathcal{M}_3= \partial_t \prts{ \mathbbm{h} \mathbbm{n} } (I-\mathcal{M}_1).
\end{equation}
Letting $F: \mathbbm{h} \mapsto I-\mathcal{M}_1 $ 
and  $G: \mathbbm{h} \mapsto \partial_t \prts{ \mathbbm{h} \mathbbm{n} }$.
Using the same calculation as in \eqref{M1,mathbbm,h}, 
we have for all $\varphi\in\mathring{\mathcal{C}}_1$ that
\begin{equation} 
\begin{aligned}
DF \depvar{\mathbbm{h}}\varphi
=  -  \prts{I+ \nabla(\mathbbm{h}\mathbbm{n}  ) }^{-1}  
\nabla \prts{\varphi\mathbbm{n} } 
\prts{I+ \nabla(\mathbbm{h}\mathbbm{n}  ) }^{-1}  . 
\end{aligned}
\end{equation}
Since $\partial_t \theta = \partial_t (\mathbbm{h}\mathbbm{n}) = \partial_t \mathbbm{h}\mathbbm{n}$, 
we have
\begin{equation} 
\begin{aligned}
DG \depvar{\mathbbm{h}}\varphi
= \partial_t \varphi \mathbbm{n}.
\end{aligned}
\end{equation}
Thus, the Fr{\'e}chet derivative of  $\mathcal{M}_3$ is 
\begin{equation} \label{M3,product,rule}
\begin{aligned}
 D(GF)\depvar{\mathbbm{h}} \varphi 
&=  \prts{ G \depvar{\mathbbm{h}} } \prts{ DF\depvar{\mathbbm{h}} \varphi } 
	+ \prts{ DG \depvar{\mathbbm{h}} \varphi }  \prts{ F \depvar{\mathbbm{h}} } 
\\
& =
 \prts{
 \partial_t \mathbbm{h} \mathbbm{n}
 }
\prts{ 
- \prts{I+ \nabla(\mathbbm{h}\mathbbm{n}  ) }^{-1}  
\nabla \prts{\varphi\mathbbm{n} } 
\prts{I+ \nabla(\mathbbm{h}\mathbbm{n}  ) }^{-1}
 }
 \\
& \quad +
  \prts{
  \partial_t \varphi \mathbbm{n}
  }
  \prts{
  I- \prts{ I+\nabla(\mathbbm{h}\mathbbm{n}) }^{-1} \nabla(\mathbbm{h}\mathbbm{n}) 
  }.
\end{aligned}
\end{equation}
Composing with the mapping $h\mapsto \mathbbm{h}=h\circ\Pi$, we have the following estimates.
\begin{proposition}
There exists $\delta_0(\Sigma)\in (0,1)$, such that for all $\norm{h}_{\mathcal{C}_3} < \delta_0 $, 
if $\norm{h}_{\mathcal{C}_1} < M $ for some $M>0$, then
\begin{equation} \label{M3.estimate,0}
\begin{aligned}
\norm {  \mathcal{M}_3  \depvar{ h }  }_{ \mathcal{C}_0 }
\leq C(m, \Sigma,M) \norm{h}_{ \mathcal{C}_1} ;
\end{aligned}
\end{equation}
and for all $\varphi \in \mathring{\mathcal{C}}_1 $,
\begin{equation} \label{M3.estimate}
\begin{aligned}
\norm { D \mathcal{M}_3  \depvar{ h }   \varphi  }_{\mathring{\mathcal{C}}_0}
\leq C(m, \Sigma,M) \norm{\varphi}_{\mathring{\mathcal{C}}_1} .
\end{aligned}
\end{equation}

\end{proposition}

\subsubsection{  $\mathcal{M}_4$ and $D\mathcal{M}_4$ }

In order to estimate $\mathcal{M}_4 = \prts{\nabla \Theta}^{-\top}  \prts{\nabla \Theta}^{-1}   -I$,
we first rewrite it as
\begin{equation}
\begin{aligned}
\mathcal{M}_4 &= 
\prts{\nabla \Theta}^{-\top} \prts{ I - \nabla \Theta^\top  \nabla \Theta  } \prts{\nabla \Theta}^{-1}  
\\
&=  - 
\prts{\nabla \Theta}^{-\top} 
\prts{  \nabla \theta^\top  \nabla \theta +  \nabla \theta^\top +  \nabla \theta} 
\prts{\nabla \Theta}^{-1}  ,
\end{aligned}
\end{equation}
which allows us to separate the variable $h$.

Next, we study the Fr{\'e}chet derivative.
We define
\begin{equation}
\begin{aligned}
F : A \mapsto A^{\top} A  - I ,
\end{aligned}
\end{equation}
where $A$ is a matrix.
Since
\begin{equation}
\begin{aligned}
F(A+H) -F(A) = A^{\top}H + H^{\top} A +   H^{\top} H,
\end{aligned}
\end{equation}
we have
\begin{equation}
\begin{aligned}
DF\depvar{A}H = A^{\top}H + H^{\top} A.
\end{aligned}
\end{equation}
We recall that for the inverse matrix operator $G: A\mapsto A^{-1} $ we have
\begin{equation}
\begin{aligned}
DG\depvar{A}H = - A^{-1} H A^{-1}.
\end{aligned}
\end{equation}
Thus, we have
\begin{equation}
\begin{aligned}
D (F \circ G)\depvar{A} H = DF\depvar{G\depvar{A}} DG\depvar{A}H
= -A^{-\top}\prts{ A^{-1}H + \prts{A^{-1}H}^{\top} } A^{-1}.
\end{aligned}
\end{equation}
For the operator $K: \mathbbm{h} \mapsto \nabla\Theta= I+\nabla\theta$,
we have $DK\depvar{ \mathbbm{h} }\varphi = \nabla \prts{ \varphi \mathbbm{h}  }$.
This implies that
\begin{equation}
\begin{aligned}
& D (F \circ G \circ K)\depvar{ \mathbbm{h} } \varphi 
=
 D(F\circ G) \depvar{ K(\mathbbm{h})  } DK \depvar{\mathbbm{h}} \varphi
\\
& = -\prts{ I+\nabla\prts{\mathbbm{h}\mathbbm{n}} }^{-\top}
\prts{ \prts{ I+\nabla\prts{\mathbbm{h}\mathbbm{n}} }^{-1}  \nabla\prts{\varphi\mathbbm{n}}  
+ \prts{\prts{ I+\nabla\prts{\mathbbm{h}\mathbbm{n}} }^{-1}    \nabla\prts{\varphi\mathbbm{n}}     }^{\top} } \prts{ I+\nabla\prts{\mathbbm{h}\mathbbm{n}} }^{-1}.
\end{aligned}
\end{equation}
Composing $F \circ G \circ K$ with the mapping $h\mapsto \mathbbm{h}$ and using the same argument as in the estimate of $\mathcal{M}_1$, we obtain the following estimate.
\begin{proposition}
There exists $\delta_0(\Sigma)\in (0,1)$, such that for all $\norm{h}_{\mathcal{C}_3} < \delta_0 $, 
\begin{equation}
\begin{aligned}
\norm { \mathcal{M}_4  \depvar{ h }  }_{ \mathcal{C}_0 }
\leq C(m, \Sigma) 
\norm { h }_{ \mathcal{C}_3 };
\end{aligned}
\end{equation}
and for all $\varphi \in \mathring{\mathcal{C}}_3 $,
\begin{equation}
\begin{aligned}
\norm { D \mathcal{M}_4  \depvar{ h } \varphi }_{\mathring{\mathcal{C}}_0}
\leq C(m, \Sigma) 
\norm {  \varphi }_{ \mathring{\mathcal{C}}_3 }.
\end{aligned}
\end{equation}

\end{proposition}

Therefore, we have obtained the estimates of terms $\mathcal{M}_0$ to $\mathcal{M}_4$.
In the remaining part of this paper,  
we will fix a value for $\delta_0$ such that all these estimates can be satisfied,
which can  be done by taking the minimum of finite many values. For more details on these terms, we refer to \cite{Pruss.quali, Pruss.analytic.solu}.

\begin{remark}
Notice that  the variable we consider is $h$ when we calculate the Fr{\'e}chet derivatives, rather than the point $x\in\Omega$. 
Thus, the function $\Theta(x)= x+ \mathbbm{h}(x)\mathbbm{n}(x)$ is understood as a mapping 
$h\mapsto \mathbbm{h} \mapsto \mathbbm{i} + \mathbbm{h}\mathbbm{n} $, where $\mathbbm{i}$ 
and $\mathbbm{n}$ are fixed functions defined in $\Omega$. The Fr{\'e}chet derivative is the derivative of an operator, so it is taken with respect to $h$ rather than $x$.
\end{remark}

\subsection{Fr{\'e}chet derivatives of $\mathcal{G}_i$ and their estimates} \label{section,G123,estimate}

In this section, we estimate nonlinear terms
 $\mathcal{G}_i$, $i=1,2,3$ in the surface tension equation, 
which are defined in 
\eqref{stress,eq,rewrite,temp2}
\eqref{def,G2} and \eqref{def,G3}. 
These terms have been discussed in e.g. \cite{Pruss.quali}. We also refer readers to \cite{Pruss.analytic.solu} for details of the case of flat reference surface.
In our calculation, we add more details for completeness.
We remind the readers that we still use $\mathbb{R}^m$ instead of $\mathbb{R}^3$ in order to have clearer structures.
For convenience, we do not mark the dependency on $m$ for constant terms.

\begin{equation}
\begin{aligned}
 \mathcal{G}_1 :=
&-
\prts{  \jump{ \nu^\pm  \prts{ \nabla u + \nabla u^\top
} } \mathcal{M}_0 \nabla_{\Sigma}h} \cdot  n_{\Sigma}
\\
&  -
\prts{  \jump{  \nu^\pm  \prts{ \mathcal{M}_1\nabla u +  (\mathcal{M}_1\nabla u ) ^\top
} } \prts{ n_{\Sigma} - \mathcal{M}_0 \nabla_{\Sigma}h }  } \cdot  n_{\Sigma}
\end{aligned}
\end{equation}
\begin{equation}
 \mathcal{G}_2:= \kappa  H_{\Gamma} - \kappa DH_{\Gamma}\depvar{0}h ,
\end{equation}
\begin{equation} 
\begin{aligned}
  \mathcal{G}_3 
 &=
\mathcal{P}_{\Sigma}
\jump{
\nu^\pm  (I-\mathcal{M}_1) \nabla u + \nu^\pm  \prts{ (I-\mathcal{M}_1) \nabla u }^{\top}  
} \mathcal{M}_0 \nabla_{\Sigma}h
\\
& \quad+
\mathcal{P}_{\Sigma}
\jump{
\nu^\pm \mathcal{M}_1 \nabla u + \nu^\pm  \prts{ \mathcal{M}_1 \nabla u }^{\top}  
} n_{\Sigma}
\\
& \quad -
\prts{
\prts{
\jump{
\nu^\pm  (I-\mathcal{M}_1) \nabla u + \nu^\pm  \prts{ (I-\mathcal{M}_1) \nabla u }^{\top}  
} ( n_{\Sigma} -  \mathcal{M}_0 \nabla_{\Sigma}h ) 
} \cdot n_{\Sigma}
}
\mathcal{M}_0 \nabla_{\Sigma}h.
\end{aligned}
\end{equation}

\subsubsection{Estimate of $\alpha$} \label{section.estimate.alpha}

We recall that in \eqref{sufficient,condition;solve,alpha} we defined the auxiliary  vector $\alpha:= \mathcal{M}_0 \nabla_{\Sigma}h$ to help  calculate $n_{\Gamma}$. 
Suppose that $\norm{h}_{\mathcal{C}_2} <M$. Since $\Sigma$ is fixed, from \eqref{derivative.M0.temp3}, there exists $\delta_0(\Sigma)\in(0,1) $ such that
\begin{equation} \label{surface,alpha,estimate1}
\begin{aligned}
\norm{\alpha}_{\mathcal{C}_1} 
\leq 
C \norm{\mathcal{M}_0}_{\mathcal{C}_1} 
\norm{\nabla_\Sigma h}_{\mathcal{C}_1} 
\leq 
C(\Sigma,M) \norm{h}_{\mathcal{C}_2} .
\end{aligned}
\end{equation}
Moreover, we have
\begin{equation} \label{surface,alpha,estimate1.1}
\begin{aligned}
\norm{\alpha}_{\mathcal{S}_5} 
\leq 
C(T_0,\Sigma) \norm{\mathcal{M}_0}_{\mathcal{C}_2} 
\norm{\nabla_\Sigma h}_{\mathcal{S}_5} 
\leq 
C(T_0,\Sigma,M) \norm{h}_{\mathcal{W}_5} .
\end{aligned}
\end{equation}

Now we estimate the Fr{\'e}chet derivative  $D\alpha$.
Given $z=(u,B,p,\varpi,h)\in \mathcal{W}$ 
and $$\varphi=(\varphi_u,\varphi_B,\varphi_p,\varphi_{\varpi},\varphi_h) \in \mathring{\mathcal{W}},$$
we have
\begin{equation} \label{Frechet.alpha}
\begin{aligned}
 \prts{D\alpha \depvar{z}} \varphi
&= D_h \depvar{h} \prts{ \mathcal{M}_0 \nabla_{\Sigma}h }  \varphi_h
\\
& = \prts{ D_h \mathcal{M}_0 \varphi_h } \prts{\nabla_{\Sigma}h} 
 +  \prts{  \mathcal{M}_0  } \prts{ D_h \prts{\nabla_{\Sigma}h} \varphi_h }  
 \\
& = \prts{  \mathcal{M}_0 \varphi_h L_{\Sigma} \mathcal{M}_0 }  \prts{\nabla_{\Sigma}h} 
  + \mathcal{M}_0 \prts{\nabla_{\Sigma} \varphi_h} .
\end{aligned}
\end{equation}
Here we use $D_h$ to express that the derivative is taken only with respect to $h$. 
Suppose that $\norm{h}_{\mathcal{C}_2} <M$. 
From \eqref{derivative.M0.temp3}, there exists $\delta_0(\Sigma)\in(0,1) $ such that
if $\norm{h}_{\mathcal{C}_0}<\delta_0$, then
\begin{equation} \label{surface,alpha,estimate2}
\begin{aligned}
 \norm{D\alpha \varphi } _{\mathring{\mathcal{C}}_1}
&=
\norm{ \prts{  \mathcal{M}_0 \varphi_h L_{\Sigma} \mathcal{M}_0 }  \prts{\nabla_{\Sigma}h} 
  + \mathcal{M}_0 \prts{\nabla_{\Sigma} \varphi_h} } _{\mathring{\mathcal{C}}_1}
\\
& \lesssim
\norm{\mathcal{M}_0 }^2 _{\mathcal{C}_1} 
\norm{\nabla_{\Sigma}h} _{\mathcal{C}_1}
\norm{\varphi_h} _{\mathring{\mathcal{C}}_1}
+
\norm{\mathcal{M}_0 } _{\mathcal{C}_1} 
\norm{\nabla_{\Sigma} \varphi_h }  _{ \mathring{\mathcal{C}}_1 }
\\
& \leq
C(\Sigma,M) \prts{1+ \norm{h}_{\mathcal{C}_2} }  \norm{\varphi_h} _{ \mathring{\mathcal{C}}_2 } .
\end{aligned}
\end{equation}
Moreover, we have
\begin{equation} \label{surface,alpha,estimate2.1}
\begin{aligned}
 \norm{D\alpha \varphi } _{\mathring{\mathcal{S}}_5}
&=
\norm{ \prts{  \mathcal{M}_0 \varphi_h L_{\Sigma} \mathcal{M}_0 }  \prts{\nabla_{\Sigma}h} 
  + \mathcal{M}_0 \prts{\nabla_{\Sigma} \varphi_h} } _{\mathring{\mathcal{S}}_5}
\\
& \leq
C(T_0,\Sigma)\norm{\mathcal{M}_0 }^2 _{\mathcal{C}_2} 
\norm{\nabla_{\Sigma}h} _{\mathcal{S}_5}
\norm{\varphi_h} _{\mathring{\mathcal{C}}_2}
+
C(T_0,\Sigma)
\norm{\mathcal{M}_0 } _{\mathcal{C}_2} 
\norm{\nabla_{\Sigma} \varphi_h }  _{ \mathring{\mathcal{S}}_5 }
\\
& \leq
C(T_0, \Sigma,M) 
\prts{ 1 + \norm{h} _{\mathcal{W}_5} }
\norm{\varphi_h} _{\mathring{\mathcal{W}}_5} ,
\end{aligned}
\end{equation}
where  we used the embedding theory in \cite[Proposition 5.1]{Pruss.analytic.solu} in the last inequality.
\begin{remark}
In order to make calculations concise, we will frequently use notations like 
$D \prts{\nabla_{\Sigma}h}$ instead of  using $DF$ and  $F(h):=\nabla_{\Sigma}h$. 
The simplified notation denotes the Fr\'{e}chet derivative of the mapping $h\mapsto \nabla_{\Sigma}h$ evaluated at $h$.
This is similar to notations like $(x^2)'$ or $d(x^2)$ in calculus.
Moreover, we will not distinguish $D$ and $D_h$  when there is no confusion.
\end{remark}

\subsubsection{Estimate of $\mathcal{G}_1 $}
We still assume that $\norm{h}_{\mathcal{C}_4}<\delta_0$ for a sufficiently small $\delta_0$
and $\norm{h}_{\mathcal{C}_2}\leq M$.
We rewrite $\mathcal{G}_1$ as
\begin{equation}
\begin{aligned}
 \mathcal{G}_1 &=
-
\prts{  \jump{ \nu^\pm  \prts{ \nabla u + \nabla u^\top
} } \mathcal{M}_0 \nabla_{\Sigma}h} \cdot  n_{\Sigma}
\\
&\quad -
\prts{  \jump{  \nu^\pm  \prts{ \mathcal{M}_1\nabla u +  (\mathcal{M}_1\nabla u ) ^\top
} }   
\prts{  n_{\Sigma}
 - \mathcal{M}_0 \nabla_{\Sigma}h 
 } } \cdot  n_{\Sigma}
\\
&
=: I_1 +I_2.
\end{aligned}
\end{equation}
In the term $I_1$, the operator $F: u\mapsto \jump{ \nu^\pm  \prts{ \nabla u + \nabla u^\top
} }$ is linear, which implies
\begin{equation}
\begin{aligned}
DF\depvar{u} \varphi
=
\jump{  \nu^\pm \prts{ \nabla \varphi_u + \nabla \varphi_u ^\top
} }
\end{aligned}
\end{equation}
for all 
$\varphi=(\varphi_u, \varphi_B,  \varphi_p,  \varphi_\varpi, \varphi_h)\in \mathring{\mathcal{W}}$.
Thus, the Fr{\'e}chet derivative $D I_1 $ at $z=(u,B,p,\varpi,h)$ dependents only on $u$ and $h$,
i.e.
\begin{equation}
\begin{aligned}
  D I_1  \varphi
&= - \prts{
\prts{ DF \varphi_u } \mathcal{M}_0 \nabla_{\Sigma}h
	+ F\depvar{u}   D \prts{ \mathcal{M}_0 \nabla_{\Sigma}h  } \varphi_h 
	} \cdot n_{\Sigma}
\\
&  = - \prts{
\prts{ DF \varphi_u } \alpha
	+ F\depvar{u}   (D  \alpha   \varphi_h ) 
	} \cdot n_{\Sigma} . 
\end{aligned}
\end{equation}
This implies the following estimate:
\begin{equation} \label{G1.I1.derivative.estimate}
\begin{aligned}
  \norm{ D I_1 \varphi }_{\mathring{\mathcal{S}}_4^T } 
&\leq C(T_0,\Sigma)
	\norm{ DF \varphi_u }_{ \mathring{\mathcal{S}}_4 ^T  } 
	\prts{  \norm{ \alpha }_{\mathcal{S}_4^T} + \norm{ \alpha }_{\infty} }   
	+
	C(T_0,\Sigma) \norm{ F\depvar{u} }_{ \mathcal{S}_4^T }  
	\norm{  D \alpha \varphi_h }_{ \mathring{\mathcal{C}}_1^T  }
 \\
&  \leq
 C(T_0,\Sigma)	\norm{  \varphi_u }_{ \mathring{\mathcal{W}}_1^T   } 
 \prts{ \norm{ h }_{\mathcal{W}_5^T} + \norm{ h }_{\infty}    } 
	+
	C(T_0,\Sigma) \norm{ u }_{ \mathcal{W}_1^T } 
	\norm{  \varphi_h }_{ \mathring{\mathcal{C}}_2 ^T  } 
\\
&  \leq C(\Sigma, T_0) 	\prts{  \norm{z}_{\mathcal{W}^T} + 	\norm{h}_{\infty}  }
\norm{\varphi}_{\mathring{\mathcal{W}}^T },
\end{aligned}
\end{equation}
where the constants are fixed for all $T\in (0,T_0]$.

To estimate $I_2$, we define
$G: z\mapsto \mathcal{M}_1 \nabla u   $, whose derivative is
\begin{equation} 
\begin{aligned}
 D G \depvar{z}\varphi 
 =
 \prts{ D \mathcal{M}_1 \varphi_h  } \nabla u +  \mathcal{M}_1  \nabla \varphi_u .
\end{aligned}
\end{equation}
Then we obtain
\begin{equation}  \label{G1.I2.derivative.estimate}
\begin{aligned}
   D I_2 \depvar{z}\varphi 
& =
 - \prts{ \jump{ \nu^\pm \prts{ D G \varphi_h  + \prts{ D G \varphi_h}^\top } }
 \prts{ n_{\Sigma}  - \mathcal{M}_0 \nabla_{\Sigma}h  }  }
 \cdot n_{\Sigma}
\\
&  \quad  + \prts{
 \jump{ \nu^\pm \prts{ G  + G^\top } }
 \prts{ D\prts{ \mathcal{M}_0 \nabla_{\Sigma}h } \varphi_h }
	} \cdot n_{\Sigma}
\\
&  =
 - \prts{ \jump{ \nu^\pm \prts{ D G \varphi_h  + \prts{ D G \varphi_h}^\top } }
 \prts{ n_{\Sigma}  - \alpha  }  }
 \cdot n_{\Sigma}\\
 &\quad + \prts{
 \jump{ \nu^\pm \prts{ G  + G^\top } }
 \prts{ D \alpha  \varphi_h }
	} \cdot n_{\Sigma}.
\end{aligned}
\end{equation}
Thus, using the estimates of $\alpha$ and $D\alpha$, we have
for all $T\in(0,T_0]$ that
\begin{equation}
\begin{aligned}
&  \norm { D I_2 \depvar{z}\varphi }_{\mathring{\mathcal{S}}_4^T}      
  (\Sigma,M,T_0) \norm{ \jump{ \nu^\pm \prts{ DG\varphi_h  + \prts{ DG\varphi_h}^\top } }
   }_{\mathring{\mathcal{S}}_4} 
   \norm { \prts{ n_{\Sigma}  - \alpha  } } _{\mathcal{C}_1} 
\\
&\quad 
 + (\Sigma,M,T_0) \norm{
 \jump{ \nu^\pm \prts{ G  + G^\top } }
	} _{\mathcal{S}_4}
	\norm{  \prts{ D \alpha  \varphi_h } }_{\mathring{\mathcal{C}}_1}
\\
&  \leq
  C(\Sigma,M,T_0) \prts{1+ \norm{h}_{\mathcal{C}_2} } 
  \norm{DG\varphi_h}_{\mathring{\mathcal{W}}_6} 
 +C(\Sigma,M,T_0) \norm{G}_{\mathcal{W}_6}
	\norm{\varphi_h}_{\mathring{\mathcal{C}}_2}
\\
&  \leq
  C(\Sigma,M,T_0) \prts{1+ \norm{h}_{\mathcal{C}_2} }  
  \prts{
  \norm{ \prts{ D \mathcal{M}_1\varphi_h} \nabla u }_{\mathring{\mathcal{W}}_6}
  +  \norm{ \mathcal{M}_1 \nabla \varphi_u }_{\mathring{\mathcal{W}}_6}
  }
 \\
& \quad  +  C(\Sigma,M,T_0)
  \norm{ \mathcal{M}_1 \nabla u }_{ \mathcal{W}_6 } \norm{\varphi_h}_{ \mathring{\mathcal{C}}_2 }
\\
&  \leq
  C(\Sigma,M,T_0)  \norm{D\mathcal{M}_1\varphi_h}_{\mathring{\mathcal{C}}_1}
   \norm{ \nabla u   }_{\mathcal{W}_6} 
 + C(\Sigma,M,T_0) \prts{ \norm{\mathcal{M}_1} _{\mathcal{W}_6} +  \norm{\mathcal{M}_1} _{\infty} }
	\norm{   \nabla \varphi_u } _{\mathring{\mathcal{W}}_6}
\\
& \quad +  C(\Sigma,M,T_0)  \norm{\mathcal{M}_1}_{\mathcal{C}_1} \norm{  \nabla u }_{ \mathcal{W}_6 } \norm{\varphi_h}_{\mathring{\mathcal{C}}_2}
\\
&  \leq
  C(\Sigma,M,T_0)  \norm{\varphi_h}_{\mathring{\mathcal{C}}_2}
   \norm{ \nabla u   }_{\mathcal{W}_6} 
 + C(\Sigma,M,T_0) \prts{ \norm{\nabla_\Sigma h} _{\mathcal{W}_6}  + \norm{\nabla_\Sigma h} _{\infty} }
	\norm{   \nabla \varphi_u } _{\mathring{\mathcal{W}}_6}
\\
& \quad + C(\Sigma,M,T_0)  \norm{h}_{\mathcal{C}_2} \norm{  \nabla u }_{ \mathcal{W}_6 } 
\norm{\varphi_h}_{\mathring{\mathcal{C}}_2}
\\
&  \leq
  C(\Sigma, T_0,M )  \norm{\varphi_h}_{\mathring{\mathcal{W}}_5}
   \norm{u}_{\mathcal{W}_1} 
 + C(\Sigma,M,T_0) \prts{ \norm{h} _{\mathcal{W}_5} + \norm{\nabla_\Sigma h} _{\infty} }
	\norm{\varphi_u} _{\mathring{\mathcal{W}}_1}
\\
& \quad +  C(\Sigma, T_0,M )  \norm{h}_{\mathcal{C}_2} \norm{ u }_{ \mathcal{W}_1 } 
\norm{\varphi_h}_{\mathring{\mathcal{W}}_5}
\\
&  \leq
C(\Sigma, T_0,M )
\prts{ \norm{z}_{ \mathcal{W}^T }  +  \norm{ \nabla_\Sigma h} _{ \infty } } \norm{\varphi}_{\mathring{\mathcal{W}}^T} ,
\end{aligned}
\end{equation}
where the fourth inequality is guaranteed by   \eqref{M1.estimate2} and \eqref{M1.estimate1}. 
For convenience, we have omitted the parameter $T$ in some notations of function spaces.

From  \eqref{G1.I1.derivative.estimate} and \eqref{G1.I2.derivative.estimate}, we finally obtain
\begin{equation}  \label{G1.derivative.estimate}
\begin{aligned}
\norm{ D \mathcal{G}_1 \depvar{z}\varphi }_{\mathring{\mathcal{S}}_4^T}
 \leq
 C(\Sigma,M, T_0) \prts{ \norm{z}_{\mathcal{W}^T} + \norm{h}_{\mathcal{C}_4^T}  } \norm{\varphi}_{ \mathring{\mathcal{W}}^T }.
\end{aligned}
\end{equation}

\subsubsection{Estimate of $\mathcal{G}_2 $} \label{section: estimate G2}
As a preparation for the estimate of $\mathcal{G}_2$, we first calculate the derivative 
of the mean curvature $H_{\Gamma}$
in \eqref{curvature,formula}:
\begin{equation} \label{curvature,formula2}
\begin{aligned}
&
H_{\Gamma}
=
\beta
{\rm tr} \prts{ \mathcal{M}_0 \prts{ L_{\Sigma} + \nabla_{\Sigma}\alpha } } 
-
\beta^3 \prts{  \mathcal{M}_0  \alpha    } 
\nabla_{\Sigma}\alpha  \alpha^\top.
\end{aligned}
\end{equation}
which has been discussed in \cite{Pruss.quali,Pruss.green.book}.
We include detailed calculations for completeness.

For the function $F(s):=1/s$, we have
\begin{equation}
\begin{aligned}
DF\depvar{s}r = - \frac{r}{s^2}.
\end{aligned}
\end{equation}
For the Euclidean norm  $G(x):=\abs{x}=\sqrt{x_1^2+\cdots+x_n^2}$,
its Fr{\'e}chet derivative is
\begin{equation} \label{norm,function,Frechet}
\begin{aligned}
DG\depvar{x}v= \frac{x\cdot v}{\abs{x}}.
\end{aligned}
\end{equation}
Letting $H(h):=n_\Sigma-\mathcal{M}_0 \nabla_{\Sigma}h$, 
we have for all $\varphi\in\mathring{\mathcal{W}}_5^T$ that
\begin{equation} \label{alpha,vector,Frechet}
\begin{aligned}
DH \depvar{h}\varphi
= - \prts{ D\mathcal{M}_0   \depvar{h} \varphi  } \prts{\nabla_{\Sigma}h }
- \prts{\mathcal{M}_0   \depvar{h}}  \prts{\nabla_{\Sigma} \varphi }.
\end{aligned}
\end{equation}
Since $\beta(h)=F\circ G \circ H (h)$, we have
\begin{equation}
\begin{aligned}
  D\beta \depvar{h} \varphi
&=  DF\depvar{G(H(h))} DG\depvar{H(h)} DH \depvar{h} \varphi 
\\
&  = 
\frac
{  \prts{n_\Sigma - \mathcal{M}_0 \nabla_{\Sigma}h } \cdot \prts{ \prts{ D\mathcal{M}_0\depvar{h}\varphi} \prts{ \nabla_{\Sigma}h }  + \mathcal{M}_0\depvar{h}  \nabla_{\Sigma} \varphi  } }
{ \abs { n_\Sigma - \mathcal{M}_0 \nabla_{\Sigma}h  }^3   }
\\
&  = \beta^3 \prts{ n_\Sigma - \alpha } \cdot \prts{ \prts{ D\mathcal{M}_0 \varphi} \prts{ \nabla_{\Sigma}h }  + \mathcal{M}_0  \nabla_{\Sigma} \varphi  }, 
\end{aligned}
\end{equation}
which implies 
\begin{equation}
\begin{aligned}
D(\beta^3) \depvar{h} \varphi
= 3\beta^2 \prts{ D\beta \depvar{h} \varphi }
= 3 \beta^5 \prts{ n_\Sigma - \alpha } \cdot \prts{ \prts{ D\mathcal{M}_0 \varphi} \prts{ \nabla_{\Sigma}h }  + \mathcal{M}_0  \nabla_{\Sigma} \varphi  }.
\end{aligned}
\end{equation}
Now we calculate the derivative of 
$E: h \mapsto {\rm tr} \prts{ \mathcal{M}_0 \prts{ L_{\Sigma} + \nabla_{\Sigma}\alpha } } $.
Letting $F: h\mapsto L_{\Sigma} +\nabla_{\Sigma} \alpha $, 
we have
\begin{equation} \label{surface,grad,alpha,Frechet}
\begin{aligned}
DF \depvar{h} \varphi
= \nabla_{\Sigma} 
\bigl(   \prts{ D\mathcal{M}_0   \depvar{h} \varphi  } \prts{\nabla_{\Sigma}h }
+ \prts{\mathcal{M}_0   \depvar{h}}  \prts{\nabla_{\Sigma} \varphi }   
\bigr).
\end{aligned}
\end{equation}
Letting $G: h\mapsto  \mathcal{M}_0 F $,
we have
\begin{equation}
\begin{aligned}
  DG \depvar{h} \varphi
&= 
\prts{ D\mathcal{M}_0 \depvar{h} \varphi } \prts{ F\depvar{h} } 
+  \prts{ \mathcal{M}_0 \depvar{h}  }  \prts{ DF\depvar{h} \varphi } 
\\
&  =  \prts{ D\mathcal{M}_0 \varphi } \prts{ L_{\Sigma} +\nabla_{\Sigma} \alpha  }
+ \mathcal{M}_0 \prts{  \nabla_{\Sigma} \prts{  \prts{D\mathcal{M}_0 \varphi } \prts{ \nabla_{\Sigma} h }  
    +   \mathcal{M}_0  \nabla_{\Sigma}\varphi }   } .
\end{aligned}
\end{equation}
Letting $H: A\mapsto tr( A)$, we have
\begin{equation}
\begin{aligned}
  DE \depvar{h} \varphi
&=  DH \depvar{G(h) }  DG\depvar{h} \varphi
\\
&  = {\rm tr} 
\prts{
\prts{ D\mathcal{M}_0 \varphi } \prts{ L_{\Sigma} +\nabla_{\Sigma} \alpha  }
+ \mathcal{M}_0 \prts{  \nabla_{\Sigma} \prts{  \prts{D\mathcal{M}_0 \varphi } \prts{ \nabla_{\Sigma} h }  
    +   \mathcal{M}_0  \nabla_{\Sigma}\varphi }   } 
}.
\end{aligned}
\end{equation}
Using \eqref{alpha,vector,Frechet}, we have
\begin{equation}
\begin{aligned}
  D\prts{ \mathcal{M}_0 \alpha } \depvar{h} \varphi
&=  D\prts{ \mathcal{M}_0 \mathcal{M}_0 \nabla_{\Sigma}h } \depvar{h} \varphi
\\
&  = \prts{ D \mathcal{M}_0  \varphi  } \prts{ \mathcal{M}_0  \nabla_{\Sigma}h }
  + \prts{ \mathcal{M}_0 } \prts{ D\prts{ \mathcal{M}_0  \nabla_{\Sigma}h } \varphi }.
\end{aligned}
\end{equation}
Using \eqref{alpha,vector,Frechet} and \eqref{surface,grad,alpha,Frechet}, we obtain
\begin{equation}
\begin{aligned}
  D\prts{ \nabla_{\Sigma} \alpha \alpha^{\top} } \depvar{h} \varphi
&=  \prts{  D\prts{\nabla_{\Sigma} \alpha}\depvar{h} \varphi }  \prts{ \alpha^{\top} \depvar{h} } 
  + \prts{ \nabla_{\Sigma} \alpha  \depvar{h} }  \prts{ \prts{D \prts{ \alpha^{\top} } } \depvar{h} \varphi}
\\
&  =   
\prts{ \nabla_{\Sigma} 
\prts{   \prts{ D\mathcal{M}_0   \depvar{h} \varphi  } \prts{\nabla_{\Sigma}h }
+ \prts{\mathcal{M}_0   \depvar{h}}  \prts{\nabla_{\Sigma} \varphi }   }  
}
\prts{ \alpha^{\top}  \depvar{h} }
\\
&\quad  +  \prts{ \nabla_{\Sigma} \alpha  \depvar{h} } 
\prts{
 \prts{ D\mathcal{M}_0   \depvar{h} \varphi  } \prts{\nabla_{\Sigma}h }
+ \prts{\mathcal{M}_0   \depvar{h}}  \prts{\nabla_{\Sigma} \varphi }
}^{\top}.
\end{aligned}
\end{equation}
Thus, we have
\begin{equation} \label{derivative,mean,curvature,H}
\begin{aligned}
 D H_{\Gamma} \depvar{h}\varphi
=&
\prts{\prts{D\beta }\varphi}  {\rm tr}\prts{ \mathcal{M}_0 \prts{ L_{\Sigma} +\nabla_{\Sigma}\alpha  }}
+  \prts{\beta } \prts{ D \prts{ {\rm tr}\prts{ \mathcal{M}_0 \prts{ L_{\Sigma} +\nabla_{\Sigma}\alpha  }}}\varphi}
\\
&-  \prts{ D\prts{\beta^3 } \varphi } \prts{\mathcal{M}_0\alpha } \prts{ \nabla_{\Sigma}\alpha \alpha^{\top} }
\\
&-  \prts{\beta^3 } \prts{D \prts{ \mathcal{M}_0\alpha} \varphi} \prts{ \nabla_{\Sigma}\alpha \alpha^{\top} }
-  \prts{\beta^3 } \prts{\mathcal{M}_0\alpha } \prts{ D\prts{\nabla_{\Sigma}\alpha \alpha^{\top} } \varphi}
\\
= & \  \beta^3 \prts{ n_\Sigma-\alpha } \prts{ \prts{ D\mathcal{M}_0 \varphi } \prts{\nabla_{\Sigma}h} + \mathcal{M}_0 \nabla_{\Sigma}\varphi_h }
\prts{
{\rm tr} 
\prts{
\prts{ \mathcal{M}_0 } \prts{ L_{\Sigma} +\nabla_{\Sigma} \alpha  }
}
}
\\
&+
	\beta
	\prts{
	{\rm tr} 
	\prts{
	\prts{ D\mathcal{M}_0 \varphi } \prts{ L_{\Sigma} +\nabla_{\Sigma} \alpha  }
	+ \mathcal{M}_0
	 \prts{  \nabla_{\Sigma} \prts{  \prts{D\mathcal{M}_0 \varphi } \prts{ \nabla_{\Sigma} h }  
	    +   \mathcal{M}_0  \nabla_{\Sigma}\varphi_h }   } 
	}
	}
\\	
&- 
	3\beta^5 \prts{ n_\Sigma -\alpha } \cdot 
	\prts{ \prts{ D\mathcal{M}_0 \varphi } \prts{ \nabla_{\Sigma} h}  
	+ \mathcal{M}_0  \nabla_{\Sigma} \varphi_h 
	}
	\prts{ \mathcal{M}_0 \alpha  }
	\prts{ \nabla_{\Sigma} \alpha \alpha^{\top} }
\\
&-
	\beta^3 
	\prts{ 
	\prts{ D\mathcal{M}_0 \varphi } \prts{ \mathcal{M}_0 \nabla_{\Sigma}h }
	+   \prts{ \mathcal{M}_0 } \prts{ D \prts{ \mathcal{M}_0  \nabla_{\Sigma}h } \varphi }
	}
	\prts{ \nabla_{\Sigma} \alpha \alpha^{\top} }
\\
&-
\beta^3 
\mathcal{M}_0 \alpha
\prts{
	\nabla_{\Sigma} 
	  \prts{
	  	  \prts{ D \mathcal{M}_0 \varphi } \nabla_{\Sigma}h 
	  	  + \mathcal{M}_0 \nabla_{\Sigma}\varphi_h
	} \alpha^{\top} 
	+
	 \nabla_{\Sigma} \alpha 
	\prts{
		\prts{D \mathcal{M}_0 \varphi } \nabla_{\Sigma}h 
		+  \mathcal{M}_0 \nabla_{\Sigma}\varphi_h
	} ^{\top}
}
\\
 =&: I_1 + I_2 -I_3 -I_4 -I_5.
\end{aligned}
\end{equation}
When $h=0$, we have $\mathcal{M}_0=I $ and $\nabla_{\Sigma}h =0$, which implies $\alpha=0 $ and $\beta=1 / \norm{n_\Sigma} =1$. Thus, the term $I_1$ turns to
\begin{equation}
\begin{aligned}
I_1 \depvar{0} = n_\Sigma \prts{ 0+ \nabla_{\Sigma}\varphi} {\rm tr}L_{\Sigma}
= 0
\end{aligned}
\end{equation}
since $n_\Sigma$ is perpendicular to $\nabla_{\Sigma}\varphi$.
The term $I_2$ turns to
\begin{equation}
\begin{aligned}
I_2\depvar{0} = {\rm tr} \prts{ \varphi L_{\Sigma} \prts{ L_{\Sigma} +0 } + \nabla_{\Sigma} \prts{ 0+ \nabla_{\Sigma}\varphi } I}
= \prts{ {\rm tr}L_{\Sigma}^2 + \triangle_{\Sigma} }\varphi.
\end{aligned}
\end{equation}
By definition, we have $\alpha\depvar{h}=0$ when $h=0$. Thus, we obtain 
\begin{equation}
\begin{aligned}
I_3\depvar{0} =
I_4\depvar{0} = I_5\depvar{0} =0
\end{aligned}
\end{equation}
with no need for calculation.
Therefore, same as in \cite{Pruss.quali}, we have 
\begin{equation} \label{derivative,curvature,at,0}
\begin{aligned}
DH_{\Gamma}\depvar{0}= {\rm tr}L_{\Sigma}^2 + \triangle_{\Sigma} .
\end{aligned}
\end{equation}
Since the mapping $F: h\mapsto DH_{\Gamma}\depvar{0}h$ is linear, its derivative is 
$DF\depvar{h} = DH_{\Gamma}\depvar{0}  $ for all $h$.
Thus, the Fr{\'e}chet derivative of  $\mathcal{G}_2$ is
\begin{equation} \label{estimate.G2.5terms}
\begin{aligned}
D\mathcal{G}_2\depvar{h}\varphi
=
D( H_{\Gamma}\depvar{h} - DH_{\Gamma}\depvar{0}h ) \varphi
=
I_1 + I_2 -I_3 -I_4 -I_5 - \prts{ {\rm tr}L_{\Sigma}^2 + \triangle_{\Sigma} }\varphi,
\end{aligned}
\end{equation}
with $I_1$ to $I_5$ defined in \eqref{derivative,mean,curvature,H}.
It remains to estimate  $\norm{I_1}$, $\norm{I_3}$, $\norm{I_4}$, $\norm{I_5}$ and the difference
$\norm{I_2 - \prts{ {\rm tr}L_{\Sigma}^2 + \triangle_{\Sigma} }\varphi }$
term by term.

For $I_1$, we first notice that the terms with structures like
$n_\Sigma\cdot \nabla_{\Sigma}f$ or $n_\Sigma\cdot\mathcal{M}_0 \nabla_{\Sigma}f$
will vanish. This is because the tangential gradient $\nabla_{\Sigma}$ of a function is a tangent vector, which is perpendicular to the normal vector $n_\Sigma$; moreover, the matrix $\mathcal{M}_0$ maps tangent vectors to tangent vectors.
Thus, we obtain
\begin{equation}
\begin{aligned}
&   \prts{ n_\Sigma-\alpha } \prts{ \prts{ D\mathcal{M}_0 \varphi } \prts{\nabla_{\Sigma}h} + \mathcal{M}_0 \nabla_{\Sigma}\varphi }
\\
&  
=   n_\Sigma  \prts{ D\mathcal{M}_0 \varphi } \prts{\nabla_{\Sigma}h} 
   -\alpha  \prts{ \prts{ D\mathcal{M}_0 \varphi } \prts{\nabla_{\Sigma}h} + \mathcal{M}_0 \nabla_{\Sigma}\varphi },
\end{aligned}
\end{equation}
which implies that
\begin{equation}
\begin{aligned}
 I_1&=
 \beta^3 \prts{ n_\Sigma-\alpha } \prts{ \prts{ D\mathcal{M}_0 \varphi } \prts{\nabla_{\Sigma}h} + \mathcal{M}_0 \nabla_{\Sigma}\varphi }
\prts{
{\rm tr} 
\prts{
 \mathcal{M}_0  \prts{ L_{\Sigma} +\nabla_{\Sigma} \alpha  }
}
}
\\
& =
 \beta^3 
 \prts{ 
   n_\Sigma  \prts{ D\mathcal{M}_0 \varphi } \prts{\nabla_{\Sigma}h} 
   -\alpha  \prts{ \prts{ D\mathcal{M}_0 \varphi } \prts{\nabla_{\Sigma}h} + \mathcal{M}_0
   \nabla_{\Sigma}\varphi } 
 }
\prts{
{\rm tr} 
\prts{
 \mathcal{M}_0  \prts{ L_{\Sigma} +\nabla_{\Sigma} \alpha  }
}
}
\\
& =: \beta^3 I_{11} I_{12} .
\end{aligned}
\end{equation}
Using  \cite[(5.3)]{Pruss.analytic.solu}, we have
\begin{equation}
\begin{aligned}
\norm{ I_1  }_{\mathring{\mathcal{S}}_4 } 
\leq C(\Sigma, T_0) 
\norm{ \beta^3  }_{\mathcal{C}_1} 
   \norm{ I_{11}  }_{\mathring{\mathcal{S}}_4 }
\prts{ \norm{ I_{12}  }_{\infty} +    \norm{ I_{12}  }_{\mathcal{S}_4}  } .
\end{aligned}
\end{equation}
From the formula of $\beta$ in \eqref{formula;n,Gamma}, 
we can find a sufficiently small $\delta_0(\Sigma) \in ( 0,1) $,
such that for all  $\norm{ h }_{\mathcal{C}_3}< \delta_0$,
 if $\norm{h}_{\mathcal{C}_2} <M$ for some $M>0$ then 
\begin{equation} \label{DG2,beta3.estimate.1}
\begin{aligned}
\norm{ \beta^3 }_{\mathcal{C}_1} \leq C(\Sigma,M) .
\end{aligned}
\end{equation}
Without loss of generality, we can always use the same notation $\delta_0$ and $M$ in the derivation of
\eqref{DG2,beta3.estimate.1}. This is because we can always choose 
the smallest $\delta_0$ and the largest $M$ 
from finite many candidates. 
From \cite[(5.4)]{Pruss.analytic.solu}, we have for all $T\in (0,T_0]$ that
\begin{equation}
\begin{aligned}
 \norm{ I_{11}  } _{\mathring{\mathcal{S}}_4^T}
&=
\norm{
 n_\Sigma  \prts{ D\mathcal{M}_0 \varphi } \prts{\nabla_{\Sigma}h} 
   -\alpha  \prts{ \prts{ D\mathcal{M}_0 \varphi } \prts{\nabla_{\Sigma}h} + \mathcal{M}_0
   \nabla_{\Sigma}\varphi_h } 
 } _{\mathring{\mathcal{S}}_4}
\\
& \leq \norm{
 n_\Sigma  \prts{ D\mathcal{M}_0 \varphi } \prts{\nabla_{\Sigma}h} 
   } _{\mathring{\mathcal{S}}_4}
   +
\norm{
   \alpha  \prts{ \prts{ D\mathcal{M}_0 \varphi } \prts{\nabla_{\Sigma}h} + \mathcal{M}_0
   \nabla_{\Sigma}\varphi_h } 
 } _{\mathring{\mathcal{S}}_4}
\\
& \leq
  C(\Sigma, T_0)\norm{ D\mathcal{M}_0 \varphi } _{ \mathring{\mathcal{S}}_4 } 
  \prts{
  		\norm{\nabla_{\Sigma}h}  _{\mathcal{S}_4} +   \norm{\nabla_{\Sigma}h}  _{\infty}
  	}
 \\
 & \quad  +
  C(\Sigma, T_0) \prts{ 
   \norm{ \alpha}  _{\mathcal{S}_4}  +    \norm{ \alpha}  _{\infty}  
   }
    \norm{ \prts{ D\mathcal{M}_0 \varphi } \prts{\nabla_{\Sigma}h} + \mathcal{M}_0
   \nabla_{\Sigma}\varphi_h }  _{\mathring{\mathcal{S}}_4 }
 \\
& \leq
  C(\Sigma, T_0, M) \norm{\varphi_h}_{ \mathring{\mathcal{S}}_4  } 
\prts{  		\norm{\nabla_{\Sigma}h}  _{\mathcal{S}_4} +   \norm{\nabla_{\Sigma}h}  _{\infty} }
\\
& \quad +   C(\Sigma, T_0, M)
\prts{  		\norm{\nabla_{\Sigma}h}  _{\mathcal{S}_4} +   \norm{\nabla_{\Sigma}h}  _{\infty} }
\prts{
\prts{  		\norm{\nabla_{\Sigma}h}  _{\mathcal{S}_4} +   \norm{\nabla_{\Sigma}h}  _{\infty} }
\norm{\varphi_h}_{ \mathring{\mathcal{S}}_4  } 
+ \norm{\nabla_\Sigma \varphi_h}_{ \mathring{\mathcal{S}}_4  } 
}
\\
& \leq      C(\Sigma, T_0, M)
\prts{ \norm{z}_{\mathcal{W}^T} +   \norm{\nabla_{\Sigma}h} _{\infty} }
  \norm{\varphi}_{ \mathring{\mathcal{W}}^T } .
\end{aligned}
\end{equation}
The parameter $T$ in the notations of function spaces is omitted for convenience.
When $\norm{ h }_{\mathcal{C}_0^{T_0}}\leq \delta_0$  and $\norm{ h }_{\mathcal{C}_1^{T_0}}\leq M$,
 for all $T\in (0,T_0]$ we have
\begin{equation}
\begin{aligned}
  \norm{ I_{12}
}_{\mathcal{S}_4^T}
&\leq C(T_0,\Sigma)
\norm{ \mathcal{M}_0 }_{\mathcal{C}_1}
 \norm{ L_{\Sigma} +\nabla_{\Sigma} \alpha  } _{\mathcal{S}_4}
\\
& \leq C(T_0,\Sigma)
\norm{ \mathcal{M}_0 }_{\mathcal{C}_1}
\prts{ \norm{ L_{\Sigma} } _{C^1(\Sigma)}
 +
 \norm{  \nabla_{\Sigma} \alpha  } _{\mathcal{S}_4}
 }
\\
& \leq C(\Sigma, T_0, M) \prts{
1+ \norm{   \alpha  } _{\mathcal{S}_5}
}
\leq 
C(\Sigma, T_0, M) \prts{
1+ \norm{h} _{\mathcal{W}_5}
}
\\
&
\leq C(\Sigma, T_0,M) \prts{ 1+  \norm{z} _{\mathcal{W}^T} } .
\end{aligned}
\end{equation}
Moreover, we have
\begin{equation}
\begin{aligned}
\norm{ I_{12} }_{\mathcal{C}_0}
\lesssim 
\norm{ \mathcal{M}_0 }_{\mathcal{C}_0}
 \norm{ L_{\Sigma} +\nabla_{\Sigma} \alpha  } _{\mathcal{C}_0}
\leq C(\Sigma,M).
\end{aligned}
\end{equation}
Thus, when $\norm{h}_{\mathcal{C}_0}<\delta_0 $ 
 and $\norm{h}_{\mathcal{C}_1}< M $,  for all $\varphi \in \mathring{\mathcal{S}}_4 $ 
and $ T\in(0,T_0]$ we have 
\begin{equation}  \label{G2,I1,estimate,final}
\begin{aligned}
\norm{I_1}_{ \mathring{\mathcal{S}}_4^T  } 
\leq
  C(\Sigma, T_0,M)   \prts{ 1 + \norm{z}_{\mathcal{W}^T} }
  \prts{ \norm{z}_{\mathcal{W}^T} + \norm{\nabla_{\Sigma}h}_{\infty} }
  \norm{\varphi}_{ \mathring{\mathcal{S}}_4^T  } .
\end{aligned}
\end{equation}
Similarly to the derivation of $I_1$, we  remove those vanishing products of perpendicular terms in  $I_3$ and $I_4$,
which implies 
\begin{equation} \label{G2,I3,estimate,final}
\begin{aligned}
\norm{I_3}_{ \mathring{\mathcal{S}}_4 ^T } 
\leq
  C(\Sigma, T_0,M) \prts{  1+		\norm{z}_{\mathcal{W}^T} } 
  \prts{  \norm{h}_{\mathcal{W}_5^T } + \norm{h}_{\mathcal{C}_4^T } }
\norm{\varphi}_{ \mathring{\mathcal{W}}_5 ^T } 
\end{aligned}
\end{equation}
and
\begin{equation} \label{G2,I4,estimate,final}
\begin{aligned}
  \norm{I_4}_{ \mathring{\mathcal{S}}_4^T  } 
&\leq
C(\Sigma,T_0,M) 
	\norm{ 
	\prts{ D\mathcal{M}_0 \varphi } \prts{ \mathcal{M}_0 \nabla_{\Sigma}h }
	+   \prts{ \mathcal{M}_0 } \prts{ D \prts{ \mathcal{M}_0  \nabla_{\Sigma}h } \varphi }
	} _{\mathring{\mathcal{C}}_1^T}   
	\norm{ \nabla_{\Sigma} \alpha}_{\mathcal{S}_4^T}  \norm{ \alpha^{\top} } _{\mathcal{C}_1^T}
\\
&  \leq
C(\Sigma,T_0,M)  \prts{ 1+ 	\norm{h}_{\mathcal{C}_2^T}   } \norm{h}_{\mathcal{C}_2^T}
\norm{\varphi}_{\mathring{\mathcal{W}}^T}  
 \norm{z}_{\mathcal{W}^T},
\end{aligned}
\end{equation}
provided that $\norm{h}_{\mathcal{C}_0^{T_0}}<\delta_0$ and $\norm{h}_{\mathcal{C}_2^{T_0}}<M$.
For the term $I_5$, we have to be  careful with the selection of norms in order to estimate higher-order derivatives with our nonredundant regularity. We  have for all $T\in (0,T_0]$ that
\begin{equation} \label{G2,I5,estimate,final}
\begin{aligned}
  \norm{I_5}_{ \mathring{\mathcal{S}}_4^T  }
  &\leq     C(\Sigma,T_0,M) 
\norm{\mathcal{M}_0 \alpha}_{\mathcal{C}_1} 
\prts{ \norm{\alpha}_{\mathcal{S}_4} + \norm{\alpha}_{\infty} }
\norm{\nabla_\Sigma  
	\prts{	(D\mathcal{M}_0 \varphi_h) (\nabla_\Sigma h) + \mathcal{M}_0 \nabla_\Sigma \varphi_h		}   
	} _{\mathring{\mathcal{S}}_4}
\\ 
&  \quad +  C(\Sigma,T_0,M) 
\norm{\mathcal{M}_0 \alpha}_{\mathcal{C}_1} \norm{\nabla_\Sigma \alpha}_{\mathcal{S}_4}
\norm{	
		(D\mathcal{M}_0 \varphi_h) (\nabla_\Sigma h) + \mathcal{M}_0 \nabla_\Sigma \varphi_h	
	} _{\mathring{\mathcal{C}}_1}
\\
&  \leq    C(\Sigma,T_0,M) 
\norm{h}_{\mathcal{C}_2} 
  \prts{  \norm{\nabla_\Sigma h}_{\mathcal{S}_4} + \norm{\nabla_\Sigma h}_{\infty}  }
\norm{
	  	(D\mathcal{M}_0 \varphi_h) (\nabla_\Sigma h) + \mathcal{M}_0 \nabla_\Sigma \varphi_h	   
	} _{\mathring{\mathcal{S}}_5}
\\
&  \quad  +    C(\Sigma,T_0,M) 
\norm{h}_{\mathcal{C}_2}
\norm{h}_{\mathcal{W}_5}
\norm{	(D\mathcal{M}_0 \varphi_h) (\nabla_\Sigma h) + \mathcal{M}_0 \nabla_\Sigma \varphi_h   
	} _{\mathring{\mathcal{C}}_1}
\\
&  \leq   C(\Sigma,T_0,M) 
\prts{  \norm{\nabla_\Sigma h}_{\mathcal{S}_4} + \norm{\nabla_\Sigma h}_{\infty}  }
\prts{ 
	\norm{\varphi_h}_{\mathring{\mathcal{C}}_2} \norm{h}_{\mathcal{W}_5}  
	+ \norm{\varphi_h}_{\mathring{\mathcal{W}}_5}
   }
\\
& \quad   +   C(\Sigma,T_0,M) 
\norm{h}_{\mathcal{W}_5}
\prts{
		\norm{	 \varphi_h}_{\mathring{\mathcal{C}}_1}  
		\norm{ h}_{\mathcal{C}_2}
		+  \norm{\varphi_h}_{\mathring{\mathcal{C}}_2}
	}
\\
&  \leq
  C(\Sigma,T_0,M) 
\prts{ \norm{h}_{\mathcal{C}_3^T} + \norm{z}_{\mathcal{W}^T} }
\prts{1+ \norm{h}_{\mathcal{C}_2^T} + \norm{z}_{\mathcal{W}^T} }
  \norm{\varphi}_{\mathring{\mathcal{W}}^T }.
\end{aligned}
\end{equation}
It remains  to estimate  $I_2 - \prts{ {\rm tr}L_{\Sigma}^2 + \triangle_{\Sigma} }\varphi_h$,
which can be written as
\begin{equation} \label{G2.I2.formula1}
\begin{aligned}
& I_2 - \prts{ {\rm tr}L_{\Sigma}^2 + \triangle_{\Sigma} }\varphi
\\
& =
	\beta
	\prts{
	{\rm tr} 
	\prts{
	\prts{ D\mathcal{M}_0 \varphi } \prts{ L_{\Sigma} +\nabla_{\Sigma} \alpha  }
	+ \prts{  \nabla_{\Sigma} \prts{  \prts{D\mathcal{M}_0 \varphi } \prts{ \nabla_{\Sigma} h }  
	    +   \mathcal{M}_0  \nabla_{\Sigma}\varphi }   } \mathcal{M}_0
	}
	}
- \prts{ {\rm tr}L_{\Sigma}^2 + \triangle_{\Sigma} }\varphi
\\
& =
{\rm tr} \prts{ \beta
\prts{ D \mathcal{M}_0 \varphi } \prts{  L_{\Sigma} +\nabla_{\Sigma} \alpha  }
-
L_{\Sigma}^2 \varphi 
}
+
{\rm tr} \prts{ \beta
\prts {  \nabla_{\Sigma}
\prts{
 \prts{  D \mathcal{M}_0 \varphi   } \prts{ \nabla_{\Sigma} h } 
+ \mathcal{M}_0 \nabla_{\Sigma}\varphi }  
}\mathcal{M}_0
- 
\nabla_{\Sigma}^2\varphi
}
\\
& =
{\rm tr} \prts{ \beta
\prts{ D \mathcal{M}_0 \varphi }  L_{\Sigma} 
-
L_{\Sigma}^2 \varphi 
}
+
{\rm tr} \prts{
\beta
\prts{ D \mathcal{M}_0 \varphi } \nabla_{\Sigma} \alpha 
 }  
\\
& \quad +
{\rm tr} \prts{ \beta \nabla_{\Sigma} 
	\prts{ \prts{ D \mathcal{M}_0 \varphi } \prts{ \nabla_{\Sigma} h  } } 
	\mathcal{M}_0
	}
+
{\rm tr} \prts{ 
\beta \nabla_{\Sigma}  \prts{\mathcal{M}_0 \nabla_{\Sigma}\varphi } \mathcal{M}_0
	- \nabla_{\Sigma}^2 \varphi 
}
\\
& =: {\rm tr} \prts{I_{21}+ I_{22}+ I_{23}+ I_{24} }   .
\end{aligned}
\end{equation}
We still assume that $\norm{h}_{\mathcal{C}_0}<\delta_0$, $\norm{h}_{\mathcal{C}_2}<M$
and $T\in(0,T_0]$.
We recall that the product of two bounded and Lipschitz continuous functions is still Lipschitz continuous. Thus, the estimate of $I_{21}$ to $I_{24}$ can be simplified to the estimate of their factors.
For $I_{21}$,  we have for all $T\in (0,T_0]$ that
\begin{equation} \label{estimate,G2,I2,I21}
\begin{aligned}
  \norm{I_{21}}_{ \mathring{\mathcal{S}}_4^T}
&= \norm{ \beta \prts{ D \mathcal{M}_0 \varphi }  L_{\Sigma} - L_{\Sigma}^2 \varphi  }_{ \mathring{\mathcal{S}}_4^T}
\\
&  \leq  \norm{ \beta \prts{ D \mathcal{M}_0 \varphi }  L_{\Sigma}  
- \beta L_{\Sigma}^2 \varphi
} _{ \mathring{\mathcal{S}}_4^T}
+
\norm{
 \beta L_{\Sigma}^2 \varphi
- L_{\Sigma}^2 \varphi 
} _{ \mathring{\mathcal{S}}_4^T}
\\
&  \leq  C(\Sigma ,T_0)
\norm{ \beta}_{\mathcal{C}_1} 
\norm{
\mathcal{M}_0   L_{\Sigma} \mathcal{M}_0 
-  L_{\Sigma} } _{ \mathcal{S}_4}
\norm{ \varphi_h} _{ \mathring{\mathcal{C}}_1}
+ C(\Sigma ,T_0)
\norm{ \beta -1 } _{ \mathcal{S}_4}
\norm{ \varphi_h} _{ \mathring{\mathcal{C}}_1}.
\end{aligned}
\end{equation}
From the structure of $\beta$ in \eqref{curvature.beta.term} and the fact that $\abs{n_{\Sigma}}=1$, 
there exists $\delta_0(\Sigma)\in (0,1)$, such that when $\norm{h}_{\mathcal{C}_3}<\delta_0$ 
and $\norm{h}_{\mathcal{C}_2}<M$  we have
\[
\norm{n_\Sigma -\alpha}_{\mathcal{C}_0}\geq \frac{1}{2}
\quad \text{and } \quad 
\norm{\beta}_{\mathcal{C}_1} \leq C(\Sigma,M).
\] 
Then we can obtain that
\begin{equation}
\abs{ \beta -1 } 
=
\abs{ \frac{1}{ \abs{ n_\Sigma -\alpha } } - 1 }
=
\abs{ 
\frac{   \abs{ n_\Sigma } - \abs{ n_\Sigma -\alpha } } { \abs{ n_\Sigma -\alpha } }   
}
\leq 
\frac{  \abs{ \alpha} }{ \abs{ n_\Sigma -\alpha } }   
\leq C \abs{ \alpha}.
\end{equation}
It then follows that
\begin{equation} \label{G2,I2,I21,temp1}
\begin{aligned}
\norm{ \beta -1} _{\mathcal{S}_4 }
\leq C(\Sigma,M,T_0) \norm{ \nabla_\Sigma h}_{\mathcal{S}_4}
\end{aligned}
\end{equation}
and
\begin{equation}
\begin{aligned}
\norm{ \beta -1} _{\infty }
\leq C(\Sigma) \norm{\nabla_\Sigma h}_{\infty }  .
\end{aligned}
\end{equation}
Moreover, from the structure of $\mathcal{M}_0$ in \eqref{structure.M0}, we have
\begin{equation} \label{M0-I,F4,estimate}
\begin{aligned}
\norm{\mathcal{M}_0  -I  }_{\mathcal{S}_4 }
=  \norm{\mathcal{M}_0 L_{\Sigma} h  } _{\mathcal{S}_4 }
\leq C(\Sigma,M,T_0) \norm{h}_{\mathcal{S}_4},
\end{aligned}
\end{equation}
which implies that
\begin{equation} \label{G2,I2,I21,temp2}
\begin{aligned}
&  \norm{\mathcal{M}_0   L_{\Sigma} \mathcal{M}_0 
-  L_{\Sigma} } _{\mathcal{S}_4^T}
\\
&  \leq 
\norm{\prts{\mathcal{M}_0 -I}  L_{\Sigma} \mathcal{M}_0 } _{\mathcal{S}_4}
+  \norm{ L_{\Sigma}  \prts{\mathcal{M}_0 -I}   } _{\mathcal{S}_4}
\\
&  \leq  C(T_0,\Sigma)
\norm{  L_{\Sigma} \mathcal{M}_0 } _{\mathcal{C}_1}
\norm{\mathcal{M}_0 -I  } _{\mathcal{S}_4}
+     C(T_0,\Sigma)
\norm{ L_{\Sigma} }_{C^1(\Sigma)}  \norm{ \mathcal{M}_0 -I    } _{\mathcal{S}_4}
\\
&  \leq 
 C(\Sigma,T_0,M) \norm{ \mathcal{M}_0 -I   } _{\mathcal{S}_4}
\leq C(\Sigma,T_0,M) \norm{h}_{\mathcal{S}_4^T} .
\end{aligned}
\end{equation}
From \eqref{estimate,G2,I2,I21}, \eqref{G2,I2,I21,temp1} and \eqref{G2,I2,I21,temp2},
we have for all $T\in(0,T_0]$ that
\begin{equation}
\begin{aligned}
\norm{ I_{21} } _{\mathring{  \mathcal{S} }_4^T }
\leq C(\Sigma,M, T_0 )
 \norm{h}_{\mathcal{W}^T}  \norm{\varphi_h}_{\mathring{\mathcal{W}}^T } .
\end{aligned}
\end{equation}
Using the estimates in 
\cite[Lemma 5.2]{Pruss.analytic.solu}, we obtain the estimate of $I_{22}$: 
\begin{equation}
\begin{aligned}
 \norm{ I_{22} } _{\mathring{  \mathcal{S} }_4^T}
&\leq C(\Sigma, T_0)
  \norm{\beta}  _{\mathcal{C}_1^T}
\norm{ D \mathcal{M}_0 \varphi_h } _{\mathring{  \mathcal{S} }_4^T }
\prts{
\norm{ \nabla_{\Sigma} \alpha }  _{\mathcal{S}_4^T} +  \norm{ \nabla_{\Sigma} \alpha }  _{\mathcal{C}_0^T} 
}
\\
& \leq
C(\Sigma,T_0, M) 
\prts{ \norm{ h }  _{\mathcal{W}_5^T} +  \norm{ h }  _{\mathcal{C}_4^T}  } 
\norm{\varphi_h} _{\mathring{  \mathcal{S} }_4^T }.
\end{aligned}
\end{equation}
Since $\varphi_h\in \mathring{  \mathcal{W} }_5$, 
the term $I_{23}$ has the estimate
\begin{equation}
\begin{aligned}
  \norm{ I_{23} } _{ \mathring{\mathcal{S}}_4^T }
&\leq  C(\Sigma, T_0)
 \norm{ \beta }_{\mathcal{C}_1}
 \norm{ \nabla_{\Sigma} 
	\prts{ \prts{ D \mathcal{M}_0 \varphi } \prts{ \nabla_{\Sigma} h  } } 
	}	_{\mathring{\mathcal{S}}_4}
	 \norm{ \mathcal{M}_0 } _{\mathcal{C}_1}
\\
&  \leq 
C(\Sigma,T_0,M)
 \norm{ \nabla_{\Sigma} 
	\prts{ \prts{ D \mathcal{M}_0 \varphi } \prts{ \nabla_{\Sigma} h  } } 
	}	_{\mathring{\mathcal{S}}_4}
\\
&  \leq
C(\Sigma,T_0,M)	
\norm{\nabla_{\Sigma}\prts{ D \mathcal{M}_0 \varphi } \prts{ \nabla_{\Sigma} h  } }_{\mathring{\mathcal{S}}_4}
	+  C(\Sigma,T_0,M) \norm{\prts{ D \mathcal{M}_0 \varphi } \nabla_{\Sigma}\prts{ \nabla_{\Sigma} h  } 	}_{\mathring{\mathcal{S}}_4}	
\\
&  \leq
C(\Sigma,T_0,M)	
\norm{\nabla_{\Sigma}\prts{ D \mathcal{M}_0 \varphi }}_{\mathring{  \mathcal{S} }_4 }
  \prts{ \norm{ \nabla_{\Sigma} h  } _{\mathcal{S}_4} +  \norm{ \nabla_{\Sigma} h  } _{\infty} }\\
&\quad +  C(\Sigma,T_0,M) \norm{ D \mathcal{M}_0 \varphi } _{\mathring{  \mathcal{C} }_1 }
 \norm{\nabla_{\Sigma}^2 h   	}	 _{\mathcal{S}_4} 
\\
&  \leq
C(\Sigma,T_0,M)
\norm{  D \mathcal{M}_0 \varphi }_{\mathring{  \mathcal{S} }_5 }
    \prts{ \norm{ \nabla_{\Sigma} h  } _{\mathcal{S}_4} +  \norm{ \nabla_{\Sigma} h  } _{\infty} }
+ C(\Sigma,T_0,M) \norm{ \varphi } _{\mathring{  \mathcal{C} }_1 }
 \norm{h}	 _{\mathcal{W}_5} 
\\
&  \leq
C(\Sigma,T_0,M)
\prts{	\norm{z}	 _{\mathcal{W}^T} + \norm{ \nabla_{\Sigma} h  } _{\infty}  }
\norm{\varphi} _{\mathring{\mathcal{W}} ^T}.
\end{aligned}
\end{equation}
Similarly, we obtain the estimate of $I_{24}$:
\begin{equation}
\begin{aligned}
  \norm{ I_{24} } _{\mathring{  \mathcal{S} }_4^T}
& \leq
\norm{\beta
\prts{
 \nabla_{\Sigma}  \prts{\mathcal{M}_0 \nabla_{\Sigma}\varphi_h } \mathcal{M}_0
 -   \nabla_{\Sigma}^2 \varphi_h
}
} _{\mathring{  \mathcal{S} }_4}
+   
\norm{
\prts{\beta -1}    \nabla_{\Sigma}^2 \varphi_h  
} _{\mathring{  \mathcal{S} }_4}
\\
&  \leq C(T_0)
\norm{\beta } _{\mathcal{C}_1}
\norm{  
   \prts{\mathcal{M}_0 \nabla_{\Sigma}\varphi_h } \mathcal{M}_0
 -   \nabla_{\Sigma} \varphi_h
} _{\mathring{  \mathcal{S} }_4}
+    C(T_0)
\prts{ \norm{\beta -1}  _{\mathcal{S}_4} + \norm{\beta -1}  _{\infty}  }
\norm{   \nabla_{\Sigma}^2 \varphi _h 
} _{\mathring{  \mathcal{S} }_4}
\\
&  \leq
C(\Sigma,T_0,M)
\norm{h } _{\mathcal{S}_4}
\norm{  \nabla_{\Sigma} \varphi_h } _{\mathring{  \mathcal{S} }_4}
+    C(\Sigma,T_0,M)    
\prts{ \norm{\nabla_\Sigma h}  _{\mathcal{S}_4} + \norm{\nabla_\Sigma h}  _{\infty}  }
\norm{   \nabla_{\Sigma}^2 \varphi_h  } _{\mathring{  \mathcal{S} }_4}
\\
&  \leq  C(\Sigma,T_0,M)
\prts{ \norm{h } _{\mathcal{W}_5^T} + \norm{\nabla_\Sigma h}  _{\infty} }
\norm{ \varphi } _{\mathring{\mathcal{W}}^T }  . 
\end{aligned}
\end{equation}
From the estimates of $I_{21}$ to $I_{24}$, we obtain the estimate of $I_2 - \prts{ {\rm tr}L_{\Sigma}^2 + \triangle_{\Sigma} }\varphi_h$.
For all $T\in(0,T_0]$ we have
\begin{equation} \label{G2,I2,estimate,final}
\begin{aligned}
\norm{ I_2 - \prts{ {\rm tr}L_{\Sigma}^2 + \triangle_{\Sigma} }\varphi_h } _{\mathring{\mathcal{S}}_4^T}
\leq
C(\Sigma,T_0,M) \prts{ \norm{z} _{\mathcal{W}^T} + \norm{h} _{\mathcal{C}_4^T} } 
\norm{\varphi} _{\mathring{\mathcal{W}} ^T}.
\end{aligned}
\end{equation}

Consequently, from \eqref{estimate.G2.5terms} and the estimates in 
\eqref{G2,I1,estimate,final},
\eqref{G2,I3,estimate,final},
\eqref{G2,I4,estimate,final},
\eqref{G2,I5,estimate,final}
and \eqref{G2,I2,estimate,final}, we obtain for all $T\in(0,T_0]$ that
\begin{equation} \label{G2.derivative.estimate}
\begin{aligned}
\norm{ D \mathcal{G}_2 \depvar{z} \varphi } _{\mathring{\mathcal{S}}_4^T} 
\leq
  C(\Sigma,T_0,M)
\prts{ \norm{h}_{\mathcal{C}_4^T} + \norm{z}_{\mathcal{W}^T} }
\prts{1+ \norm{h}_{\mathcal{C}_2^T} + \norm{z}_{\mathcal{W}^T} }
  \norm{\varphi}_{\mathring{\mathcal{W}}^T } .
\end{aligned}
\end{equation}

\subsubsection{Estimate of $\mathcal{G}_3 $}

We study $\mathcal{G}_3 $ by abbreviating it to
\begin{equation} 
\begin{aligned}
   \mathcal{G}_3 
 &=
\mathcal{P}_{\Sigma}
\jump{
\nu^\pm ( (I-\mathcal{M}_1) \nabla u + \prts{ (I-\mathcal{M}_1) \nabla u }^{\top}  
)
} \mathcal{M}_0 \nabla_{\Sigma}h
\\
& \quad +
\mathcal{P}_{\Sigma}
\jump{
\nu^\pm  \prts{ \mathcal{M}_1 \nabla u + \prts{ \mathcal{M}_1 \nabla u }^{\top}  
}
} n_{\Sigma}
\\
& \quad -
\prts{
\prts{
\jump{
\nu^\pm \prts{ (I-\mathcal{M}_1) \nabla u + \prts{ (I-\mathcal{M}_1) \nabla u }^{\top}  
}
} ( n_{\Sigma} -  \mathcal{M}_0 \nabla_{\Sigma}h ) 
} \cdot n_{\Sigma}
}
\mathcal{M}_0 \nabla_{\Sigma}h
\\
& =:I_1 + I_2 - I_3.
\end{aligned}
\end{equation}
The argument will be in the spirit of \cite{Pruss.analytic.solu}, we include more details for completeness.

To estimate $I_1$, we first calculate the Fr{\'e}chet derivative of $F: z\mapsto \prts{I-\mathcal{M}_1} \nabla u$:
\begin{equation} 
\begin{aligned}
DF\depvar{z}\varphi
=
\prts{ - D\mathcal{M}_1 \varphi_h  } \prts{ \nabla u }
+ \prts{ I - \mathcal{M}_1   } \prts{ \nabla \varphi_u },
\end{aligned}
\end{equation}
which implies that
\begin{equation} 
\begin{aligned}
 D I_1 \depvar{z} \varphi
& =
\mathcal{P}_{\Sigma}
\Bigl\llbracket
	\nu^\pm \bigl( \prts{ - D\mathcal{M}_1 \varphi_h  } \prts{ \nabla u }
		+ \prts{ I - \mathcal{M}_1  } \prts{ \nabla \varphi_u }  \bigr)
\Bigr.
\\
&\quad\quad\quad\quad \Bigl.
		+
		\nu^\pm \bigl(
				\prts{ - D\mathcal{M}_1 \varphi_h  } \prts{ \nabla u }
				+ \prts{ I - \mathcal{M}_1   } \prts{ \nabla \varphi_u }
			\bigr) ^\top
\Bigr\rrbracket 
  \mathcal{M}_0 \nabla_{\Sigma} h 
\\
& \quad +
\mathcal{P}_{\Sigma}
\jump{
\nu^\pm \prts{ (I-\mathcal{M}_1) \nabla u + \prts{ (I-\mathcal{M}_1) \nabla u }^{\top}  
}
} 
\prts{
	\prts{  \mathcal{M}_0 \varphi_h L_{\Sigma} \mathcal{M}_0 }  \prts{\nabla_{\Sigma}h} 
 		 + \mathcal{M}_0 \prts{\nabla_{\Sigma} \varphi_h} 
}
\\
& =: I_{11} + I_{12}.
\end{aligned}
\end{equation}
Notice that estimating   $\jump{A+A^\top}$ 
is equivalent to estimating  $\jump{A}$. 
Moreover, the projection matrix $\mathcal{P}_\Sigma$ is fixed and thus will only become a constant in the estimation.
Suppose that $\norm{h}_{\mathcal{C}_4}<\delta_0 $ for some sufficiently small $\delta_0(\Sigma)$ 
and $\norm{h}_{\mathcal{C}_2}< M $ for some $M$.
In  $I_{11}$, we first obtain
\begin{equation}  \label{G3,I11,temp1}
\begin{aligned}
&  \norm{
\jump{
		\prts{ - D\mathcal{M}_1 \varphi_h  } \prts{ \nabla u }
		+ \prts{ I - \mathcal{M}_1   } \prts{ \nabla \varphi_u }	
}  
} _{\mathring{\mathcal{S}_4}  }
\\
&  \lesssim
\norm{
		\prts{ - D\mathcal{M}_1 \varphi_h  } \prts{ \nabla u }
		+ \prts{ I - \mathcal{M}_1  } \prts{ \nabla \varphi_u }	
} _{\mathcal{W}_6}
\\
&  \leq C(\Sigma, T_0)
\norm{ D\mathcal{M}_1 \varphi_h  }_{\mathring{\mathcal{C}}_1 }  
\norm { \nabla u }_{\mathcal{W} _6}
+ C(\Sigma, T_0)
\norm{ I - \mathcal{M}_1   } _{ \mathcal{C}_1}
\norm{ \nabla \varphi_u }	 _{\mathring{\mathcal{W}} _6 }
\\
&  \leq
C(\Sigma,T_0,M) \norm{  \varphi_h }	 _{ \mathring{\mathcal{C}}_2 }  \norm{u} _{ \mathcal{W}_1 }
+
C(\Sigma,T_0,M) \norm{  \varphi_u }	 _{ \mathring{\mathcal{W}}_1 } 
\\
&  \leq  C(\Sigma,T_0,M) \prts{ 1+\norm{z}_{\mathcal{W} }  }\norm{\varphi}_{ \mathring{\mathcal{W}} }.
\end{aligned}
\end{equation}
The estimation of $\mathcal{M}_0 \nabla_{\Sigma} h =: \alpha$ has  been studied 
in Section \ref{section.estimate.alpha}.
To estimate $I_{12}$, we use the same argument as in \eqref{G3,I11,temp1} to obtain
\begin{equation} \label{G3,I12,temp1}
\begin{aligned}
\norm { \jump {\prts{ I - \mathcal{M}_1 } \nabla u } }_{\mathcal{S}_4}
\leq C(\Sigma,T_0,M) \norm{  u }_{\mathcal{W}_1}.
\end{aligned}
\end{equation}
Moreover, we have
\begin{equation} \label{G3,I12,temp2}
\begin{aligned}
\norm{
	\prts{  \mathcal{M}_0 \varphi_h L_{\Sigma} \mathcal{M}_0 }  \prts{\nabla_{\Sigma}h} 
 		 + \mathcal{M}_0 \prts{\nabla_{\Sigma} \varphi_h} 
} _{\mathring{\mathcal{C}}_1}
\leq  C(\Sigma,T_0,M)  \norm{ \varphi_h}_{\mathring{\mathcal{C}}_2}.
\end{aligned}
\end{equation}
Therefore, by \eqref{G3,I11,temp1}, \eqref{G3,I12,temp1} and \eqref{G3,I12,temp2} we have
\begin{equation} \label{G3,DI1,estimate} 
\begin{aligned}
\norm{
	DI_1\depvar{z}\varphi
} _{\mathring{\mathcal{S}}_4}
\leq
 &C(\Sigma,T_0,M)
\prts{1+ \norm{z}_{\mathcal{W}} }  \norm{\varphi}_{\mathring{\mathcal{W}}} 
 \prts{ \norm{h}_{\mathcal{W}_5} + \norm{\nabla_\Sigma h}_{\infty} }\\
	&+
 	C(\Sigma,T_0,M) 
  \norm{ u}_{\mathcal{W}_1} \norm{ \varphi_h}_{\mathring{\mathcal{C}}_2}. 
\end{aligned}
\end{equation}

Using almost the same argument as for $I_1$, we obtain the Fr{\'e}chet derivative of $I_2$ and $I_3$:
\begin{equation} 
\begin{aligned}
D I_2 \varphi
=
\mathcal{P}_{\Sigma}
\jump{
\nu^\pm 
	\prts{
		\prts{  D\mathcal{M}_1 \varphi_h  } \prts{ \nabla u }
		+   \mathcal{M}_1  \nabla \varphi_u 
		+
			\prts{
					\prts{  D\mathcal{M}_1 \varphi_h  } \prts{ \nabla u }
				+   \mathcal{M}_1  \nabla \varphi_u			} ^\top
		}
}   n_{\Sigma}, 
\end{aligned}
\end{equation}
and
\begin{equation} \label{G3.I3.1} 
\begin{aligned}
& D I_3 \varphi
\  =
\Biggl(
\biggl(
	\Bigl\llbracket
	\  
	\nu^\pm \big(
			\prts{ - D\mathcal{M}_1 \varphi_h  } \prts{ \nabla u }
			+ \prts{ I - \mathcal{M}_1  } \prts{ \nabla \varphi_u }
			\big)
	\Bigr. \biggl. \Biggl.
\\			
&\qquad\qquad\qquad  \left. \left. \Bigl.
			+
					\nu^\pm  \big(
					\prts{ - D\mathcal{M}_1 \varphi_h  } \prts{ \nabla u }
					+ \prts{ I - \mathcal{M}_1   } \prts{ \nabla \varphi_u }
				\big) ^\top
	\Bigr\rrbracket
	\prts{  
	n_\Sigma - \mathcal{M}_0 \nabla_\Sigma h
	}
	\right.
	\right.
	\\
&	+
	\Biggl.
	\biggl.
	\jump{
	\nu^\pm   (I-\mathcal{M}_1) \nabla u 
		+ \nu^\pm \prts{ (I-\mathcal{M}_1) \nabla u }^{\top}  	 
	} 
	\bigl(
		\prts{  \mathcal{M}_0 \varphi_h L_{\Sigma} \mathcal{M}_0 }  \prts{\nabla_{\Sigma}h} 
	 		 + \mathcal{M}_0 \prts{\nabla_{\Sigma} \varphi_h} 
	\bigr)
\biggl)  \cdot n_\Sigma 
\Biggl) \mathcal{M}_0 \nabla_{\Sigma} h
\\
& +
\Biggl(
\prts{
\jump{
\nu^\pm  (I-\mathcal{M}_1) \nabla u +  \nu^\pm \prts{ (I-\mathcal{M}_1) \nabla u }^{\top}  
} ( n_{\Sigma} -  \mathcal{M}_0 \nabla_{\Sigma}h ) 
} \cdot n_{\Sigma}
\Biggl)
\prts{ D \prts{ \mathcal{M}_0 \nabla_{\Sigma}h	} \varphi_h }
\\
&=: A_1 + A_2  .
\end{aligned}
\end{equation}
We remind readers to be careful with the brackets in the term $A_1$ due to its complexity.
Using the estimates in \cite[Proposition 5.1 (a)]{Pruss.analytic.solu}, we have
\begin{equation} \label{G3,DI2,estimate}
\begin{aligned}
  \norm{D I_2 \varphi }_{\mathring{ \mathcal{S} }_4}
&\lesssim
\norm{  \jump{
		\prts{  D\mathcal{M}_1 \varphi_h  } \prts{ \nabla u }
		+   \mathcal{M}_1  \nabla \varphi_u 
 }  }_{\mathring{ \mathcal{S} }_4}
\\
&  \lesssim
\norm{ 	\prts{  D\mathcal{M}_1 \varphi_h  } \prts{ \nabla u }
		+   \mathcal{M}_1  \nabla \varphi_u 
   }_{\mathring{ \mathcal{W} }_6}
\\
&  \leq C(\Sigma, T_0)
\norm{ D\mathcal{M}_1 \varphi_h }_{\mathring{ \mathcal{C} }_1} \norm{ \nabla u }_{\mathcal{W}_6}
+  C(\Sigma, T_0)
\prts{ \norm{\mathcal{M}_1 }_{\mathcal{W}_6}  + \norm{\mathcal{M}_1 }_{\infty}   }
\norm{ \nabla \varphi_u} _{\mathring{ \mathcal{W} }_6}
\\
&  \leq 
C(\Sigma,T_0,M) \norm{ \varphi_h }_{\mathring{ \mathcal{C} }_2} \norm{ \nabla u }_{\mathcal{W}_6}
+ C(\Sigma,T_0,M) 
\prts{ \norm{\nabla_\Sigma h }_{\mathcal{W}_6}  + \norm{\nabla_\Sigma h }_{\infty}   }
\norm{ \nabla \varphi_u }_{\mathring{ \mathcal{W} }_6}
\\
&  \leq 
C(\Sigma,T_0,M) \norm{ \varphi_h }_{\mathring{ \mathcal{W} }_5} \norm{  u }_{\mathcal{W}_1}
+ C(\Sigma,T_0,M) \prts{ \norm{ h }_{\mathcal{W}_5} + \norm{\nabla_\Sigma h }_{\infty}  }
 \norm{  \varphi_u }_{\mathring{ \mathcal{W} }_1} .
\end{aligned}
\end{equation}

To estimate $I_3$, we need the following estimates.
For the term $A_1$ in \eqref{G3.I3.1}, 
we abbreviate it to
\begin{equation} \label{G3,I3,simplify}
\begin{aligned}
\prts{ \prts{
\jump{\nu^\pm \prts{I_{31} + I_{31}^\top} } I_{32}
+ \jump{ \nu^\pm  \prts{ I_{33}+I_{33}^\top } } I_{34}
} \cdot n_{\Sigma} } I_{35}
\end{aligned},
\end{equation}
where
\begin{equation} \label{G3,I3,simplify,def}
\begin{aligned}
&  I_{31}:= \prts{ - D\mathcal{M}_1 \varphi_h  } \prts{ \nabla u }
			+ \prts{ I - \mathcal{M}_1  } \prts{ \nabla \varphi_u } ,
\\
&  I_{32}:= n_\Sigma - \mathcal{M}_0 \nabla_\Sigma h ,
\\
&  I_{33}:= (I-\mathcal{M}_1) \nabla u ,
\\
&  I_{34}:= \prts{  \mathcal{M}_0 \varphi_h L_{\Sigma} \mathcal{M}_0 }  \prts{\nabla_{\Sigma}h} 
	 		 + \mathcal{M}_0 \prts{\nabla_{\Sigma} \varphi_h}  ,
\\
&  I_{35}:=  \mathcal{M}_0 \nabla_{\Sigma} h .
\end{aligned}
\end{equation}
For $I_{31}$ we have
\begin{equation} \label{G3,I3,I31}
\begin{aligned}
&  \norm{ \prts{  D\mathcal{M}_1 \varphi_h  } \prts{ \nabla u }
			+ \prts{ I - \mathcal{M}_1  } \prts{ \nabla \varphi_u } }_{\mathring{ \mathcal{W} }_6}
\\
&  \leq
C(\Sigma, T_0)\norm{   D\mathcal{M}_1 \varphi_h   }_{\mathring{ \mathcal{C} }_1} 
\norm{  \nabla u }_{\mathcal{W}_6}
+ C(\Sigma, T_0)
 \norm{ I - \mathcal{M}_1  }_{\mathcal{C}_1} \norm{ \nabla \varphi_u } _{\mathring{ \mathcal{W} }_6}
\\
&  \leq
C(\Sigma,T_0,M) \norm{\varphi_h}_{\mathring{ \mathcal{C} }_2} 
\norm{  \nabla u }_{\mathcal{W}_6}
+  C(\Sigma,T_0,M) \norm{ \nabla \varphi_u } _{\mathring{\mathcal{W}}_6}
\\
&  \leq
C(\Sigma,T_0,M) \norm{\varphi_h}_{\mathring{ \mathcal{W} }_5} 
\norm{  \nabla u }_{\mathcal{W}_6}
+  C(\Sigma,T_0,M) \norm{ \varphi_u } _{\mathring{ \mathcal{W} }_1}.
\end{aligned}
\end{equation}
For $I_{32}$ we have
\begin{equation} \label{G3,I3,I32}
\begin{aligned}
\norm { n_\Sigma - \mathcal{M}_0 \nabla_\Sigma h }_{\mathcal{C}_1} \leq C(\Sigma,M) .
\end{aligned}
\end{equation}
For $I_{33}$ we have
\begin{equation} \label{G3,I3,I33}
\begin{aligned}
\norm{ (I-\mathcal{M}_1) \nabla u }_{\mathcal{W}_6}
\leq
C(\Sigma, T_0)
\norm{I-\mathcal{M}_1}_{\mathcal{C}_1} \norm{\nabla u}_{\mathcal{W}_6}
\leq C(\Sigma,T_0,M) \norm{ u}_{\mathcal{W}_1} .
\end{aligned}
\end{equation}
For $I_{34} $ we have
\begin{equation} \label{G3,I3,I34}
\begin{aligned}
&  \norm{  
\prts{ \mathcal{M}_0 \varphi_h L_{\Sigma} \mathcal{M}_0 }  \prts{\nabla_{\Sigma}h} 
+ \mathcal{M}_0 \prts{\nabla_{\Sigma} \varphi_h
} } _{\mathring{ \mathcal{C} }_1}
\\
&  \leq
\norm{ \mathcal{M}_0 \varphi_h L_{\Sigma} \mathcal{M}_0 } _{\mathring{ \mathcal{C} }_1}
\norm{\nabla_{\Sigma}h}_{\mathcal{C}_1}
+ \norm{ \mathcal{M}_0 }_{\mathcal{C}_1} 
\norm{\nabla_{\Sigma} \varphi_h}_{\mathring{ \mathcal{C} }_1}
\\
&  \leq
C(\Sigma,M) \norm{ \varphi_h} _{\mathring{ \mathcal{C} }_1}
\norm{h} _{\mathcal{C}_2}
+ C(\Sigma,M)\norm{ \varphi_h} _{\mathring{ \mathcal{C} }_2} .
\end{aligned}
\end{equation}
For $I_{35}$ 
we have
\begin{equation} \label{G3,I3,I35}
\begin{aligned}
\norm{ \mathcal{M}_0 \nabla_{\Sigma} h }_{\mathcal{C}_1}
\lesssim
\norm{ \mathcal{M}_0}_{\mathcal{C}_1} \norm{ \nabla_{\Sigma} h }_{\mathcal{C}_1}
\leq C(\Sigma,M) \norm{  h }_{\mathcal{C}_2}.
\end{aligned}
\end{equation}
Thus, we have
\begin{equation} \label{G3,I3,part1}
\begin{aligned}
&  \norm{ \prts{ \prts{
\jump{\nu^\pm  \prts{I_{31} + I_{31}^\top} } I_{32}
+ \jump{ \nu^\pm  \prts{ I_{33}+I_{33}^\top } } I_{34}
} \cdot n_{\Sigma} } I_{35}
} _{ \mathring{ \mathcal{S} }_4 }
\\
&  \leq C(\Sigma, T_0)
\prts{
\norm{\jump{I_{31} } }_{\mathring{ \mathcal{S} }_4}
 \norm{I_{32}}_{\mathcal{C}_1}
+ \norm{\jump{I_{33} } }_{\mathcal{S}_4}
 \norm{I_{34}}_{\mathring{\mathcal{C}}_1}
} 
\prts{ \norm{I_{35}}_{\mathcal{S}_4} + \norm{I_{35}}_{\infty} }
\\
& \leq C(\Sigma, T_0)
\prts{
\norm{I_{31} }_{\mathring{ \mathcal{W} }_6}
 \norm{I_{32}}_{\mathcal{C}_1}
+ \norm{I_{33} }_{\mathcal{W}_6}
 \norm{I_{34}}_{\mathring{\mathcal{C}}_1}
} 
\prts{ \norm{I_{35}}_{\mathcal{S}_4} + \norm{I_{35}}_{\infty} }
\\
&  \leq
C(\Sigma,T_0,M) \Bigl(
 \norm{\varphi_h}_{\mathring{\mathcal{W}}_5} 
	\norm{\nabla u }_{\mathcal{W}_6} 
+ \norm{\varphi_u}_{\mathring{\mathcal{W}}_1} 
+ \norm{u}_{\mathcal{W}_1} \norm{\varphi_h}_{\mathring{\mathcal{C}}_2}  \norm{h}_{\mathcal{C}_2}
\Bigr.
\\
& \quad \Bigl.
+ \norm{u}_{\mathcal{W}_1} \norm{\varphi_h}_{\mathring{\mathcal{C}}_2} 
\Bigr)
\prts{ \norm{ \nabla_\Sigma h }_{\mathcal{S}_4} + \norm{ \nabla_\Sigma h }_{\infty} }
\\
&  \leq
C(\Sigma,T_0,M) \prts{
 \norm{\varphi}_{\mathring{\mathcal{W}}} 
	\norm{z }_{\mathcal{W}} 
+ \norm{\varphi}_{\mathring{\mathcal{W}}} 
+ \norm{z}_{\mathcal{W}} \norm{\varphi}_{\mathring{\mathcal{W}}}  \norm{h}_{\mathcal{C}_2}
+ \norm{z}_{\mathcal{W}} \norm{\varphi}_{\mathring{\mathcal{W}}} 
} 
\prts{ \norm{ h }_{\mathcal{W}_5} + \norm{ \nabla_\Sigma h }_{\infty} }
\\
&  \leq
C(\Sigma,T_0,M) \prts{
 1 
+ \norm{z}_{\mathcal{W}}  
} 
\prts{ \norm{ h }_{\mathcal{W}_5} + \norm{ \nabla_\Sigma h }_{\infty} }
\norm{\varphi}_{\mathring{\mathcal{W}}}.
\end{aligned}
\end{equation}
The  term $A_2$ in  \eqref{G3.I3.1} can be abbreviated to
\begin{equation} \label{G3,I3,rewrite2}
\begin{aligned}
\prts{ \prts{
 \jump{ \nu^\pm  \prts{ I_{33}+I_{33}^\top } } I_{32}
} \cdot n_{\Sigma} } \prts{ D\alpha \varphi_h } ,
\end{aligned}
\end{equation}
where $\alpha$ is defined in \eqref{Frechet.alpha} and estimated in \eqref{surface,alpha,estimate1} and \eqref{surface,alpha,estimate2}. 
Thus, we have
\begin{equation} \label{G3,I3,part2}
\begin{aligned}
&   \norm { \prts{ \prts{
\jump{ \nu^\pm  \prts{ I_{33}+I_{33}^\top } } I_{32}
} \cdot n_{\Sigma} } \prts{ D\alpha \varphi_h }
}  _{\mathring{\mathcal{S}}_4}
\leq C(\Sigma, T_0)
\norm {   \jump{ I_{33} } } _{\mathcal{S}_4} 
\norm{ I_{32} } _{\mathcal{C}_1}
\norm{ D\alpha \varphi_h }
 _{\mathring{\mathcal{C}}_1}
\\
&   \leq
C(\Sigma,T_0,M)
\norm {   I_{33} } _{\mathcal{W}_6} 
\norm{ \varphi_h }
 _{\mathring{\mathcal{C}}_2}
\leq
C(\Sigma,T_0,M)
\norm {  u } _{\mathcal{W}_1} 
\norm{ \varphi_h }
 _{\mathring{\mathcal{C}}_2}
\leq 
C(\Sigma,T_0,M)
\norm {z} _{\mathcal{W}} 
\norm{ \varphi }
 _{\mathring{\mathcal{W}}}.
\end{aligned}
\end{equation}
From \eqref{G3,I3,part1} and \eqref{G3,I3,part2}, we obtain
\begin{equation} \label{G3,DI3,estimate}
\begin{aligned}
\norm{ D I_3 \varphi } _{\mathring{\mathcal{S}}_4 }
\leq \ 
&C(\Sigma,T_0,M) \prts{
 1 
+ \norm{z}_{\mathcal{W}}  
} 
\prts{ \norm{h}_{\mathcal{W}_5} + \norm{\nabla_\Sigma h}_{\infty}   }
\norm{\varphi}_{\mathring{\mathcal{W}}} \\
&+
C(\Sigma,T_0,M)
\norm {z} _{\mathcal{W}} 
\norm{ \varphi }
 _{\mathring{\mathcal{W}}}.
\end{aligned}
\end{equation}

Consequently, the estimate of $D\mathcal{G}_3$ follows \eqref{G3,DI1,estimate}, \eqref{G3,DI2,estimate} and \eqref{G3,DI3,estimate}.
For all $T\in(0,T_0]$ we have
\begin{equation} \label{G3.derivative.estimate}
\begin{aligned}
& \norm{ D \mathcal{G}_3 \depvar{z} \varphi } _{\mathring{\mathcal{S}}_4^T }
\leq
C(\Sigma,T_0,M)
\prts{ 1+ \norm{z}_{\mathcal{W}^T}  }  \prts{ \norm{z}_{\mathcal{W}^T}  + \norm{\nabla_\Sigma h}_{\infty}  }
\norm{\varphi}_{\mathring{\mathcal{W}}^T} .
\end{aligned}
\end{equation}

\subsection{Estimates of nonlinear terms}

In this section, we estimate the nonlinear operators $G_i$ using  the $\mathcal{S}_i$ norms 
with $ i =1, \cdots , 5$, in a  similar spirit to \cite{Pruss.quali, Pruss.analytic.solu} with more details included  for completeness.
Similar to the arguments before, we still assume $\norm{h}_{\mathcal{C}_4}<\delta_0$ for some sufficiently small $\delta_0(\Sigma)$ and $\norm{h}_{\mathcal{C}_2} < M $.
We will also omit the parameter $T$ in function spaces when there is no confusion.

\subsubsection{Term $G_1$}
We start with the Fr{\'e}chet derivative of $G_1$.
We will use the $\mathcal{S}_1$ norm, i.e., the $L^q([0,T];L^q(\Omega))$ norm.
For the 1st term we have 
\begin{equation} \label{DN1,part1}
\begin{aligned}
D \prts{ \frac{1}{2}\nabla\prts{\abs{B}^2} }\varphi 
= D (\nabla B \cdot B) \varphi
= \nabla \varphi_B \cdot B + \nabla B \cdot \varphi_B  .
\end{aligned}
\end{equation}
Its estimate is
\begin{equation} \label{DN1,part1,estimate}
\begin{aligned}
  \norm{ \prts{ D \frac{1}{2}\nabla\prts{\abs{B}^2} }\varphi }_{  \mathring{\mathcal{S}}_1 }
&\lesssim  \norm{\nabla \varphi_B}_{ \mathring{\mathcal{C}}_0 }
		\norm{B}_{ \mathcal{S}_1 }
+ \norm{\nabla B}_{ \mathcal{S}_1 }
\norm{\nabla \varphi_B}_{ \mathring{\mathcal{C}}_0 }
\\
&   \lesssim  \norm{ \varphi_B}_{ \mathring{\mathcal{C}}_1 }
		\norm{B}_{ \mathcal{S}_1 }
+ \norm{  B}_{ \mathcal{W}_6 }
\norm{ \varphi_B}_{ \mathring{\mathcal{C}}_0 }
\\
&   \lesssim  \norm{ \varphi_B}_{ \mathring{\mathcal{W}}_2 }
		\norm{B}_{ \mathcal{W}_2 }
+ \norm{  B}_{ \mathcal{W}_2 }
\norm{ \varphi_B}_{ \mathring{\mathcal{W}}_2 }
\leq C( \Sigma, T_0 ) \norm{z}_{ \mathcal{W} } \norm{ \varphi }_{ \mathring{\mathcal{W}} }  .
\end{aligned}
\end{equation}
The 2nd and 3rd terms can be estimated using exactly the same argument and function spaces.
For the 4th term we have
\begin{equation} \label{DN1,part4}
\begin{aligned}
D \prts{ \mathcal{M}_3 \nabla u } \varphi
= \prts{ D \mathcal{M}_3 \varphi_h } \nabla u + \mathcal{M}_3 \nabla \varphi_u.
\end{aligned}
\end{equation}
Using \eqref{M3.estimate}, we have
\begin{equation} \label{DN1,part4,estimate}
\begin{aligned}
   \norm{ D \prts{ \mathcal{M}_3 \nabla u } \varphi}_{ \mathring{\mathcal{S}}_1 }
&\lesssim 
\norm{ D \mathcal{M}_3 \varphi_h  }_{ \mathring{\mathcal{C}}_0 } 
\norm{ \nabla u}_{\mathcal{S}_1} 
 + \norm{\mathcal{M}_3 }_{\mathcal{S}_1}
 \norm{ \nabla \varphi_u} _{ \mathring{\mathcal{C}}_0 } 
 \\
&   \leq C(\Sigma,T_0,M) \norm{\varphi_h  }_{ \mathring{\mathcal{C}}_1 }
	\norm{  u}_{\mathcal{W}_1} 
 +  C(\Sigma,T_0,M) \norm{\partial_t h}_{ \mathcal{S}_1 }
 	 \norm{  \varphi_u} _{ \mathring{\mathcal{W}}_1 } 
\\
&   \leq C(\Sigma,T_0,M)  \norm{z}_{\mathcal{W}} \norm{\varphi}_{ \mathring{\mathcal{W}} }   .
\end{aligned}
\end{equation}
For the 5th term we have
\begin{equation} \label{DN1,part5}
\begin{aligned}
D \prts{ u\mathcal{M}_1  \nabla u}\varphi 
= \varphi_u   \mathcal{M}_1  \nabla u
	+ u  \prts{ D \mathcal{M}_1 \varphi_h } \nabla u 
  	+ u \mathcal{M}_1 \nabla  \varphi_u.
\end{aligned}
\end{equation}
Its estimate is
\begin{equation} \label{DN1,part5,estimate}
\begin{aligned}
&   \norm{ D \prts{ u\mathcal{M}_1  \nabla u}\varphi }_{ \mathring{\mathcal{S}}_1 }
\\
&   \lesssim  
	\norm {\varphi_u}_{ \mathring{\mathcal{C}}_0 }   
		\norm{ \mathcal{M}_1 } _{ \mathcal{C}_0 }  \norm{ \nabla u } _{ \mathcal{S}_1 }
	+ \norm{ u } _\infty 
	 \norm{ D \mathcal{M}_1 \varphi_h } _{  \mathring{\mathcal{C}_0 } } 
	 \norm{ \nabla u } _{ \mathcal{S}_1 } 
  	+ \norm{u}_{\mathcal{S}_1} 
  	\norm{\mathcal{M}_1}_{ \mathcal{C}_0 } 
  	\norm{\nabla\varphi_u} _{ \mathring{\mathcal{C}_0 } }
\\
&   \leq C(\Sigma,T_0,M)
\prts{ 
	\norm {\varphi_u}_{ \mathring{\mathcal{C}}_0 }   
		\norm{h} _{ \mathcal{C}_3 }  \norm{u } _{ \mathcal{W}_1 }
	+ \norm{ u } _\infty 
	 \norm{ \varphi_h } _{  \mathring{\mathcal{C}_3 } } 
	 \norm{  u } _{ \mathcal{W}_1 } 
  	+ \norm{u}_{\mathcal{W}_1} 
  	\norm{h}_{ \mathcal{C}_3 } 
  	\norm{\varphi_u} _{ \mathring{\mathcal{C}_3 } }
}
\\
&   \leq  C(\Sigma,T_0,M) \prts{  \norm{h}_{\mathcal{C}_3}  + \norm{ u } _\infty }
  \norm{z}_{\mathcal{W}} \norm{\varphi}_{ \mathring{\mathcal{W}} }  	.
\end{aligned}
\end{equation}
Using the same argument, we obtain the estimate of the 6th term:
\begin{equation} 
\begin{aligned}
\norm{ D \prts{ B\mathcal{M}_1  \nabla B}\varphi }_{ \mathring{\mathcal{S}}_1 }
\leq  C(\Sigma,T_0,M) \prts{  \norm{h}_{\mathcal{C}_1}  + \norm{ B } _\infty }
  \norm{z}_{\mathcal{W}} \norm{\varphi}_{ \mathring{\mathcal{W}} } .
\end{aligned}
\end{equation}
For the 7th term we have
\begin{equation} \label{DN1,part7}
\begin{aligned}
D  \prts{ \frac{1}{2} \mathcal{M}_1 \nabla\prts{\abs{B}^2} }\varphi 
=
\prts{  D\mathcal{M}_1 \varphi }  \nabla B B
+  \mathcal{M}_1 \nabla \varphi_B  B 
+  \mathcal{M}_1  \nabla B  \varphi_B.
\end{aligned}
\end{equation}
Estimating by terms, we obtain
\begin{equation} 
\begin{aligned}
&   \norm{\prts{  D\mathcal{M}_1 \varphi }  \nabla B B}_{ \mathring{\mathcal{S}}_1 }
\lesssim
\norm{D\mathcal{M}_1 \varphi_h }_{ \mathring{\mathcal{C}}_0 }
\norm{\nabla B}_{ \mathcal{S}_1 }
\norm{ B}_{ \infty }
\\
&   \leq C(\Sigma, T_0, M)
\norm{\varphi_h}_{ \mathring{\mathcal{W}}_5 }
\norm{ B}_{ \mathcal{W}_2 }
\norm{ B}_{ \infty }
\leq 
   C(\Sigma, T_0, M)
\norm{\varphi}_{ \mathring{\mathcal{W}} }
\norm{z}_{ \mathcal{W} }
\norm{ B}_{ \infty },
\end{aligned}
\end{equation}
\begin{equation} 
\begin{aligned}
&   \norm{ \mathcal{M}_1 \nabla \varphi_B  B }_{ \mathring{\mathcal{S}}_1 }
\lesssim 
\norm{ \mathcal{M}_1 }_{ \mathcal{C}_0 }
\norm{ \nabla \varphi_B }_{ \mathring{\mathcal{C}}_0 }
\norm{ B }_{ \mathcal{S}_1 }
\\
&   \leq  C(\Sigma, T_0,M) \norm{ h }_{ \mathcal{C}_3 }
\norm{  \varphi_B }_{ \mathring{\mathcal{W}}_2 }
\norm{ B }_{ \mathcal{W}_2 } 
\leq  C(\Sigma, T_0,M) 
\norm{ h }_{ \mathcal{C}_3 }
\norm{  \varphi }_{ \mathring{\mathcal{W}} }
\norm{ z }_{\mathcal{W}} ,
\end{aligned}
\end{equation}
and
\begin{equation} 
\begin{aligned}
&   \norm{   \mathcal{M}_1  \nabla B  \varphi_B }_{ \mathring{\mathcal{S}}_1 }
\lesssim 
\norm{ \mathcal{M}_1 }_{ \mathcal{C}_0 }
\norm{ \varphi_B }_{ \mathring{\mathcal{C}}_0 }
\norm{ \nabla B }_{ \mathcal{S}_1 }
\\
&   \leq C(\Sigma,T_0,M) \norm{ h }_{ \mathcal{C}_3 }
\norm{  \varphi_B }_{ \mathring{\mathcal{W}}_2 }
\norm{ B }_{ \mathcal{W}_2 } 
\leq  C(\Sigma, T_0,M) 
 \norm{ h }_{ \mathcal{C}_3 }
\norm{  \varphi }_{ \mathring{\mathcal{W}} }
\norm{ z }_{\mathcal{W}} .
\end{aligned}
\end{equation}
Thus, we have
\begin{equation} 
\begin{aligned}
\norm{  D  \prts{ \frac{1}{2} \mathcal{M}_1 \nabla\prts{\abs{B}^2} }\varphi }_{ \mathring{\mathcal{S}}_1 }
\leq 
C(\Sigma, T_0,M) \prts{ \norm{ h }_{ \mathcal{C}_3 } + \norm{B }_{ \infty } }
\norm{ z }_{ \mathcal{W} }
\norm{ \varphi }_{ \mathring{\mathcal{W}} }.
\end{aligned}
\end{equation}
For the 8th term we have
\begin{equation} \label{DN1,part8} 
\begin{aligned}
D \prts{ \mathcal{M}_1 \nabla p } \varphi
=  \prts{ D\mathcal{M}_1 \varphi_h } \nabla p + \mathcal{M}_1 \nabla \varphi_p .
\end{aligned}
\end{equation}
For any $\varphi\in \mathring{\mathcal{W}}$ we have the estimate
\begin{equation}  \label{DN1,part8,estimate}
\begin{aligned}
   \norm{ D \prts{ \mathcal{M}_1 \nabla p } \varphi }_{\mathring{\mathcal{S}}_1}
&\lesssim  \norm{ D\mathcal{M}_1 \varphi_h }_{\mathring{\mathcal{C}}_0} 
\norm{ \nabla p }_{ \mathcal{S}_1 }
+ \norm{ \mathcal{M}_1 }_{ \infty } \norm{ \nabla \varphi_p }_{\mathcal{S}_1}
\\
&   \leq 
C(\Sigma, T_0,M) \norm{\varphi_h }_{ \mathring{\mathcal{C}}_1 } 
\norm{ p }_{ \mathcal{W}_3 }
+ C(\Sigma, T_0,M)\norm{ \nabla_\Sigma h }_{  \infty } \norm{ \varphi_p }_{\mathring{\mathcal{W}}_3}
\\
&   \leq 
C(\Sigma, T_0,M) \norm{\varphi_h }_{ \mathring{\mathcal{W}}_5 } 
\norm{ p }_{ \mathcal{W}_3 }
+ C(\Sigma, T_0,M) \norm{ \nabla_\Sigma h }_{\infty} \norm{ \varphi_p } _{\mathring{\mathcal{W}}_3}
\\
&   \leq
 C(\Sigma, T_0,M) 
\prts{ \norm{ z }_{ \mathcal{W} } +  \norm{ \nabla_\Sigma h }_{\infty} }
\norm{\varphi }_{ \mathring{\mathcal{W}} } .
\end{aligned}
\end{equation}
For the 9th term, notice that the viscosity $\nu^\pm$ is a fixed function in the transformed equations since  
the interface has been pulled to a fixed surface.
Thus, we have
\begin{equation}  \label{DN1,part9}
\begin{aligned}
D\prts{  	\nu^\pm \mathcal{M}_4 : \nabla^2 u } \varphi
= \nu^\pm \prts{  D \mathcal{M}_4 \varphi_h} : \nabla^2 u 
	+ \nu^\pm    \mathcal{M}_4 : \nabla^2 \varphi_u .
\end{aligned}
\end{equation}
For any $\varphi\in \mathring{\mathcal{W}} $ we have the estimate
\begin{equation}  \label{DN1,part9,estimate}
\begin{aligned}
   \norm{ D\prts{  	\nu^\pm \mathcal{M}_4 : \nabla^2 u } \varphi } _{\mathcal{S}_1 }
&\lesssim \norm{  D \mathcal{M}_4 \varphi_h} _{\mathring{\mathcal{C}}_0 }
		\norm{ \nabla^2 u }_{\mathcal{S}_1}
	+ \norm{ \mathcal{M}_4 }_{\infty }
	  \norm{ \nabla^2 \varphi_u } _{\mathcal{S}_1} 
\\
&   \leq
C(\Sigma, T_0,M) \norm{ \varphi_h}_{\mathring{\mathcal{C}}_3 }  \norm{u}_{\mathcal{W}_1}
	+  C(\Sigma, T_0,M)	\norm{ \nabla_\Sigma h }_{\infty}
	  \norm{  \varphi_u } _{ \mathring{\mathcal{W}}_1 } 
\\
&   \leq
C(\Sigma, T_0,M) \norm{ \varphi_h}_{\mathring{\mathcal{W}}_5 } \norm{u}_{\mathcal{W}_1}
+  C(\Sigma, T_0,M)	\norm{ \nabla_\Sigma h }_{\infty}
	  \norm{  \varphi_u } _{ \mathring{\mathcal{W}}_1 } 
\\
&   \leq C(\Sigma, T_0,M) \prts{\norm{z}_{\mathcal{W}}+ \norm{ \nabla_\Sigma h }_{\infty} } 
\norm{ \varphi }_{\mathring{\mathcal{W}} }. 
\end{aligned}
\end{equation}
For the 10th term we have
\begin{equation}  \label{DN1,part10}
\begin{aligned}
D\prts{  	\nu^\pm \mathcal{M}_2 \cdot \prts{\nabla u } } \varphi
=
  	\nu^\pm \prts{ D \mathcal{M}_2 \varphi_h }  \cdot \prts{\nabla u } 
+   	\nu^\pm \mathcal{M}_2   \cdot \prts{\nabla \varphi_u }. 
\end{aligned}
\end{equation}
For any $\varphi\in \mathring{\mathcal{W}} $ we have the estimate
\begin{equation}  \label{DN1,part10,estimate}
\begin{aligned}
   \norm{ D\prts{  	\nu \mathcal{M}_2 \cdot \prts{\nabla u } } \varphi }_{ \mathcal{S} _1 }
&\lesssim
  	 \norm{ D \mathcal{M}_2 \varphi_h } _{\mathring{\mathcal{C}}_0 }
  	  \norm{\nabla u } _{ \mathcal{S}_1 }
+    \norm{\mathcal{M}_2} _{ \infty }  \norm{\nabla \varphi_u }_{ \mathcal{S}_1 }
\\
&   \leq
C(\Sigma,T_0,M)   	 \norm{ \varphi_h } _{\mathring{\mathcal{C}}_4 }
  	  \norm{ u } _{ \mathcal{W}_1 }
  	 +  C(\Sigma,T_0,M)   \norm{\nabla_\Sigma h} _{ \infty } 
  	  \norm{ \varphi_u }_{\mathring{\mathcal{W}}_1}
\\
&   \leq C(\Sigma,T_0,M) \prts{ \norm{z} _{\mathcal{W}} +   \norm{\nabla_\Sigma h} _{ \infty }  }
  \norm{ \varphi }_{\mathring{\mathcal{W}}}.
\end{aligned}
\end{equation}

Combining the estimates of all these 10 terms, we obtain
\begin{equation} \label{DN1,total,estimate}
\begin{aligned}
\norm{ D G_1\depvar{z} \varphi }_{  \mathcal{S} _1 }
\leq C(\Sigma,T_0,M) \prts{
 \norm{\nabla_\Sigma h}_{\infty} 
+
\prts{ 1+ \norm{u}_\infty + \norm{B}_\infty  }
  \norm{z}_{\mathcal{W}}   
   }  \norm{\varphi}_{ \mathring{\mathcal{W}} } .
\end{aligned}
\end{equation}

\subsubsection{Term $G_2$}

Now we estimate $G_2$ using the $\mathcal{S}_2$ norm, which is equal to the $\mathcal{S}_1$ norm.
The 1st and 2nd terms  $u\nabla B$ and $B \nabla u$ can be treated using 
the same argument as in \eqref{DN1,part1} and \eqref{DN1,part1,estimate} since their structures and norms are exactly the same.
Thus, we have
\begin{equation} \label{DN2,part1,estimate}
\begin{aligned}
\norm{ D \prts{ u\nabla B} \varphi }_{ \mathcal{S}_2 }
\leq C \norm{z}_{\mathcal{W}} \norm{\varphi}_{ \mathring{\mathcal{W}} }
\end{aligned}
\end{equation}
and
\begin{equation} \label{DN2,part2,estimate}
\begin{aligned}
\norm{ D \prts{ B\nabla u} \varphi }_{ \mathcal{S}_2 }
\leq C \norm{z}_{\mathcal{W}} \norm{\varphi}_{ \mathring{\mathcal{W}} }.
\end{aligned}
\end{equation}
The 3rd and 4th terms  $u \mathcal{M}_1 \nabla B$ and $B \mathcal{M}_1 \nabla u$ 
can be treated using the same argument 
as in \eqref{DN1,part5} and \eqref{DN1,part5,estimate}, which implies
\begin{equation} \label{DN2,part3,estimate}
\begin{aligned}
&   \norm{ D \prts{ u \mathcal{M}_1 \nabla B } \varphi }_{ \mathcal{S} _2 }
\\
&   \leq C(\Sigma,T_0,M)
	\norm {\varphi_u}_{ \mathring{\mathcal{C}}_0 }   
		\norm{h} _{ \mathcal{C}_3 }  \norm{B } _{ \mathcal{W}_2 }
	+ C(\Sigma,T_0,M)\norm{ u } _\infty 
	 \norm{ \varphi_h } _{  \mathring{\mathcal{C}_3 } } 
	 \norm{  B } _{ \mathcal{W}_2 } 
	 \\
	 &\quad
  	+ C(\Sigma,T_0,M)\norm{u}_{\mathcal{W}_1} 
  	\norm{h}_{ \mathcal{C}_3 } 
  	\norm{\varphi_B} _{ \mathring{\mathcal{C}} _3  }
\\
&   \leq C(\Sigma,T_0,M)
\prts{	\norm{h} _{ \mathcal{C}_3 } + 	\norm{u} _{ \infty} }
 \norm{z} _{\mathcal{W}} \norm {\varphi}_{ \mathring{\mathcal{W}} }  , 
\end{aligned}
\end{equation}
and
\begin{equation} \label{DN2,part4,estimate}
\begin{aligned}
\norm{ D \prts{ B \mathcal{M}_1 \nabla u } \varphi }_{ \mathcal{S} _2 }
\leq C(\Sigma,T_0,M)
\prts{	\norm{h} _{ \mathcal{C}_3 } + 	\norm{B} _{ \infty} }
 \norm{z} _{\mathcal{W}} \norm {\varphi}_{ \mathring{\mathcal{W}} }  .
\end{aligned}
\end{equation}
The 5th, 6th and 7th terms  $\mathcal{M}_3 \nabla B$, 
$\sigma \mathcal{M}_4 : \nabla^2 B$, and $\sigma \mathcal{M}_2 \cdot \nabla B$
 can be treated  using the same argument
 as in  \eqref{DN1,part4,estimate}, \eqref{DN1,part9,estimate},  and \eqref{DN1,part10,estimate}, 
 which implies
\begin{equation} \label{DN2,part5,estimate}
\begin{aligned}
\norm{ D \prts{\mathcal{M}_3 \nabla B } \varphi }_{ \mathcal{S} _2 }
\leq C(\Sigma,T_0,M) \norm{z}_{\mathcal{W}} \norm{\varphi}_{ \mathring{\mathcal{W}} }  ,
\end{aligned}
\end{equation}
\begin{equation} \label{DN2,part6,estimate}
\begin{aligned}
\norm{ D \prts{  \sigma \mathcal{M}_4 : \nabla^2 B  } \varphi }_{ \mathcal{S} _2 }
\leq C(\Sigma,T_0,M) \prts{\norm{z}_{\mathcal{W}}+ \norm{ \nabla_\Sigma h }_{\infty } } 
\norm{ \varphi }_{\mathring{\mathcal{W}} }   ,
\end{aligned}
\end{equation}
and
\begin{equation} \label{DN2,part7,estimate}
\begin{aligned}
\norm{ D \prts{  \sigma \mathcal{M}_2 \cdot \prts{\nabla B }  } \varphi }_{  \mathcal{S} _2 }
 \leq C(\Sigma,T_0,M) \prts{ \norm{z} _{\mathcal{W}} +  \norm{ \nabla_\Sigma h }_{\infty } }
  \norm{ \varphi }_{\mathring{\mathcal{W}}}  .
\end{aligned}
\end{equation}
Consequently, we have the same estimate as in \eqref{DN1,total,estimate}:
\begin{equation} \label{DN2,total,estimate}
\begin{aligned}
\norm{ D G_2\depvar{z} \varphi }_{  \mathcal{S} _2 }
\leq C(\Sigma, T_0,M ) \prts{
\norm{ \nabla_\Sigma h }_{\infty }
+
\prts{ 1+ \norm{u}_\infty + \norm{B}_\infty  }
  \norm{z}_{\mathcal{W}}   	
   }  \norm{\varphi}_{ \mathring{\mathcal{W}} } .
\end{aligned}
\end{equation}
\subsubsection{Term $G_3$}
For term $G_3$, we only need to estimate  $\mathcal{M}_1 : \nabla u$.
Its Fr{\'e}chet derivative is
\begin{equation} \label{DN3,total} 
\begin{aligned}
D G_3 \varphi
= D \prts{ \mathcal{M}_1 : \nabla u } \varphi
=  \prts{ D\mathcal{M}_1 \varphi_h } : \nabla u + \mathcal{M}_1 : \nabla \varphi_u.
\end{aligned}
\end{equation}
Suppose that $\norm{ h } _{ \mathcal{C}_4 } < \delta_0$ and 
$\norm{ h } _{ \mathcal{C}_2 } < M$.
For all $\varphi\in \mathring{\mathcal{W}} $, using similar ideas 
to \cite[(5.20)]{Pruss.analytic.solu},
 we have
\begin{equation} \label{DN3,total,estimate} 
\begin{aligned}
&   \norm{ D G_3 \varphi }_{ \mathring{\mathcal{S}}_3 }
=
\norm{ D \prts{ \mathcal{M}_1 : \nabla u } \varphi }_{ \mathring{\mathcal{S}}_3 }
\\
& \leq \norm{ D \prts{ \mathcal{M}_1 : \nabla u } \varphi }_{ \mathring{W}^{1,q}([0,T]; \dot{W}^{-1,q}(\Omega))}
+ \norm{ D \prts{ \mathcal{M}_1 : \nabla u } \varphi }_{  L^q ([0,T]; W^{1,q}(\Omega)) }
\\
& \leq
C(\Sigma,T_0,M) \prts{ 
\norm{\nabla_\Sigma h}_\infty + \norm{ h}_{\mathcal{W}_5^T}  + \norm{ u}_{\mathcal{W}_1^T}
  }
\prts{
\norm{ \varphi_u }_{\mathring{\mathcal{W}}_1^T}  + \norm{ \varphi_h }_{\mathring{\mathcal{W}}_5^T}
}
\\
& \quad + C\norm{D\mathcal{M}_1 \varphi_h}_{\mathring{\mathcal{C}}_3^T}
\norm{\nabla u }_{\mathcal{W}_1^T}
+ C\norm{\mathcal{M}_1 }_{\mathcal{C}_3^T} \norm{\nabla \varphi_u}_{\mathring{\mathcal{W}}_1^T}
\\
& \leq 
 C(\Sigma, T_0,M) \prts{ \norm{\nabla_\Sigma h}_\infty + \norm{z}_{\mathcal{W}} } 
\norm{\varphi}_{\mathring{\mathcal{W}}^T} .
\end{aligned}
\end{equation}
\subsubsection{Term $G_4$}
For term $G_4$, all its components have been studied in Section \ref{section,G123,estimate}.
We estimate its Fr{\'e}chet derivative
\begin{equation} \label{DN4,total} 
\begin{aligned}
D G_4 \varphi
=
\prts{ D \mathcal{G}_1 \varphi_h + D \mathcal{G}_2 \varphi_h +\kappa ({\rm tr} L_\Sigma^2) \varphi_h  } n_{\Sigma} 
+ D \mathcal{G}_3 \varphi_h .
\end{aligned}
\end{equation}
When $\norm{ h } _{ \mathcal{C}_4 } < \delta_0$ and 
$\norm{ h } _{ \mathcal{C}_2 } < M$,
using \eqref{G1.derivative.estimate}, \eqref{G2.derivative.estimate} 
and \eqref{G3.derivative.estimate}, we obtain 
for all $\varphi\in \mathring{\mathcal{W}} $ and all $T\in(0,T_0]$ that
\begin{equation} \label{DN4,total,estimate} 
\begin{aligned}
   \norm{D G_4 \varphi }_{ \mathring{\mathcal{S}}_4^T } 
&\lesssim
\norm{ D \mathcal{G}_1 \varphi_h} _{ \mathring{\mathcal{S}}_4 } 
+ \norm{ D \mathcal{G}_2 \varphi_h }_{ \mathring{\mathcal{S}}_4 } 
+ \norm{{\rm tr} L_\Sigma^2 \varphi_h}_{ \mathring{\mathcal{S}}_4 }
+\norm{ D \mathcal{G}_3 \varphi_h }_{ \mathring{\mathcal{S}}_4 } 
\\
&   \leq
 C(\Sigma,T_0,M) \prts{ \norm{z}_{\mathcal{W}} + \norm{h}_{\mathcal{C}_4}  } \norm{\varphi}_{ \mathcal{W} }
\\
&   \quad  +  C(\Sigma,T_0,M)
\prts{ \norm{h}_{\mathcal{C}_4} + \norm{z}_{\mathcal{W}} }
\prts{1+ \norm{h}_{\mathcal{C}_2} + \norm{z}_{\mathcal{W}} }
  \norm{\varphi}_{\mathring{\mathcal{W}} }   
	+ C \norm{{\rm tr} L_\Sigma^2 \varphi_h}_{ \mathring{\mathcal{S}}_4 }
\\
&   \quad    +
 C(\Sigma,T_0,M) \prts{  1 + \norm{z}_{\mathcal{W}}  } 
 \prts{ \norm{h}_{\mathcal{C}_4} + \norm{z}_{\mathcal{W}} }  
 \norm{\varphi}_{\mathring{\mathcal{W}}}
\\
&   \leq
 C(\Sigma,T_0,M)
\prts{1+ \norm{h}_{\mathcal{C}_2^T} + \norm{z}_{\mathcal{W}^T} }
\prts{ \norm{h}_{\mathcal{C}_4^T} + \norm{z}_{\mathcal{W}^T} }
  \norm{\varphi}_{\mathring{\mathcal{W}} ^T } 
+ C \norm{{\rm tr} L_\Sigma^2 \varphi_h}_{ \mathring{\mathcal{S}}_4^T }  .
\end{aligned}
\end{equation}

\subsubsection{Term $G_5$}
In the term $G_5$, we recall that $b$ is a fixed auxiliary function for linearity. 
The details of the term $b$ can be found in \cite{Pruss.quali}.
For all  $\varphi\in\mathring{\mathcal{W}}$ we have
\begin{equation} \label{DN5,formula}
\begin{aligned}
   D G_5 \varphi 
&=
-\prts{ (D\mathcal{M}_0\varphi_h) (\nabla_\Sigma h) } \cdot u
+ \prts{ (I-\mathcal{M}_0) \nabla_\Sigma\varphi_h }   \cdot u
\\
&   \quad   
+ \prts{ (I-\mathcal{M}_0)\nabla_\Sigma h } \cdot \varphi_u
-\varphi_u \cdot \nabla_\Sigma h 
+ (b-u)\cdot  \nabla_\Sigma \varphi_h
\\
&   =: -I_1 + I_2 + I_3 -I_4 +I_5.
\end{aligned}
\end{equation}
Suppose that $\norm{h}_{\mathcal{C}_4}<\delta_0$ and $\norm{h}_{\mathcal{C}_2}<M$.
From  \cite[Proposition 5.1 (a) and Lemma 5.5]{Pruss.classical.Stefan}, we have
for all $T\in(0,T_0]$ that
\begin{equation} \label{DN5,I1,estimate}
\begin{aligned}
   \norm{ I_1} _{\mathring{ \mathcal{S} }_5^T }
&\leq C(\Sigma, T_0)
\norm{ u^\top (D\mathcal{M}_0\varphi_h) }_{\mathring{ \mathcal{S} }_5 }  
\prts{ \norm{ \nabla_\Sigma h }_{\mathcal{S}_5} + \norm{ \nabla_\Sigma h }_{\infty} }
\\
&   \leq C(T_0)
\norm{ u}_{\mathcal{S}_5}
\norm{ D\mathcal{M}_0\varphi_h }_{\mathring{ \mathcal{C} }_2 }  
\prts{ \norm{ \nabla_\Sigma h }_{\mathcal{S}_5} + \norm{ \nabla_\Sigma h }_{ \infty } }
\\
&   \leq C(\Sigma,T_0,M)
\norm{ u}_{\mathcal{S}_5}
\norm{\varphi_h }_{\mathring{ \mathcal{C} }_2 }  
\prts{ \norm{ \nabla_\Sigma h }_{\mathcal{S}_5} + \norm{ \nabla_\Sigma h }_{ \infty } }
\\
&   \leq C(\Sigma,T_0,M)
\norm{ u}_{\mathcal{W}_1}
\norm{\varphi_h }_{\mathring{ \mathcal{W} }_5 }  
\prts{ \norm{ h }_{\mathcal{W}_5} + \norm{ \nabla_\Sigma h }_{\infty} }
\\
&   \leq C(\Sigma,T_0,M)
\norm{ z}_{\mathcal{W}^T}
\prts{ \norm{ z }_{\mathcal{W}^T} + \norm{ \nabla_\Sigma h }_{ \infty} }
\norm{\varphi }_{\mathring{ \mathcal{W} } ^T }  .
\end{aligned}
\end{equation}
Since $I-\mathcal{M}_0 = (I-h L_\Sigma)^{-1} \prts{ (I-h L_\Sigma) - I } = - \mathcal{M}_0 L_\Sigma h $,
we have
\begin{equation} \label{DN5,I2,estimate}
\begin{aligned}
   \norm{ I_2} _{\mathring{ \mathcal{S} }_5^T }
&=
 \norm{ \prts{ (\mathcal{M}_0 L_\Sigma h) \nabla_\Sigma\varphi_h }   \cdot u } 
 _{\mathring{ \mathcal{S} }_5 }
\\
&   \leq
C(\Sigma,T_0,M)
\prts{ \norm{h}_{\mathcal{S}_5} + \norm{h}_{ \infty } }
\prts{ \norm{u}_{\mathcal{S}_5} + \norm{u}_{ \infty  } }
\norm{ \nabla_\Sigma\varphi_h } _{\mathring{ \mathcal{S} }_5 }
\\
&   \leq C(\Sigma,T_0,M)
\prts{ \norm{ z }_{\mathcal{W}^T} + \norm{  h }_{ \infty  } }
\prts{ \norm{ z }_{\mathcal{W}^T} + \norm{ u }_{ \infty  } }
\norm{\varphi }_{\mathring{ \mathcal{W} } ^T}  .
\end{aligned}
\end{equation}
Similarly, we obtain the estimates for $I_3$,  $I_4$ and $I_5$:
\begin{equation} \label{DN5,I3,estimate}
\begin{aligned}
   \norm{ I_3} _{\mathring{ \mathcal{S} }_5^T }
&= 
\norm {\prts{ (\mathcal{M}_0 L_\Sigma h)\nabla_\Sigma h } \cdot \varphi_u } 
_{\mathring{ \mathcal{S} }_5 }
\\
&   \leq
C(\Sigma,T_0,M)
\prts{ \norm {h}_{\mathcal{S}_5} + \norm {h}_{\infty} }
 \prts{ \norm {\nabla_\Sigma h}_{\mathcal{S}_5} + \norm {\nabla_\Sigma h}_{\infty} }
\norm { \varphi_u } 
_{\mathring{ \mathcal{S} }_5 }
\\
&   \leq
C(\Sigma,T_0,M)
\prts{ \norm {h}_{\mathcal{S}_5} + \norm {h}_{\infty} }
 \prts{ \norm {h}_{\mathcal{W}_5} + \norm {\nabla_\Sigma h}_{\infty} }
\norm { \varphi_u } 
_{\mathring{ \mathcal{S} }_5 }
\\
&   \leq
C(\Sigma,T_0,M)
\prts{ \norm {z}_{\mathcal{W}^T} + \norm {h}_{\infty} }
 \prts{ \norm { z}_{\mathcal{W}^T} + \norm {\nabla_\Sigma h}_{\infty} }
\norm { \varphi } 
_{\mathring{ \mathcal{W} } ^T } , 
\end{aligned}
\end{equation}
\begin{equation} \label{DN5,I4,estimate}
\begin{aligned}
 \norm{ I_4} _{\mathring{ \mathcal{S} }_5^T }
&\leq C(\Sigma, T_0)
 \prts{ \norm{\nabla_\Sigma h }_{ \mathcal{S}_5 } +  \norm{\nabla_\Sigma h }_{ \infty } }
 \norm{ \varphi_u }_{\mathring{ \mathcal{S} }_5 }
\\
& \leq C(\Sigma,T_0)
 \prts{ \norm{z }_{\mathcal{W}^T} +  \norm{\nabla_\Sigma h }_{ \infty } }
 \norm{ \varphi}_{\mathring{\mathcal{W}}^T } , 
\end{aligned}
\end{equation}
and
\begin{equation} \label{DN5,I5,estimate}
\begin{aligned}
   \norm{ I_5} _{\mathring{ \mathcal{S} }_5^T }
&\leq
C(\Sigma, T_0)
 \prts{ \norm{b-u }_{ \mathcal{S}_5 } +  \norm{b-u }_{ \infty } }
 \norm{ \varphi_h }_{\mathring{ \mathcal{S} }_5 }
\\
&   \leq
C(\Sigma,T_0)
 \prts{ \norm{b-u }_{ \mathcal{S}_5^T } +  \norm{b-u }_{ \infty } }
 \norm{ \varphi }_{\mathring{\mathcal{W}}^T} .
\end{aligned}
\end{equation}

Consequently, we have 
\begin{equation} \label{DN5,total,estimate}
\begin{aligned}
   \norm{ DG_5 \varphi} _{\mathring{ \mathcal{S} }_5 }
&  \leq  C(\Sigma, T_0,M)
\prts{ 1+ \norm{z}_{\mathcal{W}} + \norm{h}_{\infty} + \norm{u}_{\infty}  }
 \prts{ \norm{z}_{\mathcal{W}} +  \norm{h}_{\mathcal{C}_4} }
	  \norm{ \varphi}_{\mathring{\mathcal{W}} }
\\
&   \quad + C(\Sigma, T_0) \prts{ \norm{b-u}_{\mathcal{S} _5} + \norm{b-u}_{\infty} }
	  \norm{ \varphi}_{\mathring{\mathcal{W}} } .
\end{aligned}
\end{equation}

\subsubsection{Estimate of operator $G$}

Combining \eqref{DN1,total,estimate}, \eqref{DN2,total,estimate}, \eqref{DN3,total,estimate}, \eqref{DN4,total,estimate} and \eqref{DN5,total,estimate},  we have  the  following estimate.
\begin{proposition} \label{final.estimate.proposition}
Given any $C^3$ surface $\Sigma$, there exists $\delta_0(\Sigma)\in(0,1)$ sufficiently small,
such that for all $T_0>0$ and all $z=(u,B,p,\varpi,h)\in \mathcal{W} $, 
if
\begin{enumerate} 
\item $\norm{h}_{\mathcal{C}_4^{T_0}}<\delta_0$;
\item $ \norm{h}_{\mathcal{C}_2^{T_0}} \leq M$ for some $M>0$;
\end{enumerate} 
then for all $T\in(0,T_0]$ we have the estimate 
\begin{equation} \label{final estimate}
\begin{aligned}
 & \norm{ DG\depvar{z}\varphi} _{\mathring{\mathcal{S}}^T} 
\\
& \leq
C(\Sigma,T_0,M)\prts{ 1+ \norm{u}_{L^\infty([0,T]\times\Omega)}
 + \norm{B}_{L^\infty([0,T]\times\Omega)} + \norm{h}_{\mathcal{C}_2^T} +  \norm{z}_{\mathcal{W}^T} }
\\
&\quad\quad\quad \cdot \prts{  \norm{h}_{ \mathcal{C}_4^T }   	 + \norm{z}_{ \mathcal{W}^T }   	}
     \norm{\varphi}_{ \mathring{\mathcal{W}} ^T }
\\
&   \quad   + C \norm{{\rm tr} L_\Sigma^2 \varphi_h}_{ \mathring{\mathcal{S}}_4^T } 
+ C(\Sigma, T_0) \prts{ \norm{b-u}_{\mathcal{S} _5 ^T } + \norm{b-u}_{L^\infty([0,T]\times\Sigma)} }
	  \norm{ \varphi}_{\mathring{\mathcal{W}} ^T } 
\end{aligned}
\end{equation}
for all $\varphi = (\varphi_u, \varphi_B, \varphi_p, \varphi_\varpi, \varphi_h) \in \mathring{\mathcal{W}}^T $.
\end{proposition}

\section{Local Existence}  \label{Local existence}

In this section, we study the existence of strong solutions.
Without loss of generality, we fix $T_0>0$ and only consider time intervals $[0,T]\subseteq [0,T_0]$.
We  first consider the case that the initial  interface is a $C^3$ surface.
In this case, we choose the initial interface itself to be the reference surface, which implies that the initial height function $h_0=0$.
Next, we study the case that the initial  interface is a $W^{3-\frac{2}{q}, q}$ surface and is  
 close to some  $C^3$ surface in the sense of $W^{3-\frac{2}{q}, q} $ norm.

\subsection{$C^3$ initial interface} \label{C3 initial interface}

Suppose that we have the  initial condition
\[
u_0\in W^{2-\frac{2}{q}, \, q}(\Omega\setminus \Gamma_0) \cap C(\Omega),
\quad B_0\in W^{2-\frac{2}{q}, \, q}(\Omega),
\quad and \quad \Gamma_0\in C^3,
\]
which satisfy:
\begin{enumerate}
\item 
	$\div u_0 =0$  in $\Omega\setminus \Gamma_0$;  $\div B_0=0$ in $\Omega$;
\item
	 $u_0 =B_0 =0$ on $\partial\Omega$;
\item
	 $\Gamma_0$ is a  closed interface  and $\Gamma_0 \cap \partial\Omega=\emptyset$;
\item
	 $\mathcal{P}_{\Gamma_0}  \jump{\nu^\pm \prts{\nabla u_0 + \nabla u_0 ^\top}  }n_{\Gamma_0} =0 $.
\end{enumerate}
Letting the reference surface be $\Sigma:= \Gamma_0$, then we have $h_0=0$, $\overline{u}_0=u_0$ and $\overline{B_0}=B_0$.
The solution can be  obtained by 
first finding an auxiliary solution with initial value $u_0$ and $B_0$,
and then perturb it.

Picking an arbitrary $T_0>0$, 
using the same argument as in \cite[Theorem 2]{Pruss.quali}, we can 
 extend the initial value $u_0$ to a function $u_b\in \mathcal{W}^{T_0}_1$.
Letting $b$ be the restriction of $u_b$ to $\Sigma$, i.e.
\begin{equation}
b:= u_b |_{[0,T_0]\times\Sigma}.
\end{equation}
We recall the solution operator $S$ in Section \ref{Linear part},
which allows us to define the auxiliary solution by
\begin{equation}
z_\alpha := (u_\alpha, B_\alpha, p_\alpha, \varpi_\alpha,  h_\alpha )
:= S_{(u_0, B_0, 0, b)}( 0,0, 0,  0, 0).
\end{equation}

Next, similarly to \cite{Pruss.quali}, we consider the equation
\begin{equation} \label{L,N,equation,1}
L(z +z_\alpha) = G(z +z_\alpha), \quad z(0)=0.
\end{equation}
with $z \in  \mathring{\mathcal{W}}$.
The sum $z+z_\alpha$ will then solve 
the transformed equations with initial value $(u_0,B_0,h_0)= (u_0,B_0,0)$.
The equation \eqref{L,N,equation,1}  implies
\[
Lz  = G(z +z_\alpha) -L z_\alpha.
\]
When $t=0$, the right-hand side turns to $G(z_0) -L (z_0)$. 
The compatibility conditions for $G(z_0) -L (z_0)$ are 
 $\div u_0 =0$  and
 $(I-n_{\Gamma_0} \otimes n_{\Gamma_0} )\jump{\nu^\pm \prts{ \nabla u_0+ \nabla u_0^\top} }n_{\Gamma_0} =0 $,
 which are exactly included in the requirements of the initial conditions.
Thus, using the solution operator $S_{(0,0,0,b)}$ and the fact that $L z_\alpha =0$, 
we obtain the equation
\begin{equation}
z = S_{(0,0,0,b)} G(z+z_\alpha)   =: K(z) .
\end{equation}
It remains to find the fixed point of $K$ in the space $\mathring{\mathcal{W}}$,
which can be done by a contraction mapping argument.
Let  $r_0>0$ be a sufficiently large fixed number. For $r\in (0,r_0]$ and $T\in(0,T_0]$, we define  
\begin{equation}
 \mathcal{B}_{r}^{T}  :=\{ w\in \mathring{\mathcal{W}}^T : \norm{w}_{\mathcal{W}^T}\leq r \} .
\end{equation}
 Our goal is to show that $K$ is  a contraction mapping on  $\mathcal{B}_{r}^{T} $ for suitable $r$ and $T$.
 
In the auxiliary solution $z_\alpha$, we have $h_\alpha(0)=h_0=0$, which implies that $\norm{h_\alpha(0)}_{C^2}=0$. 
Moreover, by \cite[Proposition 5.1]{Pruss.analytic.solu} we have $h_\alpha\in \mathcal{W}_5^{T_0} \hookrightarrow \mathcal{C}_2^{T_0}$, which implies that 
$h_\alpha\in \mathcal{C}_4^T= C([0,T];C^2(\Sigma))$.
Thus, given any $\varepsilon>0$, there exists a sufficiently small $T_1>0$ 
such that $\norm{h_\alpha(t)}_{C^2(\Sigma)}< \varepsilon$ on $[0,T_1]$, 
i.e., $\norm{h_\alpha}_{\mathcal{C}_4^{T_1}} < \varepsilon$.

Given any $z_1, z_2 \in \mathcal{B}_{r}^{T_0} $, we have the estimate
\begin{equation} 
\begin{aligned}
\norm{ K(z_1 ) - K(z_2 ) }_{\mathcal{W}^T}
\leq \norm{ S } \norm{  G(z_1+z_\alpha) - G(z_2 + z_\alpha)}
\\
\leq C \norm{ S }  \sup_{0\leq c\leq 1} \norm{ DG\depvar{c z_1 + (1-c)z_2 + z_\alpha}  }_{\mathcal{S}^T}
\norm{z_1- z_2}_{\mathring{\mathcal{W}}^T} .
\end{aligned}
\end{equation}
Letting $z_\xi:= c z_1 + (1-c)z_2$  for abbreviation, we estimate the operator $DG$.

\begin{proposition} \label{DN,small,bounded,1}
The term $ 1+ \norm{u}_\infty + \norm{B}_\infty + \norm{h}_{\mathcal{C}_2} +  \norm{z}_{\mathcal{W}} $
in \eqref{final estimate} is bounded on $\mathcal{B}_{r_0}^{T_0}$.
\end{proposition}

\begin{proof}
We start with the estimation of  $\norm{u_\xi}_{L^\infty([0,T_0]\times\Omega)} $ and $\norm{u_\alpha}_{L^\infty([0,T_0]\times\Omega)} $.

Since $u_\alpha\in \mathcal{W}_1^{T_0} \hookrightarrow  C([0,T_0];C^1(\Omega\setminus\Sigma))$
and $u_0\in W^{2-\frac{2}{q},q }\hookrightarrow C^{1,\alpha}$ with $\alpha= 1-(n+2)/q$ (see e.g. \cite[Section 4]{Pruss.quali}) with $n$ being the dimension. 
There exists $M_1(T_0)>0 $ such that 
\begin{equation}
 \norm{u_\alpha}_{L^\infty([0,T_0]\times\Omega)}< M_1 . 
\end{equation}
Since $u_\xi\in \mathcal{B}_{r}^{T_0} $, 
we have by \cite[Proposition 5.1 (a)]{Pruss.analytic.solu} that 
\begin{equation}
  \norm{u_\xi}_{\mathring{C}([0,T_0];C^1(\Omega\setminus\Sigma))} 
  \leq C \norm{u_\xi}_{\mathring{\mathcal{W}}^{T_0} }
  \leq C r_0,
\end{equation}
where the constant $C$ is independent of $T_0$ (when initial value is $0$).
Thus, we have
\begin{equation} 
\begin{aligned}
\norm{ u_\xi + u_\alpha }_{L^\infty ([0,T_0]\times\Omega)} 
\leq  Cr_0 + M_1.
\end{aligned}
\end{equation}
Using the same argument and \cite[Proposition 5.1 (d)]{Pruss.analytic.solu}, we can also obtain
\begin{equation} 
\begin{aligned}
\norm{ B_\xi + B_\alpha }_{L^\infty ([0,T_0]\times\Omega)} 
\leq  Cr_0 + M_1
\end{aligned}
\end{equation}
and
\begin{equation} 
\begin{aligned}
\norm{ h_\xi + h_\alpha }_{\mathcal{C}_2^{T_0}} 
\leq  Cr_0 + M_1,
\end{aligned}
\end{equation}
where we still use the notation $M_1$ without loss of generality.
For the term $\norm{z}_{\mathcal{W}}$ we have
\[
\norm{z}_{\mathcal{W}^T} \leq \norm{z_\alpha}_{\mathcal{W}^T} + \norm{z_\xi}_{\mathring{\mathcal{W}}^T}
\leq C(z_0,T_0,\Sigma) +  r_0
\]
for all $T\in(0,T_0]$, which completes the proof.
\end{proof}

Next, using the same idea as in \cite{Pruss.quali} we claim that the norm of  
$  \norm{h}_{ \mathcal{C}_4^T } + \norm{z}_{ \mathcal{W}^T }  $
in \eqref{final estimate} 
 can be as small as we need by picking a sufficiently small $T\in (0,T_0]$.

\begin{proposition} \label{h,z,DN,small,temp1}
There exists a constant $C_1$, such that for any $\varepsilon>0$, there exists $T(\varepsilon)>0$ and $r(\varepsilon)>0$, such that
for all $z_\xi \in \mathcal{B}_{r}^{T} $ we have
\begin{equation} 
\norm{h_\xi + h_\alpha}_{ \mathcal{C}_4^T } + \norm{z_\xi + z_\alpha}_{ \mathcal{W}^T } 
< C_1 \varepsilon .
\end{equation}
\end{proposition}

\begin{proof}

Since $z_\alpha\in \mathcal{W}$, which consists of Sobolev spaces,
we obtain
\begin{equation} \label{h,z,DN,small,eq1}
\lim_{T\to 0} \norm{z_\alpha}_{ \mathcal{W}^T}=0  
\end{equation}
using the Lebesgue dominated convergence theorem.
Thus, we can pick a sufficiently small $T(\varepsilon) \leq T_0$ such that
\begin{equation} \label{h,z,DN,small,eq2}
\norm{z_\alpha}_{ \mathcal{W}^T}< \varepsilon .
\end{equation}
Without loss  of generality, we assume $r<\varepsilon$, which implies that 
\begin{equation}
\norm{z_\xi+ z_\alpha}_{ \mathcal{W}^{T}}< \varepsilon + r < 2\varepsilon. 
\end{equation}

Since $h_\xi\in \mathring{\mathcal{W}}_5^{T_0}$, for all $T\in(0,T_0]$,  
we have
\begin{equation} 
\norm{h_\xi}_{ \mathring{\mathcal{C}}_2^T }
\leq
\norm{h_\xi}_{ \mathring{\mathcal{C}}_2^{T_0} } 
\leq C \norm{h_\xi}_{ \mathring{\mathcal{W}}_5^{T_0} }
< Cr < C\varepsilon. 
\end{equation}
where the constant $C$ is independent of $T$ or $T_0$.

Since $h_\alpha (0) = 0$ and $h_\alpha \in C([0,T_0];C^2(\Sigma))=: \mathcal{C}_4^{T_0}$,  there exists a sufficiently small $T(\varepsilon)$ such that for all $t\in[0,T]$ we have
\begin{equation} 
\norm{h_\alpha(t)}_{ C^2(\Sigma) }
< \varepsilon ,
\end{equation}
which implies
\begin{equation} 
\norm{h_\alpha}_{ \mathring{\mathcal{C}}_4^T }
< \varepsilon.
\end{equation}
This completes the proof.

\end{proof}

\begin{proposition} \label{trL,DN,small,temp1}

Let  $q>5$ be fixed. Let $\Sigma $ be a compact $C^3$ surface in $\mathbb{R}^3$. 
For all $\varepsilon>0$, there exists $T(q,\Sigma, \varepsilon)>0$,
such that 
\begin{equation} 
\norm{{\rm tr} L_\Sigma^2 \varphi_h}_{ \mathring{\mathcal{S}}_4^T }
< \varepsilon  \norm{ \varphi_h }_{ \mathring{\mathcal{S}}_4^T  } 
\end{equation}
for all $ \varphi_h \in \mathring{\mathcal{W}}_5^T $.

\end{proposition}

\begin{proof}

For convenience, we abbreviate $\varphi_h$ by $\varphi$ and let
\[
r:=1-\frac{1}{q}, \quad s:=\frac{1}{2}-\frac{1}{2q}, \quad \text{and} \quad  f:={\rm tr} L_\Sigma^2 . 
\]
From  the definition of $\mathcal{S}_4$, we have
\begin{equation} 
\begin{aligned}
\norm{ {\rm tr} L_\Sigma^2 \, \varphi}_{  \mathring{\mathcal{S}}_4^T } 
=
\norm{ f \varphi}_{  \mathring{\mathcal{S}}_4^T } 
\leq 
\norm{ f \varphi}_{ W^{s,q}([0,T]; L^q(\Sigma))} 
+  \norm{ f \varphi}_{ L^q([0,T]; W^{r,q}(\Sigma))} . 
\end{aligned}
\end{equation}
\\
Step 1:
\\
Given any $t\in[0,T]$,
we estimate  $\norm{f \varphi} _{ W^{r,q}(\Sigma)} (t)$.
First, we have
\begin{equation} \label{f,phi,Lq,Lq}
\begin{aligned}
\norm{f \varphi} _{ L^q(\Sigma)}
\leq  \norm{f} _{ L^\infty (\Sigma)} \norm{\varphi} _{ L^q(\Sigma)}.
\end{aligned}
\end{equation}
Next, we estimate the Gagliardo seminorm:
\begin{equation}
\begin{aligned}
&  \seminorm{f\varphi}_{W^{r,q}(\Sigma)}
:= \prts{ \int_\Sigma \int_\Sigma \frac{\abs{f(x)\varphi(x) - f(y)\varphi(y)}^q }{\abs{x-y}^{n+rq} }  dydx }^{\frac{1}{q}}
\\
&   \leq C(q) \prts{\int_\Sigma \int_\Sigma \frac{\abs{f(x)}^q\abs{\varphi(x) -\varphi(y)}^q 
+ \abs{f(x)-f(y)}^q \abs{\varphi(y)}^q }{\abs{x-y}^{n+rq} }  dydx  }^{\frac{1}{q}}
\\
&  \leq  C(q) \norm{f}_{L^\infty(\Sigma)} \seminorm{\varphi}_{W^{r,q}(\Sigma)} 
+ C(q) \seminorm{f}_{W^{r,q}(\Sigma)} \norm{\varphi}_{L^\infty(\Sigma)}.
\end{aligned}
\end{equation}
Thus, we have
\begin{equation}  \label{f,phi,F4,temp0.5}
\begin{aligned}
&  \norm{f\varphi}_{W^{r,q}(\Sigma)}
:= \norm{f \varphi} _{ L^q(\Sigma)} + \seminorm{f\varphi}_{W^{r,q}(\Sigma)}
\\
&  \leq
C(q)  \norm{f}_{L^\infty(\Sigma)} \norm{\varphi}_{W^{r,q}(\Sigma)}
+ C(q) \seminorm{f}_{W^{r,q}(\Sigma)} \norm{\varphi}_{L^\infty(\Sigma)} .
\end{aligned}
\end{equation}
Taking the $L^q$ norm on $[0,T]$ and notice that $f$ is independent of $t$, we have
\begin{equation} \label{f,phi,F4,temp1}
\begin{aligned}
  \norm{f\varphi}_{L^q([0,T]; W^{r,q}(\Sigma))}
&\leq
C(q)  \norm{f}_{L^q([0,T];L^\infty(\Sigma))} \norm{\varphi}_{L^\infty ([0,T]; W^{r,q}(\Sigma))}
\\
&   \quad + C(q)  \norm{f}_{L^q([0,T];W^{r,q}(\Sigma))} \norm{\varphi}_{L^\infty([0,T]\times\Sigma)} 
\\
&  \leq 
C(q,\Sigma) T^{\frac{1}{q}} \norm{f}_{L^\infty(\Sigma)}  \norm{\varphi}_{\mathring{\mathcal{W}}_5 ^T }
 + C(q,\Sigma)  T^{\frac{1}{q}} \norm{f}_{W^{r,q}(\Sigma)} \norm{\varphi}_{\mathring{\mathcal{W}}_5 ^T } ,
\end{aligned}
\end{equation}
where we used the embedding theory in \cite[Proposition 5.1]{Pruss.analytic.solu} in the 2nd inequality.
\\
Step 2:
\\
Now we estimate $\norm{f \varphi} _{ W^{s,q}([0,T];L^q(\Sigma)) }$.
From \eqref{f,phi,F4,temp0.5} and \eqref{f,phi,F4,temp1},
we have
\begin{equation} \label{f,phi,F4,temp2}
\begin{aligned}
  \norm{f\varphi}_{L^q([0,T]; L^q(\Sigma))}
&\leq
C(q)  \norm{f}_{L^q([0,T];L^\infty(\Sigma))} \norm{\varphi}_{L^\infty ([0,T]; W^{r,q}(\Sigma))} 
\\
&  \leq
C(q,\Sigma) T^{\frac{1}{q}} \norm{f}_{L^\infty(\Sigma)}  \norm{\varphi}_{\mathring{\mathcal{W}}_5 ^T}.
\end{aligned}
\end{equation}
It remains to estimate the Gagliardo seminorm.
Since $f$ is independent of time, we have
\begin{equation} \label{f,phi,F4,temp3}
\begin{aligned}
&  \seminorm{f\varphi}_{W^{s,q}([0,T]; L^q(\Sigma))} 
:= \prts{ \int_0^T  \int_0^T \frac{\norm{f\varphi(t) - f\varphi(\tau)  }_{L^q (\Sigma)}^q    }{ \abs{t-\tau}^{1+sq} } dt d\tau  }^{\frac{1}{q}}
\\
&  \leq 
   \prts{ \int_0^T  \int_0^T 
   	\frac{ \norm{f}^q_{L^q (\Sigma)} \norm{\varphi(t) - \varphi(\tau)  }_{L^\infty (\Sigma)}^q    }
   	{ \abs{t-\tau}^{1+sq} } dt d\tau  
   	}^{\frac{1}{q}}
\\
&  \leq 
   \prts{ \int_0^T  \int_0^T 
   	\frac{ \norm{f}^q_{L^q (\Sigma)} 
   	\norm{\varphi}_{ \mathring{C}^1([0,T];C^1(\Sigma)) }^q    
   	\abs{t-\tau}^q
   	}
   	{ \abs{t-\tau}^{1+sq} } dt d\tau  
   	}^{\frac{1}{q}}
\\
&  \leq 
   \prts{ 
   \norm{f}^q_{L^q (\Sigma)} 
   	\norm{\varphi}_{ \mathring{C}^1([0,T];C^1(\Sigma)) }^q
   \int_0^T  \int_0^T 	    
   	\abs{t-\tau}^\frac{q-1}{2} dt d\tau  
   	}^{\frac{1}{q}}
\\
&  \leq   C(\Sigma,T_0)
\norm{f}_{L^\infty (\Sigma)} 
   	\norm{\varphi}_{ \mathring{\mathcal{W}}_5^T }    
   \prts{ 
   \int_0^T  \int_0^T 	    
   	T^\frac{q-1}{2} dt d\tau  
   	}^{\frac{1}{q}}   	
\\
&  \leq   C(\Sigma,T_0)
 T^\frac{q+3}{2q} \norm{f}_{L^\infty (\Sigma)} 
   	\norm{\varphi}_{ \mathring{\mathcal{W}}_5^T }     	 .   
\end{aligned}
\end{equation}
From \eqref{f,phi,F4,temp2} and \eqref{f,phi,F4,temp3}, we obtain
\begin{equation} \label{f,phi,F4,temp4}
\begin{aligned}
\norm{f\varphi}_{W^{s,q}([0,T]; L^q(\Sigma))} 
\leq 
C(q,\Sigma,T_0) \prts{ T^\frac{1}{q} +  T^\frac{q+3}{2q} } \norm{f}_{L^\infty(\Sigma)} \norm{\varphi}_{\mathring{\mathcal{W}}_5 ^T } . 
\end{aligned}
\end{equation}

Consequently, from \eqref{f,phi,F4,temp1} and \eqref{f,phi,F4,temp4}, by assuming $T<1$ without any loss of generality, we have
\begin{equation} \label{f,phi,F4,final}
\begin{aligned}
  \norm{ {\rm tr} L_\Sigma^2 \varphi}_{  \mathcal{S}_4^T } 
&\leq 
C(q,\Sigma,T_0) T^{\frac{1}{q}} 
\prts{ \norm{{\rm tr} L_\Sigma^2 }_{L^\infty(\Sigma)} + \norm{{\rm tr} L_\Sigma^2 }_{W^{r,q}(\Sigma)} }
 \norm{\varphi}_{\mathring{\mathcal{W}}_5 ^T} 
\\
&  \leq
C_1(q,\Sigma,T_0) T^{\frac{1}{q}} 
 \norm{\varphi}_{\mathring{\mathcal{W}}_5 ^T} .
\end{aligned}
\end{equation}
For any $\varepsilon>0$, 
a  sufficiently small  $T$  such that $C_1 T^{1/q} < \varepsilon$ completes the proof.
\end{proof}

Using the same idea as in \cite{Pruss.quali, Pruss.analytic.solu},
we obtain the following result.
\begin{proposition} \label{b-u,DN,small,temp1}
There exists a constant $C(\Sigma)$ such that
for any $\varepsilon >0$, there exists $T(\varepsilon)>0$ and $r(\varepsilon)>0$,
such that
for all $z_\xi \in \mathcal{B}_{r}^{T} $ we have
\begin{equation} 
 \norm{b-u}_{\mathcal{S} _5^T} + \norm{b-u}_{\mathring{C}([0,T];C(\Sigma))} 
<   C(\Sigma) \varepsilon.
\end{equation}

\end{proposition}

\begin{proof}
Notice that $\mathcal{S} _5$ consists of Sobolev spaces and Lebesgue spaces. 
Without loss of generality, we require $r<\varepsilon$.
Using the same arguments as in \eqref{h,z,DN,small,eq1} and \eqref{h,z,DN,small,eq2}, 
we can find a sufficiently small $T$,
such that 
\begin{equation} 
\norm{b-u}_{\mathcal{S} _5^T} 
\leq  
\norm{b-u_\alpha}_{\mathcal{S} _5^T}  + \norm{u_\xi}_{\mathring{\mathcal{S}} _5^T}  
\leq \varepsilon + r 
< 2\varepsilon.
\end{equation}

On the other hand, for a sufficiently small $T$, we have
\begin{equation} 
\norm{b-u}_{\infty}
\leq
\norm{b-u_\alpha}_{\mathring{C}([0,T];C(\Sigma))}  
+ C(T_0, \Sigma)  \norm{u_\xi}_{\mathring{\mathcal{W}}_1 ^T }  
<  \varepsilon  +  C(\Sigma) r ,
\end{equation}
which finishes the proof.
\end{proof}

Consequently, we can obtain the smallness of $DG$, which implies
the existence of a strong solution.

\begin{proof} [Proof of Theorem \ref{existence.result.1}]
From
Proposition \ref{DN,small,bounded,1},
Proposition \ref{h,z,DN,small,temp1}, 
Proposition \ref{trL,DN,small,temp1} and
Proposition \ref{b-u,DN,small,temp1},
we can let $T$ and $r$ be sufficiently small such that
\begin{equation}
\begin{aligned}
\norm{K(z_1) -K(z_2)}_{\mathring{\mathcal{W}}}
\leq \frac{1}{2} \norm{z_1 -z_2}_{\mathring{\mathcal{W}}}
\end{aligned}
\end{equation}
for all $z_1, \, z_2 \in \mathcal{B}_{r}^{T} $.
Thus, by the contraction mapping theorem, 
the operator $K$ has a unique fixed point $z_\beta\in \mathcal{B}_{r}^{T}  $,
which implies that 
\[ \overline{z} := (\overline{u}_\gamma, \overline{B}_\gamma, \overline{p}_\gamma, \varpi_\gamma, h_\gamma) := z_\alpha+ z_\beta  \]
is the unique solution to the transformed equations on $[0,T]$.
Here we reuse the bars in the transformed terms. 

Similarly to \cite{Pruss.quali}, we can recover the original solution.
Since $\Sigma\in C^3$ and
\[ h \in \mathcal{W}_5 \hookrightarrow C^1([0,T];C^1(\Sigma)) \cap C^0([0,T];C^2(\Sigma)), \]
the diffeomorphism $\Theta_h\in C^1([0,T];C^1(\Sigma)) \cap C^0([0,T];C^2(\Sigma))$.  
Thus, in the original solution, the terms  $u=\overline{u}\circ\Theta_h^{-1}$,
$B =\overline{B}\circ\Theta_h^{-1}$ and
$p=\overline{p}\circ\Theta_h^{-1}$ 
 have the same regularity as 
 $\overline{u} $,   $\overline{B} $ and  $\overline{p} $
 and the equations are satisfied almost everywhere on corresponding domains.
The jump of the pressure $\jump{p}=\varpi\circ\Theta_h^{-1}$ also satisfies the requirement in 
Definition \ref{definition.strong-solution}.
This finishes the proof of the theorem.  
\end{proof}

\subsection{$W^{3-\frac{2}{q},q}$ initial interface}  \label{section: W interface}

In this section, we prove Theorem   \ref{existence.result.2}.
Suppose that we have 
a $C^3$ surface $\Sigma$ and the nearest point projection property is  valid in $B(\Sigma;\varrho_0)$
for some $\varrho_0(\Sigma) > 0 $.
Let $M_0 >0$ be an arbitrary number and let $\varepsilon_0(M_0) >0$ be  a number to be determined later.
Let $(u_0,B_0,\Gamma_0)$ be an initial value and $h_0$ be the corresponding height function, such that:
\begin{enumerate}
\item 
$u_0\in W^{2-\frac{2}{q}, \, q}(\Omega\setminus \Gamma_0)\cap C(\Omega)$,
\quad $B_0\in W^{2-\frac{2}{q}, \, q}(\Omega)$,
\quad and   $\Gamma_0$ is a  $W^{3-\frac{2}{q}, \, q}$ surface;
	\item 
	$\norm{u_0}_{W^{2-\frac{2}{q} ,  \, q } (\Omega\setminus\Gamma_0)}\leq M_0$,\quad
	 $\norm{u_0}_{L^\infty(\Omega)}\leq M_0$, \quad
	$\norm{B_0}_{W^{2-\frac{2}{q} ,  \, q } (\Omega)}\leq M_0$, 	\quad
	\item 
	$\Gamma_0 \subseteq B(\Sigma;\varrho_0)$,\quad 
	$\norm{h_0}_{W^{3-\frac{2}{q}} (\Sigma)}< \varepsilon_0 $,\quad  
	\item
	$\div u_0 =0$  in $\Omega\setminus \Gamma_0$, \quad $\div B_0=0$ in $\Omega$;
  \item
  $u_0 =B_0 =0$ on $\partial\Omega$;
 \item
 $ \mathcal{P}_{\Gamma_0}  \jump{\nu^\pm  \prts{\nabla u_0 + \nabla u_0^\top} }n_{\Gamma_0} =0 $.
\end{enumerate}
Using the argument in Section \ref{Transformation of problem},
we obtain the initial condition of the transformed problem
\[
(\overline{u}_0, \overline{B}_0 , h_0) = (u_0\circ \Theta_{h_0} , B_0 \circ \Theta_{h_0} , h_0) .
\]
From the structure of $\Theta_{h_0}$,
we know that there exists $\delta_0(\Sigma)\in(0,1)$
 and a constant $C(\Sigma,\varrho_0)$, such that for all $\norm{h_0}_{C^2(\Sigma)}<\delta_0$ 
we have
\begin{equation} \label{u0,B0,transformed.estimate}
\norm{\overline{u}_0}_{W^{2-\frac{2}{q} ,  \, q } (\Omega\setminus\Sigma)} < C M_0 \quad \text{and} \quad
\norm{\overline{B}_0}_{W^{2-\frac{2}{q} ,  \, q } (\Omega)} < C M_0.
\end{equation}
The number $\delta_0$ is only dependent on the fixed reference surface $\Sigma$,
so we may assume $\varepsilon_0 < \delta_0$ without any loss of generality. 
Let $T_0 $ be a fixed number.
Using  the same argument as in \cite[Theorem 2]{Pruss.quali}, 
we obtain the auxiliary term $b$ in \eqref{transformed.interface.speed.eq}
 by extending  $\overline{u}_0$ to  $\overline{u}_b\in \mathcal{W}^{T_0}_1$
 and letting 
 \[ b:= \overline{u}_b |_{[0,T_0]\times\Sigma}. \]
Using \cite[Proposition 2]{Pruss.quali}, we can extend
\[ \div \overline{u}_0 \quad \text{and} \quad  \mathcal{P}_\Sigma
\jump{\nu^\pm \prts{\nabla \overline{u}_0 + \nabla \overline{u}_0 ^\top }  }n_\Sigma  \] 
to two auxiliary functions $\alpha_3\in \mathcal{S}_3^{T_0}$ 
and $\alpha_4\in \mathcal{S}_4^{T_0}$, respectively.
Using the solution operator $S$ from Section \ref{Linear part},
we obtain an auxiliary solution 
\begin{equation}
z_\alpha := (u_\alpha, B_\alpha, p_\alpha, \varpi_\alpha,  h_\alpha )
:= S_{(\overline{u}_0, \overline{B}_0, h_0, b)}( 0,0, \alpha_3,  \alpha_4, 0).
\end{equation}
Similarly to \cite{Pruss.quali},  we consider the equation
\begin{equation} \label{L,N,equation,1,pert}
L(z +z_\alpha) = G(z +z_\alpha), \quad z(0)=0
\end{equation}
with $z \in  \mathring{\mathcal{W}}$,
which can be rewritten as
\[
Lz  = G(z +z_\alpha) -L z_\alpha.
\]
The compatibility conditions for $G(z_0) -L (z_0)$ are  exactly the transformation (via $\Theta_{h_0}$) of the initial conditions 
 \[ \div u_0 =0  \quad \text{and}  \quad  \mathcal{P}_{\Gamma_0} \jump{\nu^\pm  \prts{\nabla u_0 + \nabla u_0^\top } } n_{\Gamma_0} =0 . \]
Using the solution operator $S_{(0,0,0,b)}$, we rewrite the equation as
\begin{equation}
z = S_{(0,0,0,b)} \prts{ G(z+z_\alpha) - L z_\alpha  }  =: K(z) .
\end{equation}
Similarly to 
Section \ref{C3 initial interface},
we fix  $r_0>0$ and define for all $r\in (0,r_0]$ and $T\in(0,T_0]$ that  
\begin{equation}
 \mathcal{B}_{r}^{T}  :=\{ w\in \mathring{\mathcal{W}}^T : \norm{w}_{\mathcal{W}^T} \leq r \} .
\end{equation}
It remains to find suitable $r$ and $T$ 
such that  $K$ is  a contraction mapping on the set $\mathcal{B}_{r}^{T} $.
Since $h_0\neq 0$, we need to slightly modify the estimates in Section \ref{C3 initial interface}.

Given any $z_1, z_2 \in \mathcal{B}_{r}^{T} $, we consider the same estimate as stated
in Section \ref{C3 initial interface}:
\begin{equation} 
\begin{aligned}
\norm{ K(z_1) - K(z_2) }_{\mathcal{W}^T}
&\leq \norm{ S } \norm{  G(z_1+z_\alpha) - G(z_2 + z_\alpha)}_{\mathring{\mathcal{S}}^T}
\\
&\leq C \norm{S }  \sup_{0\leq c\leq 1} \norm{ DG\depvar{c z_1 + (1-c)z_2 + z_\alpha}  }
\norm{z_1- z_2}_{\mathring{\mathcal{W}}^T} .
\end{aligned}
\end{equation}
Letting $z_\xi:= c z_1 + (1-c)z_2$  and $\varphi:= z_1-z_2$,
we need to estimate \eqref{final estimate}
using similar ideas to \cite{Pruss.quali, Pruss.analytic.solu}.
We rewrite the inequality for convenience. Suppose $\norm{h}_{\mathcal{C}_2^{T_0}}\leq M_h$,
then for all $T\in(0,T_0]$ we have
\begin{equation} \label{DG estimate, rewrite1}
\begin{aligned}
 & \norm{ DG\depvar{z_\alpha+z_\xi }\varphi} _{\mathring{\mathcal{S}}^T} \\
 &\leq
C(\Sigma,T_0,M_h)\Bigl( 1+ \norm{u_\alpha+u_\xi}_{L^\infty([0,T]\times\Omega)}
 + \norm{B_\alpha+B_\xi}_{L^\infty([0,T]\times\Omega)} 
 \Bigr.
\\
& \quad\quad\quad \quad\quad\quad 
\Bigl. + \norm{h_\alpha+h_\xi}_{\mathcal{C}_2^T}  +  \norm{z_\alpha+z_\xi}_{\mathcal{W}^T} \Bigr)
\prts{  \norm{h_\alpha+h_\xi}_{ \mathcal{C}_4^T }   	 + \norm{z_\alpha+z_\xi}_{ \mathcal{W}^T }   	}
     \norm{\varphi}_{ \mathring{\mathcal{W}} ^T }
\\
&   \quad   + C \norm{{\rm tr} L_\Sigma^2 \varphi_h}_{ \mathring{\mathcal{S}}_4^T } 
+ C(T_0,\Sigma) \prts{ \norm{b - u_\alpha - u_\xi}_{\mathcal{S} _5 ^T } + \norm{b- u_\alpha - u_\xi}_{L^\infty([0,T]\times\Sigma)} }
	  \norm{ \varphi}_{\mathring{\mathcal{W}} ^T } .
\end{aligned}
\end{equation}

The estimates in  
Proposition \ref{trL,DN,small,temp1} and
Proposition \ref{b-u,DN,small,temp1} can be obtained without any change.
We slightly modify the arguments in Proposition \ref{DN,small,bounded,1} 
and Proposition \ref{h,z,DN,small,temp1}.

\begin{proposition}[Modification of Proposition \ref{DN,small,bounded,1}] 
\label{DN,small,bounded,2}
For $M_0$ and $\varepsilon_0$ defined in the beginning of Section \ref{section: W interface}.
For all $T\in(0,T_0]$ and all $z_\xi\in \mathcal{B}_{r_0}^{T}$,
we have
\begin{equation}  \label{W interface, bounded term1}
1+ \norm{u_\alpha + u_\xi}_\infty + \norm{B_\alpha + B_\xi}_\infty + \norm{h_\alpha + h_\xi}_{\mathcal{C}_2^T} +  \norm{z_\alpha + z_\xi}_{\mathcal{W}^T}
\leq C(\Sigma,T_0,r_0,M_0).
\end{equation}
The constant $C(\Sigma,T_0,r_0,M_0)$ is independent of $h_0$ or $\varepsilon_0$. 
\end{proposition}

\begin{proof}
We recall that the solution operator $S$ in Section \ref{Linear part} is continuous with respect to the initial value.
Without loss of generality, we assume $\varepsilon_0\leq M_0$.
Since 
$$\norm{\overline{u}_0}_{W^{2-\frac{2}{q},q}(\Omega\setminus\Gamma_0)}\leq C M_0,\quad 
\norm{\overline{B}_0}_{W^{2-\frac{2}{q},q}(\Omega)}\leq C M_0, \quad \text{and }
\norm{h_0}_{W^{3-\frac{2}{q},q}(\Sigma)} < \varepsilon_0 \leq M_0,$$
from the derivation of $z_\alpha$, we have for all $T\in(0,T_0]$ that
\begin{equation} \label{z,alpha,estimate1}
\norm{z_\alpha}_{\mathcal{W}^{T}}
\leq \norm{z_\alpha}_{\mathcal{W}^{T_0}} \leq C(M_0,T_0,\Sigma).
\end{equation}
The details of the estimate of $(u_\alpha,p_\alpha,h_\alpha)$ in \eqref{z,alpha,estimate1} can be found 
in \cite[Section 8.2.3 - 8.2.4]{Pruss.green.book}.
The estimate of $B_\alpha$ is stated in \eqref{B,parabolic,estimate}.
From the embedding theory in \cite[Proposition 5.1 (d)]{Pruss.analytic.solu}, we have
for all $T\in(0,T_0]$ that
\[
\norm{h_\alpha}_{\mathcal{C}_2^{T}} 
\leq
\norm{h_\alpha}_{\mathcal{C}_2^{T_0}} 
\leq C(T_0,\Sigma) \norm{h_\alpha}_{\mathcal{W}_5^{T_0}} 
\leq C(M_0, T_0,\Sigma) .
\]
The rest of the proof can be carried out using similar arguments to Proposition \ref{DN,small,bounded,1},
which implies \eqref{W interface, bounded term1}.
\end{proof}
We also modify Proposition \ref{h,z,DN,small,temp1}
since $h_0\neq 0$ in the current case.
\begin{proposition}[Modification of Proposition \ref{h,z,DN,small,temp1}] \label{h,z,DN,small,temp1,pert}
There exists a constant $C_1$, such that for any $\varepsilon_1>0$, there exists $T(\varepsilon_1)>0$ and $r(\varepsilon_1)>0$, such that
for all $z_\xi \in \mathcal{B}_{r}^{T} $ we have
\begin{equation} 
\norm{h_\xi + h_\alpha}_{ \mathcal{C}_4^T } + \norm{z_\xi + z_\alpha}_{ \mathcal{W}^T } < \varepsilon_0 + C_1 \varepsilon_1 .
\end{equation}
Here the number $\varepsilon_0$ is the upper bound of $\norm{h_0}_{C^2(\Sigma)}$ as stated at the beginning of this section.
\end{proposition}

\begin{proof}

We only need to modify the estimate of $h_\alpha$ since $h_0\neq 0$ in the current case.  
Notice that $\norm{h_\alpha (0)}_{C^2} = \norm{h_0}_{C^2} <\varepsilon_0$ by the assumption of the initial conditions. 
Since $h_\alpha \in C([0,T_0];C^2(\Sigma))=: \mathcal{C}_4^{T_0}$,  we can let  $T(\varepsilon)$
be sufficiently small, such that for all $t\in[0,T]$ we have
\begin{equation} 
\norm{h_\alpha(t)}_{ C^2 }
\leq \norm{h_0}_{ C^2 } + \norm{h_\alpha(t)- h_0}_{ C^2 }
< \varepsilon_0 + \varepsilon_1,
\end{equation}
which implies
\begin{equation} 
\norm{h_\alpha}_{\mathcal{C}_4^T }
< \varepsilon_0 + \varepsilon_1.
\end{equation}
This completes the proof.

\end{proof}

We can now prove Theorem \ref{existence.result.2} in a similar spirit to \cite{Pruss.quali, Pruss.analytic.solu}.

\begin{proof}[Proof of Theorem \ref{existence.result.2}]

First, we verify the condition 
\[ \norm{h}_{\mathcal{C}_4^{T} } :=  \norm{h}_{C([0,T];C^2(\Sigma)) } <\delta_0 \]
as stated in  Proposition \ref{final.estimate.proposition},
which enables us  to use the estimate in \eqref{final estimate}.
Notice that $\delta_0$ only depends on $\Sigma$ and thus is a fixed number.
Without loss of generality, we assume  $\varepsilon_0 < \delta_0 /4$
and $C_1 \varepsilon_1  < \delta_0 /4 $. 
From Proposition \ref{h,z,DN,small,temp1,pert}, 
there exists $T_1\in (0,T_0]$,
such that
$ \norm{h}_{\mathcal{C}_4^{T_1} } < \delta_0 $.

Given any $\varepsilon_0 $ and $\varepsilon_1 $, from 
Proposition \ref{trL,DN,small,temp1},
Proposition \ref{b-u,DN,small,temp1},
Proposition \ref{DN,small,bounded,2}
and Proposition \ref{h,z,DN,small,temp1,pert},
we can find $T\in (0,T_1]$ and $r\in (0,r_0]$ sufficiently small,  such that
\begin{equation} \label{final estimate.v2}
\begin{aligned}
 \norm{ DG\depvar{z_\alpha + z_\xi }\varphi} _{\mathring{\mathcal{S}}^T} 
& \leq
 C_2(M_0,T_0,r_0) \prts{  \varepsilon_0 + C_1 \varepsilon_1   	}
     \norm{\varphi}_{ \mathring{\mathcal{W}} ^T }
  + C_3(T_0) \varepsilon_1 \norm{\varphi_h}_{ \mathring{\mathcal{W}}_5^T } 
\\
& \quad + C_4( T_0) \varepsilon_1
	  \norm{ \varphi}_{\mathring{\mathcal{W}} ^T } .
\end{aligned}
\end{equation}
Since the constants $C_2$, $C_3$ and $C_4$ are independent of the choice of $\varepsilon_0$ and $\varepsilon_1$, for sufficiently small $\varepsilon_0$ and $\varepsilon_1$,
we can obtain a contraction mapping.
The rest of the proof can be proceeded using the same arguments as in 
the proof of Theorem \ref{existence.result.1}.

\end{proof}

%
%
%
 

\appendix

\section{Products of Sobolev Functions}

We write the estimate of the product of two functions from different spaces. 
The dependency of constant terms on parameters are carefully studied. 
\begin{proposition}
Let $\Omega$ be a bounded open set. 
Let $T_0>0$, $s\in(0,1)$, $r\in(0,1)$ and $q\geq 1$ be fixed numbers.
For all $T\in(0,T_0]$, suppose that
$f\in C^1([0,T];C(\Omega))\cap C([0,T];C^1(\Omega))$ 
and
$g\in W^{s,q}([0,T];L^q(\Omega))\cap L^q([0,T];W^{r,q}(\Omega))$.
Then
\[
\norm{fg}_{W^{s,q}L^q \cap L^q W^{r,q} }
\leq C \norm{f}_{C^1 C\cap CC^1 } \norm{g}_{W^{s,q}L^q \cap L^q W^{r,q} }.
\]
\end{proposition}

\begin{proof}

Step 1:
\\
For all $t\in [0,T]$ we have
\[
\norm{fg}_{W^{r,q}(\Omega)} (t) := \norm{fg}_{L^q} (t) + \seminorm{fg}_{W^{r,q}} (t).
\]
We omit the variable $t$ when there is no confusion.
For the $L^q$ norm we have
\[
\norm{fg}_{L^q(\Omega)} 
\leq  \norm{f}_{C(\Omega)}   \norm{g}_{L^q(\Omega)} .
\]
For the seminorm we have
\begin{equation}
\begin{aligned}
& \seminorm{fg}_{W^{r,q}(\Omega)}  
:= \prts{ \int_\Omega \int_\Omega \frac{\abs{ f(x)g(x)-f(y)g(y)}^q }{\abs{x-y}^{n+rq}} dydx }^{\frac{1}{q}}
\\
& \leq \prts{ C(q)
\int_\Omega \int_\Omega \frac{ \abs{ f(x)}^q \abs{g(x)-g(y)}^q +  \abs{ f(x)-f(y)}^q \abs{g(y)}^q  }{\abs{x-y}^{n+rq}} dydx  }^{\frac{1}{q}}
\\
& \leq \prts{ C(q)
\int_\Omega \int_\Omega \frac{ \norm{ f}_{C^0}^q \abs{g(x)-g(y)}^q +  \norm{ f}_{C^1}^q \abs{x-y}^q \abs{g(y)}^q  }{\abs{x-y}^{n+rq}} dydx  }^{\frac{1}{q}}
\\
&  \leq C(q) \norm{f}_{C^0} \seminorm{g}_{W^{r,q}} 
+ C(q) \norm{f}_{C^1} \prts{\int_\Omega \abs{g(y)}^q \prts{\int_\Omega \abs{x-y}^{-n-rq+q} dx }dy  }^{\frac{1}{q}}
\\
& \leq C(q) \norm{f}_{C^0} \seminorm{g}_{W^{r,q}} +  C(q, {\rm diam}(\Omega) ) \norm{f}_{C^1} \norm{g}_{L^q}
\\
& \leq   C(q,{\rm diam}(\Omega)) \norm{f}_{C^1} \norm{g}_{W^{r,q}},
\end{aligned}
\end{equation}
where $C(q,{\rm diam}(\Omega))$ is an increasing function of ${\rm diam}(\Omega)$.
Thus, we have
\begin{equation}
\begin{aligned}
&\norm{fg}_{L^q W^{r,q}}
:= \prts{ \int_0^T \norm{fg}^q_{W^{r,q}(\Omega)}(t) dt }^{\frac{1}{q}}
\\
& \leq C(q,{\rm diam}(\Omega)) 
\prts{ 
	\int_0^T \norm{f}^q_{C^1(\Omega)}(t) \norm{g}^q_{W^{r,q}(\Omega)}(t) dt 
	}^{\frac{1}{q}}
\\
& \leq C(q,{\rm diam}(\Omega)) 
\norm{f}_{C^0C^1} \norm{g}_{L^q W^{r,q}} ,
\end{aligned}
\end{equation}
where the constant is still an increasing function of ${\rm diam}(\Omega)$ and it is independent of $T_0$ or $T$.

Step 2:
We recall that
\[
\norm{fg}_{W^{s,q}L^q}  := \norm{fg}_{L^qL^q}  + \seminorm{fg}_{W^{s,q}L^q} 
\]
with the seminorm defined as
\[
 \seminorm{f}_{W^{s,q}L^q} :=
  \prts{ \int_0^T \int_0^T \frac{\abs{ f(t)-f(\tau)}^q }{\abs{x-y}^{1+sq}} dtd\tau }^{\frac{1}{q}}.
\]
We have the estimate
\begin{equation}
\begin{aligned}
&  \seminorm{fg}_{W^{s,q}L^q} :=
  \prts{ \int_0^T \int_0^T \frac{\norm{ f(t)g(t)-f(\tau)g(\tau)}_{L^q(\Omega)}^q }{\abs{t-\tau}^{1+sq}} dtd\tau }^{\frac{1}{q}}
\\
&
\leq
C(q) \prts{ \int_0^T \int_0^T \frac{\norm{ f(t)\prts{g(t)-g(\tau)}}_{L^q(\Omega)}^q +  \norm{\prts{f(t)-f(\tau)} g(\tau)}_{L^q(\Omega)}^q }{\abs{t-\tau}^{1+sq}} dtd\tau }^{\frac{1}{q}}
\\
&
\leq
C(q) \prts{ \int_0^T \int_0^T \frac{\norm{ f(t)}_{C^0(\Omega)}^q\norm{g(t)-g(\tau)}_{L^q(\Omega)}^q +  \norm{ f(t) -f(\tau)}_{C^0(\Omega)}^q\norm{g(\tau)}_{L^q}^q }{\abs{t-\tau}^{1+sq}} dtd\tau }^{\frac{1}{q}}
\\
&
\leq
C(q) \prts{ \int_0^T \int_0^T \frac{\norm{ f(t)}_{C^0(\Omega)}^q\norm{g(t)-g(\tau)}_{L^q(\Omega)}^q +  \norm{ f}_{C^1C^0}^q \abs{t-\tau}^q\norm{g(\tau)}_{L^q}^q }{\abs{t-\tau}^{1+sq}} dtd\tau }^{\frac{1}{q}}
\\
& \leq
C(q) \norm{f}_{C^0 C^0} \seminorm{g}_{W^{s,q}L^q}
+ C(q) \norm{f}_{C^1 C^0} \prts{ \int_0^T \norm{g(\tau)}_{L^q}^q \prts{ \int_0^T  \abs{t-\tau}^{-1-sq+q} dt }   d\tau }^{\frac{1}{q}}
\\
& \leq
C(q) \norm{f}_{C^0 C^0} \seminorm{g}_{W^{s,q}L^q}
+ C(q,T) \norm{f}_{C^1 C^0} \norm{g}_{L^q L^q} 
\\
& \leq
C(q,T) \norm{f}_{C^1 C^0} \norm{g}_{W^{s,q}L^q} ,
\end{aligned}
\end{equation}
where $C(q,T)$ is an increasing function of $T$.

Consequently, since $\Omega$ and $T_0$ are fixed and $T\leq T_0$, 
we obtain for all $T\in(0,T_0]$ that
\begin{equation}
\begin{aligned}
& \norm{fg}_{W^{s,q}([0,T];L^q(\Omega)) \cap L^q([0,T]; W^{r,q}(\Omega)) }
:= \norm{fg}_{ L^q W^{r,q} } +  \norm{fg}_{W^{s,q}L^q }
\\
& \leq 
C(q,{\rm diam}(\Omega)) 
\norm{f}_{C^0C^1} \norm{g}_{L^q W^{r,q}} 
+ C(q,T_0) \norm{f}_{C^1 C^0} \norm{g}_{W^{s,q}L^q}
\\
& \leq  C(q,{\rm diam}(\Omega),T_0)
\norm{f}_{C^1 ([0,T];C(\Omega)) \cap C([0,T]; C^1 (\Omega)) } 
\norm{g}_{W^{s,q}([0,T];L^q(\Omega)) \cap L^q([0,T]; W^{r,q}(\Omega)) } .
\end{aligned}
\end{equation}
\end{proof}

\section{Geometric Terms in the Hanzawa Transformation}
\label{Appendix. Geometric Terms}

The representation of geometric quantities of the free interface $\Gamma(t)$ using the refernce surface $\Sigma$ and height function $h(t,\cdot)$ has been studied in  \cite{Pruss.quali, Pruss.green.book}.
We rewrite the main steps with more details included for completeness.  
In this part, we will temporarily consider the general $d$-dimensional space $\mathbb{R}^d$.
  
\subsection{Tangent and normal vectors}

Suppose that   $\Sigma$ is a $C^k$ surface, 
then it can be locally parameterized 
by a $C^k$ function $\Phi$, 
i.e., 
for all $x\in\Sigma$ there exists a neighborhood $B(x;a)\cap\Sigma$ such that 
there exists a domain $D\subseteq \mathbb{R}^{d-1}$ and a diffeomorphism $\Phi(s)$ 
from $D$ to $B(x;a)\cap\Sigma$, where $s=(s_1,\cdots, s_{d-1})\in D$.
Let $x=\Phi(s)\in\Sigma$, then the tangent vectors at $x$ are 
\begin{equation}
\tau^{\Sigma}_{i}(s) = \partial_{i}\Phi(s), \quad i=1,\cdots, d-1,
\end{equation} 
which form a basis of the tangent space $T_x \Sigma$ . 
We will also use notations $\tau^{\Sigma}_{(i),k}$ and $\tau_{\Sigma,k}^{(i)}$ 
to denote the $k$-th entry of the $i$-th vector. 
These $d-1$ vectors depend on the choice of $\Phi$. 
Thus, we directly view $\tau^{\Sigma}_{i}$ as a function defined in $D$.
The normal vector $n_{\Sigma}$  is independent of $\Phi$ and
thus can be viewed as a function defined on $\Sigma$.

To simplify calculations, 
we introduce another basis
$\{ \tau_{\Sigma}^{1}, \cdots, \tau_{\Sigma}^{d-1} \}$
of the tangent plane $T_x \Sigma$.
The new basis satisfies $\tau_{\Sigma}^{i}\cdot \tau^{\Sigma}_{j}=\delta_{ij}$,
where $\delta_{ij}=1$ if $i=j$ and  $\delta_{ij}=0$ if $i\neq j$.
We omit the name of surfaces in superscripts or subscripts when there is no ambiguity.
Suppose $\xi= \sum c_i\tau^i = \sum c^i\tau_i $, 
then we have
$c_i= \xi \cdot \tau_i$ and $c^i= \xi \cdot \tau^i$.
We refer readers to \cite{Pruss.green.book} for more details.

For every height function $h$, its corresponding surface is
$\Gamma_h(t):= \Theta_h(t,\Sigma) $,
which can be parameterized using $\Theta_h \circ \Phi $.
Its tangent vectors at $y:=\Theta(x)=\Theta(\Phi(s))$ are
\begin{equation} 
\begin{aligned}
\tau^{\Gamma}_i(s) = \partial_{s_i} \prts{ \Theta(\Phi(s)) }
= \sum_{j}  \prts{ \partial_i \Phi_j} \prts{ \partial_j \Theta \circ \Phi }.
\end{aligned}
\end{equation}
To find the normal vector $n_\Gamma$, 
we first seek for $\alpha \in T_x\Sigma$ 
such that  $n_\Sigma - \alpha$ is perpendicular   to
 $T_{\Theta(x)} \Gamma$. We refer readers to \cite[Section 2.2.2]{Pruss.green.book} for more details.
For  convenience, we include some key steps 
in the derivation of  $\alpha$.
When $x\in B(\Sigma;\varrho)$ with $\varrho = \varrho_0 /3$,
we have $\Theta_h(x)= x+ h(x)n_{\Sigma}(x) $, which implies
\begin{equation} \label{local,region,diffeo}
\Theta_h(\Phi(s))= \Phi(s)+ h(\Phi(s))n_{\Sigma}(\Phi(s)) . 
\end{equation}
Taking derivatives of the equation \eqref{local,region,diffeo}, we have
\begin{equation} \label{tangent,vector.Gamma}
\begin{aligned}
& \tau^{\Gamma}_{i}(s)
=
\partial_{s_i}\Theta(\Phi(s))
=
\partial_{s_i}\prts{ \Phi(s) + \theta(\Phi(s)) }
\\
& = 
\partial_{i} \Phi(s) 
+  \partial_{s_i} (h\circ\Phi)\depvar{s}  n_{\Sigma}\depvar{ \Phi\depvar{s} }
+  h\depvar{ \Phi\depvar{s} }   \partial_{s_i} (n_{\Sigma}\circ\Phi)\depvar{s} .
\end{aligned}
\end{equation}
In this work, we let the parameterization $\Phi: D\to \Sigma $ be  fixed once it has been chosen.
To make  calculations concise, for any function $f$ on $\Sigma$ 
we will abbreviate $f\circ \Phi$ to $f$ when there is no ambiguity.
We will also use the notation $\partial_i f$ to represent the derivative $\partial_{s_i} (f(\Phi(s))$ 
when there is no ambiguity.
This follows the convention in \cite{Pruss.green.book}.

Now we simplify \eqref{tangent,vector.Gamma}.
Since $\abs{n_{\Sigma}}\equiv 1$ , we have
\begin{equation} 
\begin{aligned}
2 \partial_{s_i}n_{\Sigma}(\Phi(s)) \cdot n_{\Sigma}(\Phi(s))
=
 \partial_{s_i}(\abs{n_{\Sigma}(\Phi(s)) }^2 )
= 0.
\end{aligned}
\end{equation}
In fact, the vectors $\partial_{s_1}n,\cdots , \partial_{s_{d-1}}n$ form a new basis of $T_{\Phi(s)}\Sigma$.
Using the Weingarten tensor $L_{\Sigma}$, which satisfies
\begin{equation}  \label{shape.operator.1}
L_{\Sigma}\depvar{\Phi(s)} \tau^{\Sigma}_i\depvar{s} := - \partial_{s_i} (n_{\Sigma}\circ \Phi) \depvar{s},
\end{equation}
we can simplify \eqref{tangent,vector.Gamma} to
\begin{equation} \label{tangent,vector.Gamma.correct}
\begin{aligned}
\tau^{\Gamma}_{i} 
=
 \prts{I- h L_{\Sigma}}  \tau^{\Sigma}_{i}  +  \partial_{i} h \, n_{\Sigma} .
\end{aligned}
\end{equation}
Next, we  simplify the formula of $\alpha$ using \eqref{tangent,vector.Gamma.correct}.
Since  $n_{\Sigma}-\alpha $ is required to be perpendicular to $T_{\Theta(\Phi(s))}\Gamma$,
we have  $(n_{\Sigma}-\alpha ) \bot \tau^{\Gamma}_{i} $ for all $1\leq i\leq d-1$, 
which implies
\begin{equation}
\begin{aligned}
&
0
=
\prts{ n_\Sigma-\alpha } \cdot \tau^{\Gamma}_{i}
=
\prts{ n_\Sigma-\alpha } \cdot 
\prts{ \tau^{\Sigma}_{i}  +  \partial_{i}h \, n_\Sigma + h \partial_{i} \,  n_\Sigma }
\\
&
=
 n_\Sigma\cdot \tau^{\Sigma}_{i}
+ n_\Sigma\cdot  (\partial_{i}h  \,  n_\Sigma )
+ n_\Sigma\cdot ( h \partial_{i}  \,  n_\Sigma )
- \alpha \cdot \tau^{\Sigma}_{i}  
- \alpha \cdot ( \partial_{i}h  \,  n_\Sigma )
- \alpha \cdot ( h \partial_{i} \,  n_\Sigma )
\\
&
=
0+ \partial_{i}h + 0 - \alpha \cdot \tau^{\Sigma}_{i} 
 -0 - \alpha \cdot ( h  \,  \partial_{i} n_\Sigma )
=
\partial_{i}h - \alpha\cdot  \prts{  \prts{ I-hL_{\Sigma}}  \tau^{\Sigma}_{i}   }.
\end{aligned}
\end{equation} 
Since $I-hL_{\Sigma}$ is a symmetric linear transformation (see e.g. Section 2.2 in \cite{Pruss.green.book}), we have
\begin{equation} \label{partial,i,h;formula}
\partial_{i}h = \prts{  \prts{ I-hL_{\Sigma}}  \alpha  }  \cdot    \tau^{\Sigma}_{i}.
\end{equation} 
Notice that \eqref{partial,i,h;formula} is the abbreviation of:
\begin{equation} \label{partial,i,h;formula;detailed}
\partial_{s_i}(h(\Phi(s))) = \prts{  \prts{ I-hL_{\Sigma}}  \alpha  }  \cdot    \partial_i\Phi(s),
\end{equation} 
which can be solved by letting 
\begin{equation} \label{sufficient,condition;solve,alpha}
\prts{ I-hL_{\Sigma}}  \alpha = \nabla_{\Sigma}h, \quad \text{i.e.}  \quad 
  \alpha = \prts{ I-hL_{\Sigma}}^{-1} \nabla_{\Sigma}h.
\end{equation}  
The surface gradient $\nabla_{\Sigma}$, also called the tangential gradient,
is explained in Section \ref{preliminary}.
We refer interested readers to \cite{Ambrosio.BV, Maggi.BV} for details about gradient, divergence 
and other differential operators on surfaces. 
Now we have obtained that
\begin{equation} \label{formula;n,Gamma}
n_\Gamma(s)
=  \beta \prts{ n_\Sigma - \prts{ I-hL_{\Sigma}}^{-1} \nabla_{\Sigma}h } ,
\end{equation}  
where 
\begin{equation}
  \beta : =
\frac{  1 }{  \abs{ n_\Sigma - \prts{ I-hL_{\Sigma}}^{-1} \nabla_{\Sigma}h} }
\end{equation}
 as used in \cite{Pruss.green.book}.

\subsection{Mean curvature}

In this section, we express the mean curvature   $H_{\Gamma}$ in terms of $\Sigma$ and $h$ in  $\mathbb{R}^d$.
We follow the calculations in \cite[Section 2.2.5]{Pruss.green.book} with more details included for completeness.

We start with 
the representation  of  $\nabla_\Gamma f$ for an arbitrary function $f$ defined on $\Gamma$. 
From equation (2.47) in  \cite[Section 2.2.3]{Pruss.green.book}, we have the representation of tangent vectors on $\Gamma$:
\begin{equation}
\tau_{\Gamma}^{i} = \mathcal{P}_{\Gamma} \prts{ I- hL_{\Sigma} }^{-1} \tau_{\Sigma}^{i} \ ,
\quad 1\leq i \leq d-1 .
\end{equation}
Notice that for any function $f$ on $\Gamma$ (parameterized by  $\Theta_h\circ\Phi$) we have
\begin{equation}
\prts{ \sum_{j=1}^{d-1} \partial_j f  \tau_{\Gamma}^{j} } \tau^{\Gamma}_{i}  
= \partial_i f,
\end{equation}
which implies 
\begin{equation}
\nabla_{\Gamma}f = \sum_{j=1}^{d-1} \partial_j f  \tau_{\Gamma}^{j}
=   \sum_{j=1}^{d-1}   \mathcal{P}_{\Gamma} \prts{ I- hL_{\Sigma} }^{-1} \tau_{\Sigma}^{j} \partial_j f
= \mathcal{P}_{\Gamma} \prts{ I- hL_{\Sigma} }^{-1} \nabla_{\Sigma}( f\circ \Theta). 
\end{equation}
From the definition of mean curvature in Section 2.2.5 of \cite{Pruss.green.book},  we have
\begin{equation} \label{mean.curvature.formula1}
H_{\Gamma} : = - {\rm div}_{\Gamma} n_{\Gamma}
= - \sum_{i=1}^{d-1} \tau_{\Gamma}^i \cdot \partial_i n_{\Gamma}.
\end{equation}
We also refer to  \cite[Section 7.3]{Ambrosio.BV} for more details.
Now it remains to calculate $\partial_i n_{\Gamma}$ in \eqref{mean.curvature.formula1}. 
From \eqref{formula;n,Gamma} we have
\begin{equation} 
\partial_i n_{\Gamma}(s)
=
 \partial_i \beta \, \prts{ n_{\Sigma} - \prts{ I-hL_{\Sigma}}^{-1} \nabla_{\Sigma}h }
+
  \beta  \, \partial_i \prts{ n_{\Sigma} - \prts{ I-hL_{\Sigma}}^{-1} \nabla_{\Sigma}h }.
\end{equation}  
Using the same argument as in \cite[Section 2.2.5]{Pruss.green.book}, 
we  obtain the formula of  the mean curvature 
\begin{equation} \label{curvature,Gamma,formula}
\begin{aligned}
H_{\Gamma}
&= - \sum_{i=1}^{d-1} \tau_{\Gamma}^i \cdot \partial_i n_{\Gamma}
\\
&
=
 - \sum_{i=1}^{d-1}
\tau_{\Gamma}^i \cdot 
\prts{
	\partial_i \beta  \prts{ \frac{ n_\Gamma}{ \beta}  }
}
 - \sum_{i=1}^{d-1}
\tau_{\Gamma}^i \cdot 
\prts{
  \beta  \partial_i \prts{ n_\Sigma - \prts{ I-hL_{\Sigma}}^{-1} \nabla_{\Sigma}h }
}
\\
&
= 
0 + \sum_{i=1}^{d-1}
 \prts{
\mathcal{P}_{\Gamma} \prts{ I- hL_{\Sigma} }^{-1} \tau_{\Sigma}^{i}
}\cdot
\prts{
\beta   \prts{ L_{\Sigma}\tau^\Sigma_i + \partial_i \alpha }
}
\\
&
=
\sum_{i=1}^{d-1}
 \prts{
\prts{I- n_\Gamma\otimes n_\Gamma} \prts{ I- hL_{\Sigma} }^{-1} \tau_{\Sigma}^{i}
}\cdot
\prts{
\beta   \prts{ L_{\Sigma}\tau_i + \partial_i \alpha }
} 
\\
&
=
\sum_{i=1}^{d-1}
 \prts{
 \prts{ I- hL_{\Sigma} }^{-1} \tau_{\Sigma}^{i}
}\cdot
\prts{
\beta   \prts{ L_{\Sigma}\tau^\Sigma_i + \partial_i \alpha }
}
\\
&\quad -
\sum_{i=1}^{d-1}
\prts{
  n_\Gamma\cdot  
  \prts{
     \prts{ I- hL_{\Sigma} }^{-1} \tau_{\Sigma}^{i}
  }
}
\prts{
\beta  n_\Gamma\cdot \prts{ L_{\Sigma}\tau^\Sigma_i + \partial_i \alpha }
} .
\end{aligned}
\end{equation}
Still from \cite[Section 2.2.5]{Pruss.green.book}, 
we have the following equalities:
\begin{equation} \label{curvature,calculate,temp,1}
\sum_{i=1}^{d-1} \prts{ \prts{I-hL_{\Sigma}}^{-1} \tau_{\Sigma}^i } \cdot \prts{ L_{\Sigma} \tau^{\Sigma}_i } = {\rm tr} \prts{\prts{ \prts{I-hL_{\Sigma}}^{-1} L_{\Sigma}}  }
,
\end{equation}
\begin{equation} \label{curvature,calculate,temp,2}
\sum_{i=1}^{d-1} \prts{ \prts{I-hL_{\Sigma}}^{-1}  \tau_{\Sigma}^i }\cdot \partial_i \alpha 
=
 {\rm tr}
\prts{ \prts{I-hL_{\Sigma}}^{-1} \nabla_{\Sigma}\alpha },
\end{equation}
\begin{equation}\label{curvature,calculate,temp,3}
n_\Gamma \cdot \prts{ \prts{I-hL_{\Sigma}}^{-1} \tau_{\Sigma}^j} 
=
-\beta \prts{ \prts{ \prts{I-hL_{\Sigma}}^{-1} \alpha   }  } ^j ,
\end{equation}
where $1\leq j\leq d-1$.
\begin{remark}
Since $\tau^\Sigma_{i}\cdot \tau_\Sigma^{j}=\delta_{ij}$, we combine them with the normal vector $n_\Sigma$ and obtain (the notation $\Sigma$ is omitted for convenience):
\begin{equation}
\begin{aligned}
\begin{pmatrix}
 \tau_{(1),1}  & \cdots  & \tau_{(1),d} \\
 \vdots  & \ddots  & \vdots \\
 \tau_{(d-1),1}  & \cdots  & \tau_{(d-1),d} \\
 n_1  & \cdots  & n_d 
\end{pmatrix}
\begin{pmatrix}
 \tau^{(1)}_1   & \cdots  & \tau^{(d-1)}_1  & n_1 \\
 \vdots  & \ddots  & \vdots  & \vdots \\
\\
 \tau^{(1)}_d  & \cdots  & \tau^{(d-1)}_d  & n_d
\end{pmatrix}
=I_{d\times d}  .
\end{aligned}
\end{equation}
The left-hand side is commutative by the property of inverse matrices, which implies
\begin{equation} \label{another,tau,tau,delta}
\sum_{k=1}^{d-1}  \tau^{(k)}_i \tau_{(k),j}=\delta_{ij} -n_i n_j = (I- n\otimes n)_{ij} 
\end{equation}
for all $1 \leq i,j\leq d-1$.
We recall that for any vector $\eta\in\mathbb{R}^d$ the matrix $\nabla_\Sigma \eta$ is defined as
$ \prts{ \nabla_\Sigma \eta_1,  \cdots \nabla_\Sigma \eta_d  } $,
where each $ \nabla_\Sigma \eta_i$ is viewed as a column vector.
This implies that 
\begin{equation}
\begin{aligned}
\sum_{i=1}^{d-1} \tau^{(i)} \cdot \prts{ \nabla_\Sigma \eta  \tau_{(i)} }
&=
\sum_{i=1}^{d-1} \sum_{k=1}^{d}\sum_{j=1}^{d}  \tau^{(i)}_k    \prts{\nabla_\Sigma \eta}_{kj}  \tau_{(i),j}
\\
&=  \sum_{k=1}^{d}\sum_{j=1}^{d} \prts{ \sum_{i=1}^{d-1}   \tau^{(i)}_k  \tau_{(i),j} }  \prts{\nabla_\Sigma \eta}_{kj} 
= (I-n\otimes n): \nabla_\Sigma \eta 
\\
&= {\rm tr} \nabla_\Sigma \eta - n^\top \, \nabla_\Sigma \eta  \, n
= {\rm tr} \nabla_\Sigma \eta ,
\end{aligned}
\end{equation}
which can be utilized in the derivation of \eqref{curvature,Gamma,formula}, \eqref{curvature,calculate,temp,1} and \eqref{curvature,calculate,temp,2}. 
\end{remark}
Notice that the symmetric operator  $(I- h L_\Sigma)^{-1}$ in \eqref{curvature,calculate,temp,3}  maps tangent vectors to tangent vectors and maps $n_\Sigma$ to $n_\Sigma$. 
Thus, for all $1 \leq i \leq d-1$ we have
\begin{equation} \label{curvature,Gamma,temp4}
\begin{aligned}
n_\Gamma \cdot \prts{ \prts{I-hL_{\Sigma}}^{-1} \tau_{\Sigma}^i } 
&=
\beta\prts{n_\Sigma -\alpha } \cdot \prts{ \prts{I-hL_{\Sigma}}^{-1} \tau_{\Sigma}^i } 
=
- \beta \alpha \cdot \prts{ \prts{I-hL_{\Sigma}}^{-1} \tau_{\Sigma}^i } 
\\
&
=
-\beta \prts{  \prts{I-hL_{\Sigma}}^{-1} \alpha  } \cdot \tau_{\Sigma}^i  
=:
-\beta \prts{  \prts{I-hL_{\Sigma}}^{-1} \alpha  }^i  ,
\end{aligned}
\end{equation}
and
\begin{equation} \label{curvature,Gamma,temp5}
\begin{aligned}
n_\Gamma\cdot \prts{ L_{\Sigma}\tau_i + \partial_i \alpha }
&=
\beta \prts{ n_\Sigma - \alpha } \cdot \prts{ L_{\Sigma}\tau_i + \partial_i \alpha }
\\
&
=
\beta 
\prts{
 n_\Sigma\cdot L_{\Sigma}\tau_i + n_\Sigma\cdot \partial_i \alpha
 - \alpha \cdot L_{\Sigma}\tau_i - \alpha \cdot \partial_i \alpha
}
\\
&
=
\beta 
\prts{
 0 + n_\Sigma\cdot \partial_i \alpha
 + \alpha \cdot \partial_i n_\Sigma - \alpha \cdot \partial_i \alpha
}\\
&=
\beta 
\prts{
  \partial_i \prts{n_\Sigma\cdot  \alpha }
   - \alpha \cdot \partial_i \alpha
}\\
&=  -\beta \alpha \cdot \partial_i \alpha .
\end{aligned}
\end{equation}
Notice that $n_\Sigma\cdot  \alpha=0$ since $\alpha$ is a tangent vector of $\Sigma$.
From \eqref{curvature,Gamma,formula},  \eqref{curvature,Gamma,temp4} and \eqref{curvature,Gamma,temp5},
 we have
\begin{equation} \label{curvature,formula}
\begin{aligned}
H_{\Gamma}
&=
\beta
{\rm tr} \prts{ \prts{I-hL_{\Sigma}}^{-1} \prts{ L_{\Sigma} + \nabla_{\Sigma}\alpha } } 
-
\sum_{i=1}^{d-1}
\beta^3 \prts{  \prts{I-hL_{\Sigma}}^{-1} \alpha    } ^i
\prts{
\alpha \cdot \partial_i \alpha
} 
\\
&
=
\beta
{\rm tr} \prts{ \prts{I-hL_{\Sigma}}^{-1} \prts{ L_{\Sigma} + \nabla_{\Sigma}\alpha } } 
-
\beta^3 \prts{  \prts{I-hL_{\Sigma}}^{-1} \alpha    } 
 \,
 \nabla_{\Sigma} \alpha \, \alpha ^\top.
\end{aligned}
\end{equation}
\begin{remark} 
The term $\nabla_{\Sigma}\alpha \, \alpha ^\top$ in \eqref{curvature,formula} is obtained by the following argument.
For convenience, we let $\xi:= \prts{I-hL_{\Sigma}}^{-1} \alpha$ and  omit the notation $\Sigma$ in tangent vectors.
\begin{equation} \label{remark.grad,alpha,alpha}
\begin{aligned}
 \sum_{i=1}^{d-1}
\prts{  \prts{I-hL_{\Sigma}}^{-1} \alpha    } ^i 
\prts{ \alpha \cdot \partial_i \alpha  } 
&=
\sum_{i=1}^{d-1}
 \xi^i 
\prts{ \alpha \cdot \partial_i \alpha  } 
=
\sum_{i=1}^{d-1}
\prts{ \xi\cdot \tau^{(i)} }
\prts{ \alpha \cdot \partial_i \alpha  } 
\\
& = \sum_{i=1}^{d-1}
 \prts{ \sum_{k=1}^{d}  \xi_k \tau^{(i)}_k } 
\prts{ \alpha \cdot \partial_i \alpha  } .
\end{aligned}
\end{equation}
We remind the readers that $\partial_i\alpha$ is the abbreviation of 
$\partial_i(\alpha\circ\Phi)$ by our convention.
We view the surface gradient of vector $\alpha$ as the matrix such that
\begin{equation}
\begin{aligned}
\prts{ \nabla_{\Sigma} \alpha }_{ij}
:= \prts{ \nabla_{\Sigma} \alpha_j }_{i}.
\end{aligned}
\end{equation}
Notice that we assume all vectors to be written as row vectors, 
then for all $1 \leq k \leq d-1$ we have
\begin{equation}
\begin{aligned}
 \partial_k \alpha  
&=
\prts{ \partial_k \alpha_1, \cdots, \partial_k \alpha_d     }
\\
& =
\prts{ \nabla_{\Sigma} \alpha_1 \cdot \tau_k, \cdots, \nabla_{\Sigma} \alpha_d \cdot \tau_k     }
=
 \tau_k\nabla_{\Sigma} \alpha  .
\end{aligned}
\end{equation}
Thus, we rewrite the last term in \eqref{remark.grad,alpha,alpha} as
\begin{equation}
\begin{aligned}
 \sum_{i=1}^{d-1}
 \prts{ \sum_{k=1}^{d}  \xi_k \tau^{(i)}_k } 
\prts{ \alpha \cdot \partial_i \alpha  } 
&=
 \sum_{i=1}^{d-1}
 \prts{ \sum_{k=1}^{d}  \xi_k \tau^{(i)}_k} 
\prts{ \sum_{j=1}^{d} \alpha_j \prts{ \tau_{(i)} \cdot \prts{\nabla_{\Sigma}\alpha_j} }   }
\\ 
& =   \sum_{i=1}^{d-1}
 \prts{ \sum_{k=1}^{d}  \xi_k \tau^{(i)}_k } 
\prts{ \sum_{j=1}^{d} \alpha_j \prts{ \sum_{s=1}^{d} \tau_{(i),s}  \prts{\nabla_{\Sigma}\alpha_j}_s  }   }
\\
& =    \sum_{k=1}^{d} \sum_{j=1}^{d}  \sum_{s=1}^{d}  \sum_{i=1}^{d-1}
   \tau^{(i)}_k \tau_{(i),s}  \xi_k  \alpha_j  \prts{\nabla_{\Sigma}\alpha_j}_s   
=      \sum_{k=1}^{d} \sum_{j=1}^{d}    \xi_k  \alpha_j  \prts{\nabla_{\Sigma}\alpha_j}_k   
\\
& =   \xi  \nabla_{\Sigma}\alpha   \alpha^{\top},
\end{aligned}
\end{equation}
where the fourth equality is guaranteed  by \eqref{another,tau,tau,delta}.
\end{remark}


\section*{Acknowledgments}
T. Jing would like to thank  Professor Armin Schikorra for
his valuable discussions and suggestions, and  Professor  Matthias Köhne and Professor  Mathias Wilke for their explanation of the  details in \cite{Pruss.quali}.
T. Jing was partially funded by NSF grant DMS-2044898.
D. Wang is supported in part by the National Science Foundation under grants DMS-1907519 and DMS-2219384.


\bibliography{ref}{}
\bibliographystyle{abbrv}

\end{document}